\documentclass[10pt]{amsart}
\usepackage[margin=1in]{geometry}
\usepackage{amsmath, amsthm, amssymb, amsfonts, color, hyperref}
\usepackage{diagrams}
\hypersetup{
    colorlinks,
    citecolor=blue,
    filecolor=black,
    linkcolor=black,
    urlcolor=black
}
\usepackage[final]{showkeys}
\usepackage[disable]{todonotes}

\renewcommand{\sl}{\mathfrak{sl}}
\newcommand{\asl}{\widehat{\sl}}
\newcommand{\Tr}{\text{Tr}}
\newcommand{\hh}{\mathfrak{h}}
\newcommand{\bb}{\mathfrak{b}}

\newcommand{\FF}{\mathcal{F}}
\newcommand{\ba}{\bar{\alpha}}

\newcommand{\wt}{\text{wt}}
\newcommand{\QQ}{\mathbb{Q}}

\newcommand{\wtilde}{\widetilde}
\renewcommand{\AA}{\mathcal{A}}
\newcommand{\CC}{\mathbb{C}}
\newcommand{\RR}{\mathbb{R}}
\newcommand{\ZZ}{\mathbb{Z}}
\newcommand{\eps}{\varepsilon}
\newcommand{\Imm}{\text{Im}}
\newcommand{\hgf}{{}_2\phi_1}
\newcommand{\Res}{\text{Res}}

\newcommand{\wmu}{\wtilde{\mu}}

\newcommand{\cP}{\mathcal{P}}

\newcommand{\wsl}{\wtilde{\sl}}
\newcommand{\whh}{\wtilde{\hh}}
\newcommand{\wrho}{\wtilde{\rho}}
\newcommand{\abb}{\widehat{\bb}}
\newcommand{\hotimes}{\widehat{\otimes}}

\newcommand{\trig}{\text{trig}}
\newcommand{\cC}{\mathcal{C}}

\newcommand{\wv}{\wtilde{v}}
\newcommand{\mult}{\text{mult}}

\newcommand{\wu}{\wtilde{u}}

\newcommand{\wT}{\wtilde{T}}

\newcommand{\wF}{\wtilde{F}}
\newcommand{\hPhi}{\widehat{\Phi}}
\newcommand{\ch}{\mathrm{h}} 
\newcommand{\fg}{\mathfrak{g}}
\newcommand{\wPsi}{\wtilde{\Psi}}
\newcommand{\wPhi}{\wtilde{\Phi}}
\newcommand{\into}{\hookrightarrow}
\newcommand{\cl}{\mathrm{cl}}
\newcommand{\Sym}{\text{Sym}}
\newcommand{\wk}{\wtilde{k}}
\newcommand{\wDelta}{\wtilde{\Delta}}
\newcommand{\wJ}{\wtilde{J}}
\newcommand{\gl}{\mathfrak{gl}}
\newcommand{\agl}{\widehat{\gl}}
\newcommand{\mk}{\mathsf{k}}
\newcommand{\sT}{\mathsf{T}}

\newcommand{\cS}{\mathcal{S}}
\newcommand{\wS}{\wtilde{\cS}}
\newcommand{\Hom}{\text{Hom}}
\newcommand{\Ind}{\text{Ind}}
\newcommand{\ann}{\widehat{\mathfrak{n}}}

\theoremstyle{definition}
\newtheorem{thm}{Theorem}[section]
\newtheorem{prop}[thm]{Proposition}
\newtheorem{lemma}[thm]{Lemma}
\newtheorem{corr}[thm]{Corollary}
\newtheorem*{remark}{Remark}

\numberwithin{equation}{section}

\begin{document}
\title{Traces of intertwiners for quantum affine $\sl_2$ and Felder-Varchenko functions}
\author{Yi Sun}
\date{\today}
\email{yisun@math.mit.edu}

\begin{abstract}
We show that the traces of $U_q(\asl_2)$-intertwiners of \cite{ESV} valued in the three-dimensional evaluation representation converge in a certain region of parameters and give a representation-theoretic construction of Felder-Varchenko's hypergeometric solutions to the $q$-KZB heat equation given in \cite{FV}.  This gives the first proof that such a trace function converges and resolves the first case of the Etingof-Varchenko conjecture of \cite{EV}.

As applications, we prove a symmetry property for traces of intertwiners and prove Felder-Varchenko's conjecture in \cite{FV4} that their elliptic Macdonald polynomials are related to the affine Macdonald polynomials defined as traces over irreducible integrable $U_q(\asl_2)$-modules in \cite{EK3}.  In the trigonometric and classical limits, we recover results of \cite{EK2, EV}.   Our method relies on an interplay between the method of coherent states applied to the free field realization of the $q$-Wakimoto module of \cite{Mat}, convergence properties given by the theta hypergeometric integrals of \cite{FV}, and rationality properties originating from the representation-theoretic definition of the trace function.
\end{abstract}

\maketitle
\setcounter{tocdepth}{1}
\tableofcontents

\section{Introduction}

The present work connects two approaches for studying a family of special functions occurring in the study of the $q$-KZB heat equation, one originating in the representation theory of quantum affine algebras and one in the theory of theta hypergeometric integrals.  In \cite{ESV}, Etingof-Schiffmann-Varchenko showed that certain generalized traces for $U_q(\widehat{\fg})$-representations solve four commuting systems of difference equations: the $q$-KZB, dual $q$-KZB, Macdonald-Ruijsenaars, and dual Macdonald-Ruijsenaars systems.  In \cite{FTV, FTV2}, Felder-Tarasov-Varchenko constructed solutions to the $q$-KZB and dual $q$-KZB systems in terms of certain theta hypergeometric integrals which we term Felder-Varchenko functions.  The general philosophy of KZ-type equations predicts that these two families of solutions should be related by a simple renormalization, and this was conjectured by Etingof-Varchenko in \cite{EV, ESV, EV3}.  In the trigonometric and classical limits, the Etingof-Varchenko conjecture was verified for the $\sl_2$-case in \cite{EK2, EV, ES2, EV3, SV}.

In this paper, we show that the traces of $U_q(\asl_2)$-intertwiners of \cite{ESV} converge in a certain region of parameters and give a representation-theoretic construction of the Felder-Varchenko functions for the three-dimensional evaluation representation of $U_q(\asl_2)$.  This gives the first proof that such a trace function converges and resolves the first case of the Etingof-Varchenko conjecture.  As applications, we prove a symmetry property for traces of intertwiners and prove Felder-Varchenko's conjecture in \cite{FV4} that their elliptic Macdonald polynomials are related to the affine Macdonald polynomials defined as traces over irreducible integrable $U_q(\asl_2)$-modules in \cite{EK3}.  We take the trigonometric and classical limits explicitly and recover results of \cite{EK2, EV}.

Our method relies on an interplay between the free field realization of the $q$-Wakimoto module of \cite{Mat}, convergence properties given by the theta hypergeometric integrals of \cite{FV}, and rationality properties originating from the representation-theoretic definition of the trace function.  We apply the method of coherent states to the $q$-vertex operator expression for the $U_q(\asl_2)$-intertwiner to express its trace as a formal Jackson integral of iterated contour integrals.  Evaluation and a non-trivial manipulation of the resulting integrals allows us to identify the Jackson integral with a renormalization of the Felder-Varchenko function and therefore prove that the trace function converges.

Our work is motivated by Felder-Varchenko's conjecture in \cite{FV2, FV} that their functions satisfy the $q$-KZB heat equation, an integral equation which endows them with a $SL(3, \ZZ)$ modular symmetry.  Felder-Varchenko proved their conjecture by hypergeometric integral computations only in two special cases.  In future work, we hope to apply and extend the representation-theoretic understanding of the Felder-Varchenko functions provided by this work to show that all Felder-Varchenko functions satisfy the $q$-KZB heat equation and prove the Felder-Varchenko conjecture.

In the remainder of this introduction, we state our results more precisely and give some additional motivation and background.  For convenience, all notations will be reintroduced in full detail in later sections.

\subsection{Trace functions for $U_q(\asl_2)$}

We define now the trace function for $U_q(\asl_2)$ valued in the irreducible three-dimensional representation which appears in our main results. A more general definition appears in Section \ref{sec:tr-fn}.  For a weight $\mu$ and level $k$, let $M_{\mu, k}$ be the Verma module for $U_q(\asl_2)$ with highest weight $\mu \rho + k\Lambda_0$ and highest weight vector $v_{\mu, k}$.  Let $\mu$ and $k$ be generic and $L_2(z)$ be the evaluation representation of $U_q(\asl_2)$ corresponding to the $3$-dimensional irreducible representation of $U_q(\sl_2)$.  For $w_0 \in L_2[0]$, by \cite[Theorem 9.3.1]{EFK} there is a unique $U_q(\asl_2)$-intertwiner 
\[
\Phi_{\mu, k}^{w_0}(z): M_{\mu, k} \to M_{\mu, k} \hotimes L_2(z)
\]
which satisfies
\[
\Phi_{\mu, k}^{w_0}(z) v_{\mu, k} = v_{\mu, k} \otimes w_0 + (\text{l.o.t.}),
\]
where $(\text{l.o.t.})$ denotes terms of lower weight in the first tensor factor. Define the trace function by
\[
T^{w_0}(q, \lambda, \omega, \mu, k) := \Tr|_{M_{(\mu - 1)\rho + (k-2)\Lambda_0}}\Big(\Phi_{\mu - 1, k - 2}^{w_0}(z) q^{2\lambda\rho + 2\omega d}\Big),
\]
where we treat the trace as a $\CC$-valued function via the identification $L_2[0] \simeq \CC \cdot w_0$ and we remark that the expression is independent of $z$. 

In \cite{ESV}, it was shown that $T^{w_0}(q, \lambda, \omega, \mu, k)$ satisfies two commuting systems of infinite difference equations in $(\lambda, \omega)$ and $(\mu, k)$, the Macdonald-Ruijsenaars and dual Macdonald-Ruijsenaars equations.  It was further shown that a generalization of $T^{w_0}(q, \lambda, \omega, \mu, k)$ valued in the tensor product of multiple representations satisfies the $q$-KZB and dual $q$-KZB equations, which are difference equations in $\lambda$ and $\mu$.

\subsection{Felder-Varchenko functions}

The other key object in the present work is the Felder-Varchenko function, which we again define for the three-dimensional representation here, leaving a more general definition to Section \ref{sec:fv-fn}.  In the region of parameters
\[
|q|, |q^{-2k}|, |q^{-2\omega}| < 1, 
\]
define the Felder-Varchenko function as the theta hypergeometric integral given by
\[
u(q, \lambda, \omega, \mu, k) := q^{-\lambda\mu - \lambda - \mu - 2} \oint_{\cC_t} \frac{dt}{2\pi i t} \Omega_{q^2}(t; q^{-2\omega}, q^{-2k}) \frac{\theta_0(tq^{2\mu}; q^{-2k})}{\theta_0(tq^{-2}; q^{-2k})} \frac{\theta_0(t q^{2\lambda}; q^{-2\omega})}{\theta_0(t q^{-2}; q^{-2\omega})},
\]
where the phase function is 
\[
\Omega_{a}(z; r, p) := \frac{(z a^{-1}; r, p)(z^{-1} a^{-1} rp; r, p)}{(z a; r, p)(z^{-1} a rp; r, p)},
\]
the theta function is $\theta_0(u; q) := (u; q) (u^{-1} q; q)$, the single and double $q$-Pochhammer symbols are $(u; q) := \prod_{n \geq 0} (1 - uq^n)$ and $(u; q, r) := \prod_{n, m \geq 0} (1 - u q^n r^m)$, and the contour $\cC_t$ is the unit circle.  Away from this region, $u$ is defined by analytic continuation and extends to a function on the region of parameters satisfying $|q^{-2\omega}| \neq 1$ and $|q^{-2k}| \neq 1$.

In \cite{FV}, the function $u$ was shown to be a projective solution to the $q$-KZB heat equation, an integral equation which is the difference analogue of the KZB heat equation and endows the functions $u$ with a $SL(3, \ZZ)$-modular symmetry.  In \cite{FTV, FTV2}, generalizations of $u$ valued in the tensor product of multiple representations were shown to satisfy the $q$-KZB and dual $q$-KZB equations.

\subsection{Statement of the main results}

Our main result links the trace function $T^{w_0}(q, \lambda, \omega, \mu, k)$ to the Felder-Varchenko function in a certain region of the parameters $q$, $q^{-2\mu}$, $q^{-2\lambda}$, $q^{-2\omega}$, and $q^{-2k}$.  We term this region the \textit{good region of parameters} and define it as the region where the parameters satisfy
{\renewcommand{\theequation}{\ref{eq:good-region}}
\begin{equation}
0 < |q^{-2\omega}| \ll |q^{-2\mu}| \ll |q^{-2\lambda}| \ll |q^{-2k}| \ll |q|, |q|^{-1}.
\end{equation}
\addtocounter{equation}{-1}}%
We show that the trace function converges and admits an integral formula on this region of parameters.

{\renewcommand{\thethm}{\ref{thm:int-trace}}
\begin{thm}
For $q^{-2\mu}$ and then $q^{-2\omega}$ sufficiently close to $0$ in the good region of parameters (\ref{eq:good-region}), the trace function converges and has value
\begin{align*}
T^{w_0}(q, \lambda, \omega, \mu, k) &= \frac{q^{\lambda\mu - \lambda + 2}(q^{-4}; q^{-2\omega})}{\theta_0(q^{2\lambda}; q^{-2\omega})(q^{2\lambda - 2} q^{-2\omega}; q^{-2\omega})(q^{-2\lambda - 2}; q^{-2\omega})} \frac{ (q^{-2k};q^{-2k})(q^4 q^{-2k}; q^{-2k})}{(q^{-2\mu + 2};q^{-2k})(q^{2\mu + 2} q^{-2k}; q^{-2k})}\\
&\phantom{==}\frac{(q^{-2 \omega + 2}; q^{-2\omega}, q^{-2 k})^2}{(q^{-2 \omega - 2}; q^{-2 \omega}, q^{-2k})^2} \oint_{\cC_t} \frac{dt}{2\pi it} \Omega_{q^2}(t; q^{-2\omega}, q^{-2k}) \frac{\theta_0(tq^{-2\mu}; q^{-2k})}{\theta_0(tq^{-2}; q^{-2k})} \frac{\theta_0(tq^{2\lambda}; q^{-2\omega})}{\theta_0(tq^{-2}; q^{-2\omega})},
\end{align*}
where the integration cycle $\cC_t$ is the unit circle.
\end{thm}
\addtocounter{thm}{-1}}

{\renewcommand{\thethm}{\ref{corr:trace-fv}}
\begin{corr} 
For $q^{-2\mu}$ and then $q^{-2\omega}$ sufficiently close to $0$ in the good region of parameters (\ref{eq:good-region}), the trace $T^{w_0}(q, \lambda, \omega, \mu, k)$ is related to the Felder-Varchenko function by 
\begin{multline*}
T^{w_0}(q, \lambda, \omega, \mu, k) =  \frac{q^{-\mu + 4} (q^{-4}; q^{-2\omega})}{\theta_0(q^{2\lambda}; q^{-2\omega})(q^{2\lambda - 2} q^{-2\omega}; q^{-2\omega})(q^{-2\lambda - 2}; q^{-2\omega})} \\
 \frac{(q^{-2 \omega + 2}; q^{-2\omega}, q^{-2 k})^2}{(q^{-2 \omega - 2}; q^{-2 \omega}, q^{-2k})^2} \frac{ (q^{-2k};q^{-2k})(q^4 q^{-2k}; q^{-2k})}{(q^{-2\mu + 2};q^{-2k})(q^{2\mu + 2} q^{-2k}; q^{-2k})}u(q, \lambda, \omega, -\mu, k). 
\end{multline*}
\end{corr}
\addtocounter{thm}{-1}}

\begin{remark}
Together Theorem \ref{thm:int-trace} and Corollary \ref{corr:trace-fv} form the quantum affine generalization of \cite[Theorem 8.1]{EV, ESV, EV3} and resolve the first case of Etingof-Varchenko's conjecture from \cite{EV}.  Further, Theorem \ref{thm:int-trace} gives the first proof that a trace function for the quantum affine algebra is analytic, as conjectured in \cite[Remark 1]{ESV}.\footnote{Integral formulas for traces of screened vertex operators are given in \cite{Kon2, Kon} as noted in \cite[Remark 1]{ESV}, but there has been no prior analysis of the convergence of the Jackson integrals.}
\end{remark}

We obtain also a symmetry property for a renormalization of the trace function motivated by representation theory.  Define a normalized trace by
\[
\wT^{w_0}(q, \lambda, \omega, \mu, k) := \Tr|_{M_{-\rho - 2 \Lambda_0}}(q^{2\lambda \rho + 2 \omega d})^{-1}T^{w_0}(q^{-1}, -\lambda, -\omega, \mu, k),
\]
where we interpret $T^{w_0}(q^{-1}, -\lambda, -\omega, \mu, k)$ via the quasi-analytic continuation of $T^{w_0}(q, \lambda, \omega, \mu, k)$ to the region $|q^{-2\omega}| < 1$ and $|q^{-2k}| > 1$.  By this, we mean that the coefficients of $T^{w_0}(q, \lambda, \omega, \mu, k)$ as a formal series in $q^{-2\omega}$ converge to rational functions in $q^{-2\mu}$ and $q^{-2k}$ for $|q^{-2k}| < 1$ and $|q^{-2\mu}| < 1$, and $T^{w_0}(q^{-1}, -\lambda, -\omega, \mu, k)$ is interpreted by evaluating the same rational function coefficients in the region $|q^{-2k}| > 1$ and $|q^{-2\mu}| < 1$.  We show the following symmetry property for the renormalized trace.

{\renewcommand{\thethm}{\ref{thm:tnorm-sym}}
\begin{thm} 
The function $\wT^{w_0}(q, \lambda, \omega, \mu, k)$ is symmetric under interchange of $(\lambda, \omega)$ and $(\mu, k)$.
\end{thm}
\addtocounter{thm}{-1}}

\begin{remark}
Theorem \ref{thm:tnorm-sym} is the generalization to the quantum affine case of the symmetry of \cite[Proposition 6.3]{EV}.  In our setting, an additional difficulty which does not arise in the trigonometric limit is the necessity of introducing quasi-analytic continuation.  We expect the other symmetry of \cite[Theorem 1.5]{EV} to generalize to our setting as well.
\end{remark}

\subsection{Relation to elliptic and affine Macdonald polynomials}

In \cite{FV4}, Felder-Varchenko introduced the elliptic Macdonald polynomials $\wJ_{\mu, k}(q, \lambda, \omega)$ as hypergeometric theta functions in terms of their eponymous functions.  In the trigonometric limit, they showed that these functions were related to Macdonald polynomials, and they conjectured that in the general case they were related to the affine Macdonald polynomials $J_{\mu, k, 2}(q, \lambda, \omega)$ of \cite{EK3}. We prove this conjecture in Theorem \ref{thm:fv-conj}.
{\renewcommand{\thethm}{\ref{thm:fv-conj}}
\begin{thm}
Let $\wk = k + 4$.  For $q^{-2\mu}$ and then $q^{-2\omega}$ sufficiently close to $0$ in the good region of parameters (\ref{eq:good-region}), the elliptic and affine Macdonald polynomials are related by
\begin{multline*}
J_{\mu, k, 2}(q, \lambda, \omega)\\
 = \frac{\wJ_{\mu, \wk}(q, \lambda, \omega)}{f(q, q^{-2\omega})} \frac{(q^{-4}; q^{-2\omega})(q^{-2\omega}; q^{-2\omega})^3}{(q^{-4}; q^{-2\wk})(q^{-2\wk}; q^{-2\wk})} \frac{(q^{-2\omega + 2}; q^{-2\omega}, q^{-2\wk})^2}{(q^{-2\omega - 2}; q^{-2\omega}, q^{-2\wk})^2} q^{\mu + 4}(q^{-2\mu - 6}; q^{-2\wk})(q^{2\mu + 2} q^{-2\wk}; q^{-2\wk}),
\end{multline*}
where $f(q, q^{-2\omega})$ is the normalizing function of Proposition \ref{prop:ek-const}.
\end{thm}
\addtocounter{thm}{-1}}
\begin{remark}
This theorem shows that $\wJ_{\mu, k}(q, \lambda, \omega)$ and $J_{\mu, k, 2}(q, \lambda, \omega)$ are proportional with constant of proportionality explicit aside from dependence on the normalization constant $f(q, q^{-2\omega})$ for the denominator of the affine Macdonald polynomial left undetermined in \cite{EK3}.  In future work, we plan to use the results of the present paper to explicitly evaluate $f(q, q^{-2\omega})$.
\end{remark}

\subsection{Relation to the Felder-Varchenko conjecture}

One of our motivations for studying the trace function $T(q, \lambda, \omega, \mu, k)$ originates in two streams of prior work.  In \cite{ESV}, it was shown that traces of intertwining operators for quantum affine algebras satisfy the Macdonald-Ruijsenaars, dual Macdonald-Ruijsenaars, $q$-KZB, and dual $q$-KZB equations, giving four commuting systems of difference equations.\footnote{In this paper, we consider only traces valued in the dimension $3$ irreducible representation, for which the $q$-KZB and dual $q$-KZB equations are trivial.}  On the other hand, in \cite{FTV, FTV2}, it was shown that the Felder-Varchenko functions satisfy the $q$-KZB and dual $q$-KZB equations, and in \cite{FV2, FV} it was conjectured that they also satisfy the $q$-KZB heat equation, which is non-trivial in the case considered in this paper.

In \cite{EV, ESV, EV3}, Etingof-Varchenko posed the natural conjecture that the two resulting functions are related by a simple renormalization, and in the trigonometric and classical limits proofs of this conjecture for the $\sl_2$ case were given in \cite{EV, EK2, SV}.  Corollary \ref{corr:trace-fv} resolves this conjecture in the case of functions valued in the three dimensional representation.  This gives a representation-theoretic interpretation of the Felder-Varchenko function and provides the first proof that the trace function for $U_q(\asl_2)$ converges for parameters lying in a certain region.

In future work, we aim to use this interpretation to study the Felder-Varchenko functions using the representation theory of quantum affine algebras.  In particular, we hope to prove symmetry properties similar to that of Theorem \ref{thm:tnorm-sym} and to approach the Felder-Varchenko conjecture that the Felder-Varchenko function satisfies the $q$-KZB heat equation.  This would yield generalizations of the corresponding theorems shown in the trigonometric limit in \cite{EV3}.

\subsection{Relation to geometry of Laumon spaces}

In \cite{Neg1, Neg2}, Negut realizes certain intertwiners for $\sl_n$ and $\agl_n$ via geometric actions of these algebras on the equivariant cohomology of certain moduli spaces known as (affine) Laumon spaces.  He interprets traces of these intertwiners as generating functions for the integrals of Chern polynomials of tangent bundles of the Laumon spaces, thereby relating the generating functions to Calogero-Moser systems via results of \cite{Eti3}.  This picture is expected to admit quantization, giving a relation between $U_q(\asl_n)$-intertwiners and the $K$-theory of the affine Laumon space; for instance, an action of the quantum loop algebra $U_q(L\sl_n)$ on the $K$-theory of the affine Laumon space was constructed in \cite{Tsy}.  Under this expected correspondence, our trace function would thereby encode certain intersection-theoretic computations in $K$-theory giving enumerative information about the affine Laumon spaces.

\subsection{Outline of method and organization}

We briefly outline our method. The main technical inputs to our computation of the trace function are the free field construction of the $q$-Wakimoto module of \cite{Mat} and the method of coherent states for one loop correlation functions.  Combining these tools, choosing contours carefully to ensure convergence, and performing some intricate integral manipulations yields the Jackson integral expression in Proposition \ref{prop:ff-comp} for the trace function.  In the good region of parameters, we then check that this expression agrees with the formal expansion of the Felder-Varchenko function to deduce that the trace function converges and has the integral form of Theorem \ref{thm:int-trace}.  

In the remainder of the paper, we take degenerations and give some applications of Theorem \ref{thm:int-trace}.  In Theorems \ref{thm:class-lim} and \ref{thm:trig-consist}, we take the classical and trigonometric limits of our expression and show that they reproduce prior results from \cite{EK3} and \cite{EV}.  In Theorem \ref{thm:tnorm-sym}, we show that a normalized quasi-analytic continuation of our trace function has a symmetry property predicted by representation theory.  Finally, in Theorem \ref{thm:fv-conj}, we use the BGG resolution for irreducible integrable modules to compute the affine Macdonald polynomial in terms of our trace functions.  The resulting expression takes the form of the hypergeometric theta functions posed as elliptic Macdonald polynomials by Felder-Varchenko in \cite{FV4}, resolving their conjecture on the connection between the two.

The remainder of this paper is organized as follows. The technical heart of the paper is in Sections \ref{sec:tr-fn} to \ref{sec:ci}. In Section \ref{sec:tr-fn}, we define our notations for $U_q(\asl_2)$ and for trace functions for $U_q(\asl_2)$-intertwiners.  In particular, we fix a coproduct which agrees with \cite{Mat} and is the opposite of that of \cite{ESV}.  In Section \ref{sec:fv-fn}, we fix notation for the Felder-Varchenko function and give a formal series expansion and quasi-analytic continuation for it.  In Section \ref{sec:ff}, we fix notations for the free field realization of $U_q(\asl_2)$-modules given in \cite{Mat} and compute normalizations of $q$-vertex operator expressions for intertwiners.  In Section \ref{sec:ci}, we apply the method of coherent states to the $q$-vertex operators from Section \ref{sec:ff} to obtain a contour integral formula for the trace function in Theorem \ref{thm:int-trace} and identify it with the Felder-Varchenko function in Corollary \ref{corr:trace-fv}. 

In the remaining sections, we apply and take degenerations of Theorem \ref{thm:int-trace}. In Sections \ref{sec:class} and Sections \ref{sec:trig}, we verify in Theorems \ref{thm:class-lim} and \ref{thm:trig-consist} that our computations are consistent with existing computations in the classical and trigonometric limits. In Section \ref{sec:sym}, we show that a renormalization of the trace function satisfies a symmetry property motivated by representation theory.  In Section \ref{sec:aff-mac}, we relate the affine and elliptic Macdonald polynomials to our trace functions using the BGG resolution and dynamical Weyl group for $U_q(\asl_2)$-modules and prove Felder-Varchenko's conjecture on their relation. Appendices \ref{sec:ell} and \ref{sec:comps} contain notations and estimates for elliptic functions and computations of OPE's and one loop correlation functions which occur in the method of coherent states.

\subsection{Acknowledgments} 

The author thanks P. Etingof for suggesting the problem and for many helpful discussions.  Y. S. was supported by a NSF Graduate Research Fellowship (NSF Grant \#1122374).

\section{The trace function for $U_q(\asl_2)$-intertwiners} \label{sec:tr-fn}

In this section, we give our notations and conventions for $U_q(\asl_2)$ and define the trace functions for $U_q(\asl_2)$-intertwiners which will be the focus of this work.  Our coproduct is the opposite of that in \cite{ESV} and therefore some of our variable shifts are also different.

\subsection{The Cartan subalgebra of $\asl_2$ and $\wsl_2$}

Denote by $\asl_2$ the affinization of $\sl_2$ and by $\wsl_2$ its central extension.  The Cartan subalgebra of $\wsl_2$ and its dual are given by
\[
\whh = \CC \alpha \oplus \CC c \oplus \CC d \text{ and } \whh^* = \CC \alpha \oplus \CC \Lambda_0 \oplus \CC \delta,
\]
where $\Lambda_0 = c^*$ and $\delta = d^*$.  Define $\rho = \frac{1}{2}{\alpha}$, and recall that $\sl_2$ has dual Coxeter number $\ch = 1 + (\theta, \rho) = 2$, where $\theta = \alpha$ is the highest root.  Define $\alpha_1 := \alpha$, $\alpha_0 := \delta - \alpha \in \whh^*$, and $\wrho := \rho + 2 \Lambda_0$.

The algebra $\wsl_2$ admits a non-degenerate invariant form $(-, -)$ whose restriction to $\whh$ has non-trivial values
\[
(\alpha, \alpha) = 2,  \qquad (c, d) = 1, \qquad (d, d) = 0
\]
and agrees with the form on $\hh$.  This defines the identification $\whh \simeq \whh^*$ given by
\[
\alpha \mapsto \alpha, \qquad c \mapsto \delta, \qquad d \mapsto \Lambda_0.
\]
Transporting the form to $\whh^*$ yields the non-trivial values
\[
(\alpha, \alpha) = 2, \qquad (\delta, \Lambda_0) = 1, \qquad (\Lambda_0, \Lambda_0) = 0.
\]
The Cartan matrix $A = (a_{ij})$ for $\asl_2$ is defined by $a_{ij} := (\alpha_i, \alpha_j)$. 

\subsection{The algebras $U_q(\asl_2)$ and $U_q(\wsl_2)$}

Let $q$ be a non-zero complex number transcendental over $\QQ$, and for an integer $n$ define the $q$-number by $[n] = \frac{q^n - q^{-n}}{q - q^{-1}}$, $q$-factorial by $[n]! = [n] \cdot \cdots \cdot [1]$, $q$-falling factorial by $[n]_l = [n] \cdots [n - l + 1]$, and $q$-binomial coefficient by $\binom{n}{m}_q = \frac{[n]!}{[m]![n - m]!}$.  The quantum affine algebra $U_q(\asl_2)$ is the Hopf algebra generated as an algebra by $e_i, f_i, q^{\pm h_i}$ for $i = 0, 1$ with relations
\begin{align*}
[q^{h_i}, q^{h_j}] &= 0 \qquad q^{h_i} e_j q^{-h_j} = q^{(h_i, \alpha_j)} e_j \qquad q^{h_i} f_j q^{-h_j} = q^{-(h_i, \alpha_j)} f_j \qquad [e_i, f_j] = \delta_{ij} \frac{q^{h_i} - q^{-h_i}}{q - q^{-1}} \\
\sum_{k = 0}^{1 - a_{ij}}& (-1)^k \binom{1 - a_{ij}}{k}_q e_i^{1 - a_{ij} - k} e_j e_i^k = 0 \qquad \sum_{k = 0}^{1 - a_{ij}} (-1)^k \binom{1 - a_{ij}}{k}_q f_i^{1 - a_{ij} - k} f_j f_i^k = 0.
\end{align*}
The coproduct, antipode, and counit of $U_q(\asl_2)$ are
\begin{align*}
\Delta(e_i) &= e_i \otimes 1 + q^{h_i} \otimes e_i, \qquad \Delta(f_i) = f_i \otimes q^{-h_i} + 1 \otimes f_i, \qquad \Delta(q^{h_i}) = q^{h_i} \otimes q^{h_i};\\
S(e_i) &= - q^{-h_i} e_i, \qquad S(f_i) = - f_iq^{h_i}, \qquad S(q^{h_i}) = q^{-h_i};\\
\eps(e_i) &= \eps(f_i) = 0, \qquad \eps(q^h) = 1.
\end{align*}
We centrally extend $U_q(\asl_2)$ to $U_q(\wsl_2)$ by adding a generator $q^d$ which commutes with $q^{h_i}$ and whose commutators with $e_i$ and $f_i$ are
\[
[q^d, e_i] = [q^d, f_i] = 0 \text{ for $i \neq 0$}, \qquad q^de_0 q^{-d} = q e_0, \qquad q^df_0q^{-d} = q^{-1} f_0
\]
and on which the coproduct, antipode, and counit are
\[
\Delta(q^d) = q^d \otimes q^d, \qquad S(q^d) = q^{-d}, \qquad \eps(q^d) = 1.
\]
Define the subalgebras $U_q(\abb_+) = \langle e_i, q^{\pm h_i}\rangle$, $U_q(\abb_-) = \langle f_i, q^{\pm h_i}\rangle$, $U_q(\ann_+) = \langle e_i \rangle$, and $U_q(\ann_-) = \langle f_i\rangle$ of $U_q(\asl_2)$.

\begin{remark}
This coproduct agrees with that of \cite{Mat, Kon} and is opposite to that of \cite{ESV}.
\end{remark}

\subsection{Verma modules for $U_q(\asl_2)$}

We denote by $M_{\mu, k} := M_{\mu \rho + k \Lambda_0}$ the Verma module for $U_q(\asl_2)$ with highest weight $\mu \rho + k \Lambda_0$ and by $v_{\mu, k} \in M_{\mu, k}$ a canonically chosen highest weight vector; this module is of level $k$.  We extend it to a $U_q(\wsl_2)$-module by letting $q^d$ act by $1$ on $v_{\mu, k}$. Define the restricted dual of $M_{\mu, k}$ by
\[
M_{\mu, k}^\vee := \bigoplus_{m \geq 0} M_{\mu, k}[-m \delta]^*,
\]
where the action of $U_q(\asl_2)$ is given by $(u \cdot \phi)(m) := \phi(S(u) m)$.  Let $v_{\mu, k}^*$ be the dual vector to the highest weight vector $v_{\mu, k}$. For generic $(\mu, k)$, the Verma module $M_{\mu, k}$ is irreducible; we will make this assumption throughout the paper outside of Section \ref{sec:aff-mac}.

\begin{remark}
In the notation of \cite{Mat}, $M_{\mu, k}$ is the Verma module of spin $\frac{\mu}{2}$ at level $k$.  In the notation of \cite{ESV}, $M_{\mu, k}$ is equal to $M^{ESV}_{\mu - k/2, k}$. 
\end{remark}

Define the algebra anti-automorphism $\omega: U_q(\asl_2) \to U_q(\asl_2)$ by $\omega(e_i) = f_i$, $\omega(f_i) = e_i$, and $\omega(q^h) = q^h$.  Define a symmetric form on $M_{\mu, k}$ by 
\[
\FF(a v_{\mu, k}, b v_{\mu, k}) = \langle \omega(a) v_{\mu, k}^*, b v_{\mu, k}\rangle, 
\]
where $a, b \in U_q(\ann_-)$ and $\omega(a) v_{\mu, k}^* \in M_{\mu, k}^\vee$.  On $M_{\mu, k}[\mu\rho + k \Lambda_0 - \lambda\rho - a\delta]$, it is related to the Shapovalov $F(-, -)$ form of \cite[Theorem 4.1.16]{Jos1995} by 
\[
\FF(-, -) = C_\lambda q^{-(\mu\rho + k \Lambda_0, \lambda\rho + a \delta)} F(-, -)
\]
for some constant $C_\lambda$. Noting that by \cite[Lemma 3.4.8]{Jos1995} the Kostant partition function satisfies
\[
P(\lambda\rho + a \delta)(\lambda \rho + a \delta) = \sum_{\beta > 0} \sum_{n \geq 1} P(\lambda \rho + a \delta - n \beta) \beta,
\]
we obtain the following scaling of the Kac-Kazhdan determinant formula for $\FF$.

\begin{prop}[{\cite[Theorem 4.1.16]{Jos1995}}] \label{prop:kk-det}
When restricted to $M_{\mu, k}[\mu \rho + k \Lambda_0 - \lambda \rho - a \delta]$, the form $\FF$ has determinant
\[
C \prod_{\beta > 0} \prod_{n \geq 1} (1 - q^{-2(\beta, \mu\rho + k \Lambda_0 + \wrho) + n(\beta, \beta)})^{P(\lambda\rho + a \delta - n \beta)},
\]
where $C$ is a non-zero constant depending on choice of basis, $\beta$ ranges over positive roots $\{\alpha, \pm \alpha + m \delta, m \delta \mid m > 0\}$, and $P$ denotes the Kostant partition function.
\end{prop}

\begin{remark}
Proposition \ref{prop:kk-det} is the quantum analogue of the Kac-Kazhdan determinant formula of \cite{KK}.
\end{remark}

\subsection{Evaluation modules for $U_q(\asl_2)$}

Let $M_{-\mu}^\vee$ denote the restricted dual Verma module of lowest weight $-\mu$ for $U_q(\sl_2)$; if $\mu$ is a positive integer, let $L_\mu \subset M_{-\mu}^\vee$ be the corresponding irreducible finite dimensional module.  Let $M_{-\mu}^\vee(z)$ and $L_\mu(z)$ denote the corresponding affinizations.  We choose an explicit basis $w_{2\mu}, \ldots, w_{-2\mu}$ for $L_{2\mu}(z)$ so that $U_q(\asl_2)$ acts by
\begin{align*}
e_1 w_{2m} \otimes z^n &= [\mu - m] w_{2m+2} \otimes z^n\\
e_0 w_{2m} \otimes z^n &= [\mu + m] w_{2m-2} \otimes z^{n+1}\\
f_1 w_{2m} \otimes z^n &= [\mu + m] w_{2m-2} \otimes z^n\\
f_0 w_{2m} \otimes z^n &= [\mu - m] w_{2m+2} \otimes z^{n-1}\\
q^{h_1} w_{2m} \otimes z^n &= q^{2m} w_{2m} \otimes z^n\\
q^{h_0} w_{2m} \otimes z^n &= q^{-2m} w_{2m} \otimes z^n \\
q^d w_{2m} \otimes z^n &= q^n w_{2m} \otimes z^n.
\end{align*}

\begin{remark}
This basis is related to the basis $v_{\mu, m}$ in \cite{Mat} by $w_{2m} = v_{\mu, -m}$.
\end{remark}

\subsection{Completed tensor product and intertwiners}

Let $V(z)$ be a finite-dimensional evaluation representation for $U_q(\asl_2)$.  Define the completed tensor product by
\[
M_{\mu, k} \hotimes V(z) := (M_{\mu, k}^\vee)^* \otimes V(z) \simeq \Hom_\CC(M_{\mu, k}^\vee, V(z)),
\]
where by $(M_{\mu, k}^\vee)^*$ we mean the full linear dual, the $U_q(\asl_2)$-action on $\Hom_\CC(M_{\mu, k}^\vee, V(z))$ is given in Sweedler notation by $(u \cdot \rho)(m) = u_{(2)} \rho(S(u_{(1)}) m)$, and the isomorphism is given by $\phi^* \otimes v \mapsto \Big(\psi \mapsto \phi^*(q^{-2\wrho} \psi) v\Big)$.  As a vector space we have an isomorphism
\[
M_{\mu, k} \hotimes V(z) = (M_{\mu, k}^\vee)^* \otimes V(z) \simeq \prod_{a \geq 0} \Big(M_{\mu, k}[-a \delta] \otimes V(z)\Big)
\]
under which elements of $M_{\mu, k} \hotimes V(z)$ are sums $\sum_{i = 0}^\infty m_i \otimes v_i$ with $m_i, v_i$ homogeneous and $\lim_{i \to \infty} \deg(m_i) = \infty$.  For $v \in V[\tau]$, if $M_{\mu, k}^\vee$ is irreducible, by \cite[Theorem 9.3.1]{EFK} there is a unique $U_q(\asl_2)$-intertwiner 
\[
\Phi_{\mu, k}^{v}(z): M_{\mu, k} \to M_{\mu - \tau, k} \hotimes V(z)
\]
which satisfies
\[
\Phi_{\mu, k}^v(z) v_{\mu, k} = v_{\mu, k} \otimes v + (\text{l.o.t.}),
\]
where $(\text{l.o.t.})$ denotes terms of lower weight in the first tensor factor. If $V_1, \ldots, V_n$ are finite-dimensional representations, $v_i \in V_i[\tau_i]$, define the iterated intertwiner by the composition of intertwiners
\[
\Phi_{\mu, k}^{v_1, \ldots, v_n}(z_1, \ldots, z_n): M_{\mu, k} \to M_{\mu - \tau_n, k} \hotimes V_n(z_n) \to \cdots \to M_{\mu - \tau_1 - \cdots - \tau_n, k} \hotimes V_1(z_1) \hotimes \cdots \hotimes V_n(z_n).
\]

\begin{remark}
These intertwiners are known as Type I intertwiners.  Due to our coproduct convention, they correspond to Type II intertwiners in the coproduct of \cite{ESV}.
\end{remark}

\subsection{Definition of the trace function}

Define the trace function $\Psi^{v_1, \ldots, v_n}(z_1, \ldots, z_n; \lambda, \omega, \mu, k)$ as the formal power series in $q^{-2\omega}$ given by
\[
\Psi^{v_1,\ldots, v_n}(z_1, \ldots, z_n; \lambda, \omega, \mu, k) := \Tr|_{M_{\mu, k}}\Big(\Phi^{v_1, \ldots, v_n}_{\mu, k}(z_1, \ldots, z_n)q^{2\lambda \rho + 2\omega d}\Big).
\]
In the case $n = 1$, the trace function $\Psi^{v}(z; \lambda, \omega, \mu, k)$ is independent of $z$, and we use the notation 
\begin{equation} \label{eq:sym-tr-def}
T^v(q, \lambda, \omega, \mu, k) := \Psi^v(z; \lambda, \omega, \mu - 1, k - 2) = \Tr|_{M_{(\mu - 1)\rho + (k-2)\Lambda_0}}\Big(\Phi_{\mu - 1, k - 2}^{v}(z) q^{2\lambda\rho + 2\omega d}\Big).
\end{equation}
If $v = w_0 \in L_{2m}[0]$, then $T^{w_0}(q, \lambda, \omega, \mu, k)$ lies in $L_{2m}[0] \simeq \CC \cdot w_0$, and we interpret it as a $\CC$-valued function via the identification $\CC \cdot w_0 \simeq \CC$.

\begin{remark}
Our notation for $\Psi^{v_1, \ldots, v_n}(z_1, \ldots, z_n; \lambda, \omega, \mu, k)$ is related to that of \cite{ESV} by the use of the opposite coproduct and the variable shifts
\[
\Psi^{v_1, \ldots, v_n}(z_1, \ldots, z_n; \lambda, \omega, \mu, k) = \Psi_{ESV}^{v_1, \ldots, v_n}(z_1, \ldots, z_n; \lambda - \omega/2, \omega, \mu - k/2, k).
\]
We choose our conventions to make the classical and trigonometric limits more transparent.
\end{remark}

\begin{remark}
As defined, all trace functions we consider are formal power series in $q^{-2\omega}$. Our main result Theorem \ref{thm:int-trace} shows that for $w_0 \in L_2$ the formal power series for $T^{w_0}(q, \lambda, \omega, \mu, k)$ converges in a certain region of parameters, meaning it is analytic.
\end{remark}

\subsection{Rationality properties of the trace function}

The goal of this subsection is to prove Proposition \ref{prop:rat-fn-trace}, which characterizes the coefficient of the power series expansion of the trace function in $q^{-2\omega}$.  We begin by giving an explicit expression for the intertwiner in terms of the form $\FF$ on $M_{\mu, k}$.  Choose a homogeneous basis $\{g^{\lambda, a}_i\}$ for $U_q(\ann_-)$ so that $\wt(g^{\lambda, a}_i) = - \lambda \rho - a\delta$, and define the matrix elements 
\[
\FF_{\lambda, a}(\mu, k)_{ij} = \FF(g^{\lambda, a}_i v_{\mu, k}, g^{\lambda, a}_j v_{\mu, k}) \qquad
\]
of the form and the elements $\FF^{-1}_{\lambda, a}(\mu, k)_{ij}$ of its inverse matrix.

\begin{lemma} \label{lemma:sing-vect}
For $v \in V(z)$, the vector
\[
\phi_{\mu, k, v} = \sum_{\lambda, a} \sum_{i, j} \FF^{-1}_{\lambda, a}(\mu, k)_{ij}\, (q^{2\wrho} g^{\lambda, a}_i v_{\mu, k}) \otimes \omega(g^{\lambda, a}_j) v
\]
in $M_{\mu, k} \hotimes V(z)$ is singular.
\end{lemma}
\begin{proof}
Notice that $\phi_{\mu, k, v}$ is singular if and only if the composition
\[
\psi: M^\vee_{\mu, k} \overset{1 \otimes \phi_{\mu, k, v}} \to M^\vee_{\mu, k} \otimes (M_{\mu, k}^\vee)^* \otimes V(z) \to V(z)
\]
induced by $\phi_{\mu, k, v}$ and $m \otimes f \mapsto f(q^{-2\wrho} m)$ is a map of $U_q(\ann_+)$-modules.  This holds since
\begin{multline*}
\psi(\omega(g^{\nu, b}_l)v_{\mu, k}^*) = \sum_{\lambda, a} \sum_{i, j} \FF^{-1}_{\lambda, a}(\mu, k)_{ij}\, \langle \omega(g^{\nu, b}_l) v_{\mu, k}^*, g^{\lambda, a}_i v_{\mu, k}\rangle \omega(g^{\lambda, a}_j) v\\
= \sum_{i, j} \FF^{-1}_{\nu, b}(\mu, k)_{ij}\, \FF_{\nu, b}(\mu, k)_{li} \omega(g^{\nu, b}_j) v = \sum_j 1_{lj}\, \omega(g^{\nu, b}_j) v = \omega(g^{\nu, b}_l) v. \qedhere
\end{multline*}
\end{proof}

\begin{lemma} \label{lem:sing-denom}
The matrix elements $\FF_{\lambda, a}^{-1}(\mu, k)_{ij}$ of the inverse of $\FF$ are rational functions in $q^{-2\mu}$ and $q^{-2k}$ with at most simple poles whose denominators are products of linear terms of the form 
\[
(1 - q^{-2(\mu + 1) - 2m(k + 2) + 2n}), \qquad (1 - q^{2(\mu + 1) - 2(m + 1)(k + 2) + 2n}), \qquad (1 - q^{-2(m + 1)(k + 2)}) \qquad \text{ for $m \geq 0$, $n \geq 1$.}
\]
\end{lemma}
\begin{proof}
This is proven in the same way as \cite[Lemma 3.2]{ESt2}, noting that these linear terms are exactly those appearing in the Kac-Kazhdan determinant of Proposition \ref{prop:kk-det}.
\end{proof}

\begin{lemma} \label{lem:phi-denom}
If $v \in V[0]$, the only linear terms from Lemma \ref{lem:sing-denom} which appear in coefficients $\FF_{\lambda, a}^{-1}(\mu, k)_{ij}$ of $\phi_{\mu, k, v}$ are those for which $V[n\alpha] \neq 0$ or $V[-n\alpha] \neq 0$.
\end{lemma}
\begin{proof}
This is proven in the same way as \cite[Proposition 2.2]{ESt}.
\end{proof}

\begin{lemma} \label{lem:rat-fn-me}
For $w_0 \in L_2[0]$, as functions of $q^{-2\mu}$ and $q^{-2k}$, matrix elements of $\Phi^{w_0}_{\mu, k}(z)$ in the PBW basis may be written in the form $\frac{P_{ij}(q^{-2\mu}, q^{-2k})}{Q_{ij}(q^{-2\mu}, q^{-2k})}$ where
\begin{itemize}
\item $P_{ij}(q^{-2\mu}, q^{-2k})$ is a polynomial in $q^{-2\mu}$ and Laurent polynomial in $q^{-2k}$;

\item $Q_{ij}(q^{-2\mu}, q^{-2k})$ is a polynomial in $q^{-2\mu}$ and Laurent polynomial in $q^{-2k}$ given by a product of linear terms of the form 
\[
(1 - q^{-2(\mu + 1) - 2m(k + 2) \pm 2}),\qquad  (1 - q^{2(\mu + 1) - 2(m + 1)(k + 2) \pm 2}), \qquad (1 - q^{-2(m + 1)(k + 2)}) \qquad \text{ for $m \geq 0$};
\]

\item in each graded piece $M_{\mu, k}[-a \delta]$, the $P_{ij}$ contain a finite number of distinct non-zero monomials and the $Q_{ij}$ contain a finite number of linear factors.
\end{itemize}
\end{lemma}
\begin{proof}
This follows for $\phi_{\mu, k, w_0} = \Phi^{w_0}_{\mu, k}(z) v_{\mu, k}$ by Lemma \ref{lem:phi-denom}.  Each other matrix element of $\Phi^{w_0}_{\mu, k}(z)$ is computed by acting on $\phi_{\mu, k, w_0}$ by monomials composed of products of $\Delta(f_i)$ for different $i$.  Therefore, matrix elements in $M_{\mu, k}[-a\delta]$ for $a \leq A$ are finite linear combinations of the coefficients of $\phi_{\mu, k, w_0}$ in degree at least $(-A)$ scaled by constants independent of $\mu$ and $k$, yielding the desired for all matrix elements.
\end{proof}

We now constrain the structure of poles of the trace function.  This result is similar to that of \cite[Proposition 3.1]{ESt}.  However, in the quantum affine setting, it is possible for the trace function to have an infinite number of poles, but as a formal power series in $q^{-2\omega}$ each coefficient does not.

\begin{prop} \label{prop:rat-fn-trace}
For $w_0 \in L_2[0]$, the function $\Psi^{w_0}(z; \lambda, \omega, \mu, k)$ may be written in the form 
\[
\Psi^{w_0}(z; \lambda, \omega, \mu, k) = q^{\lambda\mu} \sum_{n \geq 0} q^{-2\omega n} \wPsi^{w_0}_n(z; \lambda, \mu, k)
\]
as a formal power series in $q^{-2\omega}$ where $\wPsi^{w_0}_n(z; \lambda, \mu, k)$ takes the form $\frac{P_n(q^{-2\mu}, q^{-2k})}{Q_n(q^{-2\mu}, q^{-2k})}$ with $P_n(q^{-2\mu}, q^{-2k})$ polynomial in $q^{-2\mu}$ and Laurent polynomial in $q^{-2k}$ and $Q_n(q^{-2\mu}, q^{-2k})$ a polynomial given by a finite product of linear terms of the form 
\[
(1 - q^{-2(\mu + 1) - 2m(k + 2) \pm 2}),\qquad (1 - q^{2(\mu + 1) - 2(m + 1)(k + 2) \pm 2}), \qquad (1 - q^{-2(m+1)(k+2)}) \qquad \text{ for $m \geq 0$}.
\]
\end{prop}
\begin{proof}
This follows from Lemma \ref{lem:rat-fn-me} in the same way as \cite[Proposition 3.1]{ESt} follows from \cite[Proposition 2.2]{ESt}.
\end{proof}

\section{The Felder-Varchenko function} \label{sec:fv-fn}

In this section, we introduce the Felder-Varchenko functions, specialize their definition to the three-dimensional irreducible representation of $U_q(\sl_2)$, and derive some of their properties which we will use later.  In particular, we give a series expansion and quasi-analytic continuation.

\subsection{Definition of the Felder-Varchenko function}

For $\Lambda = (\Lambda_1, \ldots, \Lambda_n)$, let 
\[
L_\Lambda = L_{\Lambda_1} \otimes \cdots \otimes L_{\Lambda_n}
\]
be the tensor product of finite dimensional irreducible representations of $U_q(\sl_2)$ with highest weights $\Lambda_1, \ldots, \Lambda_n$ and highest weight vectors $v_{\Lambda_1}, \ldots, v_{\Lambda_n}$. In \cite{FV2}, Felder-Varchenko defined the $L_\Lambda[0] \otimes L_\Lambda[0]$-valued functions
\begin{multline} \label{eq:fv-def}
u(z_1, \ldots, z_n, q, \lambda, \omega, \mu, k, q) = q^{-\lambda\mu} \int_\gamma \prod_{i, j}\Omega_{q^{\Lambda_j}}\Big(\frac{t_i}{z_j}; q^{-2\omega}, q^{-2k}\Big) \prod_{i < j}\Omega_{q^{-2}}\Big(\frac{t_i}{t_j}; q^{-2\omega}, q^{-2k}\Big) \\ \sum_{I, J} \omega_I(t_1, \ldots, t_m, z_1, \ldots, z_n, \lambda, q^{-2\omega}, \Lambda) \, \omega_J^\vee(t_1, \ldots, t_m, z_1, \ldots, z_n, \mu, q^{-2k}, \Lambda) \frac{dt_1 \cdots dt_m}{(2\pi i)^m t_1 \cdots t_m} e_I \otimes e_J.
\end{multline}
The notations in this definition are as follows.  The summation is over $I$ and $J$ so that for $e_I = e^{i_1} v_{\Lambda_1} \otimes \cdots \otimes e^{i_n} v_{\Lambda_n}$ and $e_J = e^{j_1} v_{\Lambda_1} \otimes \cdots \otimes e^{j_n} v_{\Lambda_n}$, $e_I$ and $e_J$ form bases of $L_\Lambda[0]$.  The phase function $\Omega$ is defined in (\ref{eq:phase-func-def}) by 
\[
\Omega_{a}(z; r, p) := \frac{(z a^{-1}; r, p)(z^{-1} a^{-1} rp; r, p)}{(z a; r, p)(z^{-1} a rp; r, p)},
\]
where $(u; r, p)$ is the double $q$-Pochhammer symbol of Subsection \ref{sec:theta-def}.  The weight function $\omega_I$ is defined by
\begin{align*}
\omega_{(i_1, \ldots, i_n)}(t_1, \ldots, t_m, z_1, \ldots, z_n, \lambda, r, \Lambda) &= \prod_{i < j}\frac{\theta(\frac{t_i}{t_j}; r)}{\theta(\frac{t_i}{t_j} q^2; r)} \sum_{I_1, \ldots, I_n} \prod_{l = 1}^n \prod_{i \in I_l}\prod_{j = 1}^{l-1} \frac{\theta(\frac{t_i}{z_j} q^{\Lambda_j}; r)}{\theta(\frac{t_i}{z_j} q^{-\Lambda_j}; r)}\\
&\phantom{=} \prod_{h < l} \prod_{i \in I_h, j \in I_l} \frac{\theta(\frac{t_i}{t_j}q^2; r)}{\theta(\frac{t_i}{t_j}; r)} \prod_{h = 1}^n \prod_{j \in I_h} \frac{\theta(\frac{t_j}{z_h}q^{2\lambda - \Lambda_h + 2 i_h - 2 \sum_{l = 1}^{h-1}(\Lambda_l - 2i_l)}; r)}{\theta(\frac{t_j}{z_h} q^{-\Lambda_h}, r)},
\end{align*}
where $\theta(-)$ is Jacobi's first theta function as defined in (\ref{eq:theta-def}) and the summation is over disjoint subsets $I_1, \ldots, I_n$ of $\{1, \ldots, m\}$ so that $|I_h| = i_h$.  The mirror weight function $\omega_I^\vee$ is defined by
\[
\omega_{i_1, \ldots, i_n}^\vee(t_1, \ldots, t_m, z_1, \ldots, z_n, \mu, p, \Lambda_1, \ldots, \Lambda_n) = \omega_{i_n, \ldots, i_1}(t_1, \ldots, t_m, z_n, \ldots, z_1, \mu, p, \Lambda_n, \ldots, \Lambda_1).
\]
Finally, the cycle of integration $\gamma$ is defined by analytic continuation from the domain on which 
\[
|q|, |q^{-2k}|, |q^{-2\omega}|, |q^{\Lambda_j}| < 1,
\]
where it is the $m$-fold product of unit circles.  We refer to the functions $u(z_1, \ldots, z_n, q, \lambda, \omega, \mu, k)$ as \textit{Felder-Varchenko} functions.

\begin{remark}
We have adopted different notations from those used in \cite{FV2}; in particular, we use multiplicative notation.  For $q = e^{2\pi i \eta}$, we have
\begin{equation} \label{eq:fv-trans}
u(q, \lambda, \omega, \mu, k) = u^{FV}(2 \lambda \eta, 2\mu \eta, -2\omega \eta, -2k\eta, \eta),
\end{equation}
where $u^{FV}(\lambda, \mu, \tau, p, \eta)$ denotes the function from \cite{FV2}.
\end{remark}

Let $n = 1$ and $\Lambda = (2)$, so that $L_\Lambda = L_2$, where $L_2$ is the $3$-dimensional irreducible representation of $U_q(\sl_2)$.  Notice then that $L_2[0] \otimes L_2[0]$ is $1$-dimensional, so we may consider $u(z, \lambda, \mu, \tau, p, \eta)$ as a numerical function.  In this case, we have
\[
\omega_1(t, z, \lambda, r) = \frac{\theta(\frac{t}{z}q^{2\lambda}; r)}{\theta(\frac{t}{z}q^{-2}; r)} \qquad \text{ and } \qquad \omega_1^\vee(t, z, \mu, p) = \frac{\theta(\frac{t}{z} q^{2\mu}; p)}{\theta(\frac{t}{z} q^{-2}; p)}.
\]
Specializing (\ref{eq:fv-def}) and noting that the result is independent of $z$ yields
\begin{equation} \label{eq:fv-def2}
u(q, \lambda, \omega, \mu, k) = q^{-\lambda\mu - \lambda - \mu - 2} \oint_{\cC_t} \frac{dt}{2\pi i t} \Omega_{q^2}(t; q^{-2\omega}, q^{-2k}) \frac{\theta_0(tq^{2\mu}; q^{-2k})}{\theta_0(tq^{-2}; q^{-2k})} \frac{\theta_0(t q^{2\lambda}; q^{-2\omega})}{\theta_0(t q^{-2}; q^{-2\omega})},
\end{equation}
where for $|q|, |q^{-2k}|, |q^{-2\omega}| < 1$ the contour $\cC_t$ is the unit circle and we omit the $z$ argument.  The remainder of this section will be devoted to proving some properties of the Felder-Varchenko function.

\subsection{Series expansion of the Felder-Varchenko function}

In this section, we give formal power series expansions for the Felder-Varchenko function in formal neighborhoods of $0$ and $\infty$ in $q^{-2\mu}$. Recall that the terminating $q$-Pochhammer symbol is $(u; q)_m = \frac{(u; q)}{(uq^m; q)}$ as in Subsection \ref{sec:theta-def}.

\begin{prop} \label{prop:fv-series}
As a formal power series in $q^{-2\omega}$, we have in a formal neighborhood of $0$ in $q^{-2\mu}$ that 
\begin{multline*}
q^{\lambda\mu}\frac{u(q, \lambda, \omega, \mu, k)}{\theta_0(q^{2\mu + 2}; q^{-2k})} =-q^{- \lambda - \mu - 2} \frac{(q^4;q^{-2\omega})(q^{-4}; q^{-2\omega}, q^{-2k})}{(q^{-4}; q^{-2k})(q^4; q^{-2\omega}, q^{-2k})} \frac{1}{(q^{-2k}; q^{-2k})(q^{-2\omega};q^{-2\omega})}\\
 \sum_{n \geq 0} q^{-(2\mu + 2)n} \frac{\theta_0(q^{2\lambda + 2}q^{-2kn}; q^{-2\omega})}{\theta_0(q^4 q^{-2kn};q^{-2\omega})} \prod_{l = 1}^{n} \frac{\theta_0(q^4 q^{-2kl}; q^{-2\omega})}{\theta_0(q^{-2kl}; q^{-2\omega})}
\end{multline*}
and in a formal neighborhood of $\infty$ in $q^{-2\mu}$ that 
\begin{multline*}
q^{\lambda\mu + \mu} \frac{u(q, \lambda, \omega, \mu, k)}{\theta_0(q^{2\mu - 2}; q^{-2k})} = q^{- \lambda - 2} \frac{(q^{-4}q^{-2\omega}; q^{-2k}, q^{-2\omega})}{(q^4 q^{-2k} q^{-2\omega}; q^{-2k}, q^{-2\omega})} \frac{1}{(q^{-2k}; q^{-2k})(q^{-2\omega}; q^{-2\omega})}\frac{1}{(q^{4}q^{-2k}; q^{-2k})} \\
\sum_{n \geq 0} q^{(2\mu + 2)n} \frac{\theta_0(q^{2\lambda - 2} q^{2kn}; q^{-2\omega})}{\theta_0(q^{-4}q^{2kn}; q^{-2\omega})} \prod_{l = 1}^n \frac{\theta_0(q^{-4}q^{2kl}; q^{-2\omega})}{\theta_0(q^{2kl}; q^{-2\omega})}.
\end{multline*}
\end{prop}
\begin{proof}
Denote the integrand of the Felder-Varchenko function by 
\[
v(t, q, \lambda, \omega, \mu, k) = \Omega_{q^2}(t; q^{-2\omega}, q^{-2k}) \frac{\theta_0(tq^{2\mu}; q^{-2k})}{\theta_0(tq^{-2}; q^{-2k})} \frac{\theta_0(t q^{2\lambda}; q^{-2\omega})}{\theta_0(t q^{-2}; q^{-2\omega})},
\]
and take the formal expansion 
\[
v(t, q, \lambda, \omega, \mu, k) = \sum_{n \geq 0} v_n(t, q, \lambda, \mu, k) q^{-2\omega n}.
\]
We may write $v(t, q, \lambda, \omega, \mu, k) = v_0(t, q, \lambda, \mu, k) \wv(t, q, \lambda, \omega, \mu, k)$, where $\wv$ has power series expansion
\[
\wv(t, q, \lambda, \omega, \mu, k) = 1 + \sum_{n > 0} \wv_n(t, q, \lambda, \mu, k) q^{-2\omega n}
\]
with each $\wv_n$ a Laurent polynomial in $t$ of degree at most $n$.  In particular, all poles of $v_n$ in $\CC^\times$ are those of $v_0$. For small enough $\eps > 0$, we may choose contours for $t$ with $|t| \to 0$ satisfying the hypotheses of Corollary \ref{corr:theta-ratio} for $\eps$.  On such contours, we have the estimate
\begin{equation} \label{eq:fv-contour-est}
\left|\frac{\theta_0(tq^{2\mu}; q^{-2k})}{\theta_0(tq^{-2}; q^{-2k})} t^n\right| \leq D_1(q^{-2k}, \eps) |q|^{\mu + 1} \left|t^2 q^{2\mu - 2}\right|^{\frac{\mu + 1}{2k}} |t|^n \leq C\, |t|^{\frac{\mu + 1}{k} + n}
\end{equation}
for a constant $C$ is independent of $t$.  

On a formal neighborhood of $0$ in $q^{-2\mu}$, because $\wv_n$ is a Laurent polynomial in $t$ of degree at most $n$, by (\ref{eq:fv-contour-est}) we may compute $u(q, \lambda, \omega, \mu, k)$ by deforming $\cC_t$ to $0$ and summing residues at $t = q^2 q^{-2kn}$.  As a power series in $q^{-2\omega}$, we obtain
\begin{align*} 
q^{\lambda\mu}\frac{u(q, \lambda, \omega, \mu, k)}{\theta_0(q^{2\mu + 2}; q^{-2k})}
&= \frac{q^{- \lambda - \mu - 2}}{\theta_0(q^{2\mu + 2}; q^{-2k})} \sum_{n \geq 0} \Res_{t = q^{2}q^{-2kn}} \Big(\frac{dt}{2\pi i t} \Omega_{q^2}(t; q^{-2\omega}, q^{-2k}) \frac{\theta_0(tq^{2\mu}; q^{-2k})}{\theta_0(tq^{-2}; q^{-2k})} \frac{\theta_0(tq^{2\lambda}; q^{-2\omega})}{\theta_0(tq^{-2}; q^{-2\omega})}\Big)\\
&= \frac{q^{- \lambda - \mu - 2}}{\theta_0(q^{2\mu + 2}; q^{-2k})} \sum_{n \geq 0} \frac{(q^{-2k(n+1)}; q^{-2\omega}, q^{-2k})(q^{-4}q^{-2\omega} q^{2k(n-1)}; q^{-2\omega}, q^{-2k})}{(q^4 q^{-2kn}; q^{-2\omega}, q^{-2k})(q^{-2\omega} q^{2k(n-1)}; q^{-2\omega}, q^{-2k})} \\
&\phantom{======}\frac{\theta_0(q^{2\mu + 2} q^{-2kn}; q^{-2k})}{(q^{-2k}; q^{-2k}) (q^{2k(n - 1)}; q^{-2k})_{n-1} (q^{-2kn}; q^{-2k})} \frac{\theta_0(q^{2\lambda + 2}q^{-2kn}; q^{-2\omega})}{(q^{-2\omega} q^{2kn}; q^{-2\omega})}\\
&=-q^{- \lambda - \mu - 2} \frac{(q^4;q^{-2\omega})(q^{-4}; q^{-2\omega}, q^{-2k})}{(q^{-4}; q^{-2k})(q^4; q^{-2\omega}, q^{-2k})} \frac{1}{(q^{-2k}; q^{-2k})(q^{-2\omega};q^{-2\omega})}\\
&\phantom{====} \sum_{n \geq 0} q^{-(2\mu + 2)n} \frac{\theta_0(q^{2\lambda + 2}q^{-2kn}; q^{-2\omega})}{\theta_0(q^4 q^{-2kn};q^{-2\omega})} \prod_{l = 1}^{n} \frac{\theta_0(q^4 q^{-2kl}; q^{-2\omega})}{\theta_0(q^{-2kl}; q^{-2\omega})}.
\end{align*}
On a formal neighborhood of $\infty$ in $q^{-2\mu}$, we may similarly compute $u(q, \lambda, \omega, \mu, k)$ by deforming $\cC_t$ to $\infty$ and summing residues at $t = q^{-2}q^{2kn}$, yielding
\begin{align*}
\frac{q^{\lambda\mu + \mu}u(q, \lambda, \omega, \mu, k)}{\theta_0(q^{2\mu - 2}; q^{-2k})} &= \frac{q^{- \lambda - 2}}{\theta_0(q^{2\mu - 2}; q^{-2k})} \sum_{n \geq 0} \Res|_{t = q^{-2} q^{2kn}}\Big(\frac{dt}{2\pi it} \Omega_{q^2}(t; q^{-2\omega}, q^{-2k}) \frac{\theta_0(tq^{2\mu}; q^{-2k})}{\theta_0(tq^{-2}; q^{-2k})} \frac{\theta_0(tq^{2\lambda}; q^{-2\omega})}{\theta_0(tq^{-2}; q^{-2\omega})}\Big)\\
&= \frac{q^{- \lambda - 2}}{\theta_0(q^{2\mu - 2}; q^{-2k})} \sum_{n \geq 0} \frac{\theta_0(q^{2\mu - 2} q^{2kn}; q^{-2k})}{\theta_0(q^{-4}q^{2kn}; q^{-2k})} \frac{\theta_0(q^{2\lambda - 2} q^{2kn}; q^{-2\omega})}{\theta_0(q^{-4} q^{2kn}; q^{-2\omega})} \\
&\phantom{=====}\frac{(q^{-4} q^{2kn}; q^{-2k}, q^{-2\omega})(q^{-2kn} q^{-2k}q^{-2\omega}; q^{-2k} q^{-2\omega})}{(q^{2kn} q^{-2\omega}; q^{-2k}, q^{-2\omega}) (q^{2kn}; q^{-2k})_n (q^{-2k}; q^{-2k}) (q^4 q^{-2kn} q^{-2k}q^{-2\omega}; q^{-2k}, q^{-2\omega})} \\
&= \frac{q^{- \lambda - 2}}{\theta_0(q^{2\mu - 2}; q^{-2k})} \frac{(q^{-4}q^{-2\omega}; q^{-2k}, q^{-2\omega})}{(q^4 q^{-2k} q^{-2\omega}; q^{-2k}, q^{-2\omega})} \frac{1}{(q^{-2k}; q^{-2k})(q^{-2\omega}; q^{-2\omega})}\frac{\theta_0(q^{2\mu - 2}; q^{-2k})}{(q^{4}q^{-2k}; q^{-2k})} \\
&\phantom{=====}\sum_{n \geq 0} q^{(2\mu + 2)n} \frac{\theta_0(q^{2\lambda - 2} q^{2kn}; q^{-2\omega})}{\theta_0(q^{-4}q^{2kn}; q^{-2\omega})} \prod_{l = 1}^n \frac{\theta_0(q^{-4}q^{2kl}; q^{-2\omega})}{\theta_0(q^{2kl}; q^{-2\omega})}. \qedhere
\end{align*}
\end{proof}

\subsection{Quasi-analytic continuation of the Felder-Varchenko function}

In what follows, we require the concept of quasi-analytic continuation of a formal power series in multiple variables.  Let $f^1(z, w_1, \ldots, w_m)$ and $f^2(z, w_1, \ldots, w_m)$ be formal power series of the form
\[
f^i(z, w_1, \ldots, w_m) = \sum_{n \geq 0} f_n^i(w_1, \ldots, w_m) z^n.
\]
Suppose that for $|z| < 1$, $f^1(z, w_1, \ldots, w_m)$ and $f^2(z, w_1, \ldots, w_m)$ converge on possibly disjoint regions $U_1$ and $U_2$ for $w = (w_1, \ldots, w_m)$ and that each $f_n^i(w_1, \ldots, w_m)$ admits analytic continuation to a rational function of $w$.  We say that $f^1$ is a \textit{quasi-analytic continuation} of $f^2$ if for each $n$ we have an equality of analytic continuations $f_n^1(w_1, \ldots, w_m) = f_n^2(w_1, \ldots, w_m)$. We denote this by $f^1(z, w_1, \ldots, w_m) \equiv f^2(z, w_1, \ldots, w_m)$.  As an example, we compute a quasi-analytic continuation we will use later.

\begin{lemma} \label{lem:qa-comps}
As quasi-analytic continuations in the formal variable $p$ from $|r| < 1$ to $|r| > 1$, we have 
\[
(p; r) \equiv (r^{-1}p; r^{-1})^{-1} \qquad \text{ and } \qquad (ap; p, r) \equiv (ar^{-1}p; p, r^{-1})^{-1}.
\]
\end{lemma}
\begin{proof}
For the first claim, notice that 
\[
(p; r) = \prod_{n \geq 0} (1 - p r^n) = \sum_{k \geq 0} \frac{(-1)^k p^k r^{\binom{k}{2}}}{(1 - r) \cdots (1 - r^k)} \equiv \sum_{k \geq 0} \frac{p^k r^{-k}}{(1 - r^{-1}) \cdots (1 - r^{-k})} = (r^{-1}p; r^{-1})^{-1}.
\]
For the second claim, by the first claim, we have
\[
(ap; p, r) = \prod_{n \geq 1} (ap^n; r) \equiv \prod_{n \geq 1} (a r^{-1} p^n; r^{-1})^{-1} = (a r^{-1}p; p, r^{-1})^{-1}. \qedhere
\]
\end{proof}

\begin{remark}
The domains $U_1$ and $U_2$ on which $f^1$ and $f^2$ converge in $w$ are allowed not only to be disjoint but also to be separated by a dense set of singularities, as occurs in Lemma \ref{lem:qa-comps}.  This is forbidden when analytic continuation is considered instead of quasi-analytic continuation.
\end{remark}

In what follows, we consider quasi-analytic continuation in the formal variable $q^{-2\omega}$.  We give a quasi-analytic continuation of a normalization of the Felder-Varchenko function to the region $|q| > 1, |q^{-2k}| > 1$.   Note that the integrand of $u(q, \lambda, \omega, \mu, k)$ does not admit quasi-analytic continuation in $q^{-2k}$ when $q^{-2\omega}$ is considered as a formal variable, but the normalized result of the integration does admit such a continuation.  Recall the terminating $q$-Pochhammer symbol $(u; q)_m$ of Subsection \ref{sec:theta-def}. 

\begin{lemma} \label{lem:fv-null}
For $|q| < 1$, $|q^{-2k}| < 1$, and $\frac{2(\mu + 1)}{k} > m$, we have the integral
\begin{align*}
I_m &= \oint_{|t| = 1} \frac{dt}{2\pi i t} \frac{(tq^{-2}; q^{-2k})}{(tq^2; q^{-2k})} \frac{\theta_0(t q^{2\mu}; q^{-2k})}{\theta_0(tq^{-2}; q^{-2k})} \frac{1 - t q^{2\lambda}}{1 - tq^{-2}} t^m\\
& = - q^{2m} \frac{(q^{2\mu + 2}; q^{-2k})(q^{-2\mu + 2 - 2k}; q^{-2k})}{(q^{-2k};q^{-2k})(q^4; q^{-2k})} \frac{(q^{-2\mu - 2 - 2k}; q^{-2k})_m}{(q^{-2\mu + 2 - 2k}; q^{-2k})_m} \frac{1 - q^{2\lambda + 2} - q^{-2\mu + 2 - 2km} + q^{2\lambda - 2\mu - 2km}}{1 - q^{-2\mu - 2 - 2km}}.
\end{align*}
\end{lemma}
\begin{proof}
By Corollary \ref{corr:theta-ratio}, we have for
\begin{equation} \label{eq:t-bounds}
\min_n \Big|\log |t q^{2\mu} q^{-2kn}| \Big| > \eps \text{ and } \min_n \Big|\log |t q^{-2} q^{-2kn}| \Big| > \eps
\end{equation}
that there is some $D_1(q, \eps)$ for which
\[
\left|\frac{\theta_0(t q^{2\mu}; q^{-2k})}{\theta_0(tq^{-2}; q^{-2k})}\right| \leq D_1(q, \eps) |q|^{\mu + 1} |t^2 q^{2\mu - 2}|^{\frac{\mu + 1}{k}} = D_1(q, \eps) |q|^{\mu + 1 + \frac{2(\mu^2 - 1)}{k}} |t|^{\frac{2(\mu + 1)}{k}}.
\]
Because all other terms converge to $1$ as $t \to 0$, for $\frac{2(\mu + 1)}{k} > m$ the integral vanishes after deforming the $t$-contour to a contour near $0$ which satisfies (\ref{eq:t-bounds}).  Therefore, the value of $I_m$ is the sum of residues near $0$, which occur at the poles $t = q^2 q^{-2nk}$ for $n \geq 0$.  

To compute the sum of residues, notice that 
\begin{align*}
I_m &= \oint_{|t| = 1} \frac{dt}{2\pi i t} \frac{(tq^{-2}; q^{-2k})}{(tq^2; q^{-2k})} \frac{\theta_0(t q^{2\mu}; q^{-2k})}{\theta_0(tq^{-2}; q^{-2k})} \frac{1 - t q^{2\lambda}}{1 - tq^{-2}} t^m\\
&= - q^2 \oint_{|t| = 1} \frac{dt}{2\pi i t}  \frac{\theta_0(t q^{2\mu}; q^{-2k})}{(tq^2; q^{-2k}) (t^{-1}q^{2}; q^{-2k})} (1 - t q^{2\lambda}) t^{m-1}\\
&= - q^2 \sum_{n \geq 0} \frac{\theta_0(q^{2\mu + 2} q^{-2kn}; q^{-2k})}{(q^4 q^{-2kn}; q^{-2k}) (q^{2kn}; q^{-2k})_n (q^{-2k}; q^{-2k})} (1 - q^{2\lambda + 2} q^{-2kn}) q^{2(m - 1)} q^{-2(m-1)kn} \\
&= - q^{2m} \frac{\theta_0(q^{2\mu + 2}; q^{-2k})}{(q^{-2k};q^{-2k})(q^4; q^{-2k})} \sum_{n \geq 0} \frac{(q^4; q^{-2k})_n}{(q^{-2k}; q^{-2k})_n}  (1 - q^{2\lambda + 2} q^{-2kn})q^{-(2\mu + 2)n}q^{-2kmn}\\
&= - q^{2m} \frac{\theta_0(q^{2\mu + 2}; q^{-2k})}{(q^{-2k};q^{-2k})(q^4; q^{-2k})} \Big(\frac{(q^{-2\mu + 2 - 2km}; q^{-2k})}{(q^{-2\mu - 2 - 2km}; q^{-2k})} - q^{2\lambda + 2} \frac{(q^{-2\mu + 2 - 2k(m + 1)}; q^{-2k})}{(q^{-2\mu - 2 - 2k(m + 1)}; q^{-2k})}\Big)\\
&=  - q^{2m} \frac{(q^{2\mu + 2}; q^{-2k})(q^{-2\mu + 2 - 2k}; q^{-2k})}{(q^{-2k};q^{-2k})(q^4; q^{-2k})} \frac{(q^{-2\mu - 2 - 2k}; q^{-2k})_m}{(q^{-2\mu + 2 - 2k}; q^{-2k})_m} \frac{1 - q^{2\lambda + 2} - q^{-2\mu + 2 - 2km} + q^{2\lambda - 2\mu - 2km}}{1 - q^{-2\mu - 2 - 2km}}. \qedhere
\end{align*}
\end{proof}

\begin{lemma} \label{lem:fv-null-flip}
For $|q| < 1$, $|q^{2k}| < 1$, and $\frac{2(\mu + 1)}{k} > m$, we have the integral
\begin{align*}
I_m' &= \oint_{\cC_t} \frac{dt}{2 \pi i t} \frac{(t q^2; q^{2k})}{(tq^{-2}; q^{2k})} \frac{\theta_0(tq^{-2\mu}; q^{2k})}{\theta_0(tq^2; q^{2k})} \frac{1 - tq^{2\lambda}}{1 - tq^2} t^m\\
&= - q^{-2m} \frac{(q^{-2\mu - 2}; q^{2k}) (q^{2\mu - 2 + 2k}; q^{2k})}{(q^{-4}; q^{2k})(q^{2k}; q^{2k})} \frac{(q^{2\mu + 2 + 2k}; q^{2k})_m}{(q^{2\mu - 2 + 2k}; q^{2k})_m} \frac{1 - q^{2\lambda - 2} - q^{2\mu - 2 + 2km} + q^{2\lambda + 2\mu + 2km}}{1 - q^{2\mu + 2 + 2km}},
\end{align*}
where $\cC_t$ is a contour containing the poles inside $|t| = 1$ except $t = q^2$ and excluding the poles outside $|t| = 1$ except $t = q^{-2}$.
\end{lemma}
\begin{proof}
By Corollary \ref{corr:theta-ratio}, we have for 
\begin{equation} \label{eq:t-bounds-flip}
\min_n \Big|\log|tq^{-2\mu} q^{2kn}|\Big| > \eps \text{ and } \min_n \Big|\log|tq^2q^{2kn}|\Big| > \eps
\end{equation}
that there is some $D_1(q, \eps)$ for which
\[
\left|\frac{\theta_0(tq^{-2\mu}; q^{2k})}{\theta_0(tq^2; q^{2k})}\right| \leq D_1(q, \eps) |q|^{-\mu - 1} |t^2 q^{-2\mu + 2}|^{\frac{2(\mu + 1)}{k}}.
\]
All other terms in the integrand converge to $1$ as $t \to 0$, so for $\frac{2(\mu + 1)}{k} > m$, the integral vanishes after deforming the $t$-contour to a contour near $0$ which satisfies (\ref{eq:t-bounds-flip}).  Therefore, the value of $I_m'$ is the sum of residues which are picked up when deforming the contour to $0$, which occur at the poles $t = q^{-2} q^{2kn}$ for $n \geq 0$.  We deform the contour and pick up residues to find
\begin{align*}
I_m' &= \oint_{\cC_t} \frac{dt}{2 \pi i t} \frac{(t q^2; q^{2k})}{(tq^{-2}; q^{2k})} \frac{\theta_0(tq^{-2\mu}; q^{2k})}{\theta_0(tq^2; q^{2k})} \frac{1 - tq^{2\lambda}}{1 - tq^2} t^m \\
&= -\oint_{\cC_t} \frac{dt}{2 \pi i t} \frac{\theta_0(t q^{-2\mu}; q^{2k})}{(tq^{-2}; q^{2k})(t^{-1} q^{-2}; q^{2k})} t^{m - 1} q^{-2} (1 - tq^{2\lambda})\\
&= - \sum_{n \geq 0} \frac{\theta_0(q^{-2\mu - 2} q^{2kn}; q^{2k})}{(q^{-4} q^{2kn}; q^{2k}) (q^{-2kn}; q^{2k})_n (q^{2k}; q^{2k})} q^{-2m} q^{2kn(m - 1)} (1 - q^{2\lambda - 2}q^{2kn})\\
&= - q^{-2m} \frac{\theta_0(q^{-2\mu - 2}; q^{2k})}{(q^{-4}; q^{2k})(q^{2k}; q^{2k})} \sum_{n \geq 0} \frac{(q^{-4}; q^{2k})_n}{(q^{2k}; q^{2k})_n} q^{(2\mu + 2)n + 2kmn} (1 - q^{2\lambda - 2} q^{2kn})\\
&= - q^{-2m} \frac{\theta_0(q^{-2\mu - 2}; q^{2k})}{(q^{-4}; q^{2k})(q^{2k}; q^{2k})} \Big(\frac{(q^{2\mu - 2 + 2km}; q^{2k})}{(q^{2\mu + 2 + 2km}; q^{2k})} - q^{2\lambda - 2} \frac{(q^{2\mu - 2 + 2k(m + 1)}; q^{2k})}{(q^{2\mu + 2 + 2k(m + 1)}; q^{2k})}\Big)\\
&= - q^{-2m} \frac{(q^{-2\mu - 2}; q^{2k})(q^{2\mu - 2 + 2k}; q^{2k})}{(q^{-4}; q^{2k})(q^{2k}; q^{2k})} \frac{(q^{2\mu + 2 + 2k}; q^{2k})_m}{(q^{2\mu - 2 + 2k}; q^{2k})_m} \frac{1 - q^{2\lambda - 2} - q^{2\mu - 2 + 2km} + q^{2\lambda + 2\mu + 2km}}{1 - q^{2\mu + 2 + 2km}}. \qedhere
\end{align*}
\end{proof}

\begin{remark}
In Lemma \ref{lem:fv-null}, the quantity
\[
\frac{(q^{-2k};q^{-2k})(q^4; q^{-2k})}{(q^{2\mu + 2}; q^{-2k})(q^{-2\mu + 2 - 2k}; q^{-2k})} I_m = - q^{2m} \frac{(q^{-2\mu - 2 - 2k}; q^{-2k})_m}{(q^{-2\mu + 2 - 2k}; q^{-2k})_m} \frac{1 - q^{2\lambda + 2} - q^{-2\mu + 2 - 2km} + q^{2\lambda - 2\mu - 2km}}{1 - q^{-2\mu - 2 - 2km}}
\]
is a rational function in all variables.  This suggests the correct function to quasi-analytically continue is
\[
\frac{(q^{-2k};q^{-2k})(q^4; q^{-2k})}{(q^{2\mu + 2}; q^{-2k})(q^{-2\mu + 2 - 2k}; q^{-2k})}u(q, \lambda, \omega, \mu, k).
\]
\end{remark}
Define the function
\begin{equation} \label{eq:f-def}
f(t, q, \lambda, \omega, k) =\frac{(tq^{-2} q^{-2\omega}; q^{-2\omega}, q^{-2k})(t^{-1} q^{-2} q^{-2\omega} q^{-2k}; q^{-2\omega}, q^{-2k})}{(tq^{2} q^{-2\omega}; q^{-2\omega}, q^{-2k})(t^{-1} q^2 q^{-2\omega} q^{-2k}; q^{-2\omega}, q^{-2k})} \frac{(tq^{2\lambda} q^{-2\omega}; q^{-2\omega})(t^{-1} q^{-2\lambda} q^{-2\omega}; q^{-2\omega})}{(tq^{-2} q^{-2\omega}; q^{-2\omega})(t^{-1} q^2 q^{-2\omega}; q^{-2\omega})}
\end{equation}
so that 
\[
\Omega_{q^2}(t; q^{-2\omega}, q^{-2k}) \frac{\theta_0(tq^{2\mu}; q^{-2k})}{\theta_0(tq^{-2}; q^{-2k})} \frac{\theta_0(t q^{2\lambda}; q^{-2\omega})}{\theta_0(t q^{-2}; q^{-2\omega})} = \frac{(tq^{-2}; q^{-2k})}{(tq^2; q^{-2k})} \frac{\theta_0(t q^{2\mu}; q^{-2k})}{\theta_0(tq^{-2}; q^{-2k})} \frac{1 - t q^{2\lambda}}{1 - tq^{-2}} f(t, q, \lambda, \omega, k).
\]
Note that $f(t, q, \lambda, \omega, k)$ admits a formal expansion
\[
f(t, q, \lambda, \omega, k) = \sum_{n \geq 0} \sum_{m = -n}^n q^{-2\omega n} t^m f_{n, m}(q, \lambda, k)
\]
for some rational functions $f_{n, m}(q, \lambda, k)$ of $q$, $q^{2\lambda}$, and  $q^{-2k}$ with $f_{0, 0}(q, \lambda, k) = 1$.  We now find a quasi-analytic continuation of $f(t, q, \lambda, \omega, k)$ and then use it in conjunction with our computations in Lemmas \ref{lem:fv-null} and \ref{lem:fv-null-flip} to obtain the desired continuation of the Felder-Varchenko function in Proposition \ref{prop:qa-fv-2}.

\begin{lemma} \label{lem:qa-phase}
The quasi-analytic continuation of $f(t, q, \lambda, \omega, k)$ to $|q^{-2k}| > 1$ is
\[
f(t, q, \lambda, \omega, k) \equiv \frac{(tq^{2} q^{-2\omega} q^{2k}; q^{-2\omega}, q^{2k})(t^{-1} q^2 q^{-2\omega}; q^{-2\omega}, q^{2k})}{(tq^{-2} q^{-2\omega} q^{2k}; q^{-2\omega}, q^{2k})(t^{-1} q^{-2} q^{-2\omega}; q^{-2\omega}, q^{2k})} \frac{(tq^{2\lambda} q^{-2\omega}; q^{-2\omega})(t^{-1} q^{-2\lambda} q^{-2\omega}; q^{-2\omega})}{(tq^{-2} q^{-2\omega}; q^{-2\omega})(t^{-1} q^2 q^{-2\omega}; q^{-2\omega})}.
\]
\end{lemma}
\begin{proof}
This follows by applying the computation of Lemma \ref{lem:qa-comps} repeatedly.
\end{proof}

\begin{prop} \label{prop:qa-fv}
The quasi-analytic continuation of
\[
\frac{(q^{-2k};q^{-2k})(q^4; q^{-2k})}{(q^{2\mu + 2}; q^{-2k})(q^{-2\mu + 2 - 2k}; q^{-2k})}u(q, \lambda, \omega, \mu, k)
\]
to $|q| < 1$, $|q^{-2\omega}| < 1$, and $|q^{-2k}| > 1$ is
\[
q^{-\lambda\mu - \lambda - \mu + 2} \frac{(q^{-4}; q^{2k})(q^{2k}; q^{2k})}{(q^{-2\mu - 2}; q^{2k})(q^{2\mu - 2 + 2k}; q^{2k})}\oint_{\cC_t} \frac{dt}{2 \pi i t} \Omega_{q^{-2}}(t; q^{-2\omega}, q^{2k}) \frac{\theta_0(tq^{-2\mu}; q^{2k})}{\theta_0(tq^2; q^{2k})} \frac{\theta_0(t q^{2\lambda}; q^{-2\omega})}{\theta_0(tq^{2}; q^{-2\omega})},
\]
where $\cC_t$ contains poles inside $|t| = 1$ except $t = q^2$ and excludes poles outside $|t| = 1$ except $t = q^{-2}$.
\end{prop}
\begin{proof}
Define $\wu(q, \lambda, \omega, \mu, k) = q^{\lambda \mu + \lambda + \mu + 2} u(q, \lambda, \omega, \mu, k)$.  By definition, we have
\[
\wu(q, \lambda, \omega, \mu, k) = \oint_{|t| = 1} \frac{dt}{2 \pi i t} \frac{(tq^{-2}; q^{-2k})}{(tq^2; q^{-2k})} \frac{\theta_0(t q^{2\mu}; q^{-2k})}{\theta_0(tq^{-2}; q^{-2k})} \frac{1 - t q^{2\lambda}}{1 - tq^{-2}} \sum_{n \geq 0} \sum_{m = -n}^n q^{-2n\omega} t^m f_{n, m}(q, \lambda, k).
\]
By Lemmas \ref{lem:fv-null} and \ref{lem:fv-null-flip}, we have on a formal neighborhood of $0$ in $q^{-2\mu}$ that 
\begin{align*}
&\frac{(q^{-2k};q^{-2k})(q^4; q^{-2k})}{(q^{2\mu + 2}; q^{-2k})(q^{-2\mu + 2 - 2k}; q^{-2k})} \wu(q, \lambda, \omega, \mu, k)\\
&\phantom{=}= - \sum_{n \geq 0}q^{-2n\omega} \sum_{m = -n}^n q^{2m} \frac{(q^{-2\mu - 2 - 2k}; q^{-2k})_m}{(q^{-2\mu + 2 - 2k}; q^{-2k})_m} \frac{1 - q^{2\lambda + 2} - q^{-2\mu + 2 - 2km} + q^{2\lambda - 2\mu - 2km}}{1 - q^{-2\mu - 2 - 2km}} f_{n, m}(q, \lambda, k)\\
&\phantom{=}\equiv q^4 \sum_{n \geq 0}q^{-2n\omega} \sum_{m = -n}^n q^{-2m} \frac{(q^{2\mu + 2 + 2k}; q^{2k})_m}{(q^{2\mu - 2 + 2k}; q^{2k})_m} \frac{q^{2\mu - 2 + 2km} - q^{2\lambda + 2\mu + 2km} - 1 + q^{2\lambda - 2}}{1 - q^{2\mu + 2 + 2km}} f_{n, m}(q, \lambda, k)\\
&\phantom{=}= q^4 \frac{(q^{-4}; q^{2k})(q^{2k}; q^{2k})}{(q^{-2\mu - 2}; q^{2k})(q^{2\mu - 2 + 2k}; q^{2k})}\sum_{n \geq 0} q^{-2\omega n} \sum_{m = -n}^n f_{n, m}(q, \lambda, k) I_m'\\
&\phantom{=}= q^4 \frac{(q^{-4}; q^{2k})(q^{2k}; q^{2k})}{(q^{-2\mu - 2}; q^{2k})(q^{2\mu - 2 + 2k}; q^{2k})}\sum_{n \geq 0} q^{-2\omega n}\!\!\! \sum_{m = -n}^n f_{n, m}(q, \lambda, k) \oint_{\cC_t} \frac{dt}{2 \pi i t} \frac{(t q^2; q^{2k})}{(tq^{-2}; q^{2k})} \frac{\theta_0(tq^{-2\mu}; q^{2k})}{\theta_0(tq^2; q^{2k})} \frac{1 - tq^{2\lambda}}{1 - tq^2} t^m\\
&\phantom{=}= q^4 \frac{(q^{-4}; q^{2k})(q^{2k}; q^{2k})}{(q^{-2\mu - 2}; q^{2k})(q^{2\mu - 2 + 2k}; q^{2k})}\oint_{\cC_t} \frac{dt}{2 \pi i t} \frac{(t q^2; q^{2k})}{(tq^{-2}; q^{2k})} \frac{\theta_0(tq^{-2\mu}; q^{2k})}{\theta_0(tq^2; q^{2k})} \frac{1 - tq^{2\lambda}}{1 - tq^2} f(t, q, \lambda, \omega, k).
\end{align*}
Thus, the coefficients of each formal expansion in $q^{-2\omega}$ have the same analytic continuations as formal series in $q^{-2\mu}$ and $q^{-2k}$ for the first and formal series in $q^{2\mu}$ and $q^{2k}$ for the second, so the above equality holds at the level of formal series in $q^{-2\omega}$ with coefficients which are rational functions in $q^{2\mu}$ and $q^{2k}$.  Substituting the quasi-analytic continuation of $f(t, q, \lambda, \omega, k)$ from Lemma \ref{lem:qa-phase}, we conclude that 
\begin{align*}
&\frac{(q^{-2k};q^{-2k})(q^4; q^{-2k})}{(q^{2\mu + 2}; q^{-2k})(q^{-2\mu + 2 - 2k}; q^{-2k})} \wu(q, \lambda, \omega, \mu, k)\\
&\phantom{==}\equiv q^4 \frac{(q^{-4}; q^{2k})(q^{2k}; q^{2k})}{(q^{-2\mu - 2}; q^{2k})(q^{2\mu - 2 + 2k}; q^{2k})}\oint_{\cC_t} \frac{dt}{2 \pi i t} \frac{(t q^2; q^{2k})}{(tq^{-2}; q^{2k})} \frac{\theta_0(tq^{-2\mu}; q^{2k})}{\theta_0(tq^2; q^{2k})} \frac{1 - tq^{2\lambda}}{1 - tq^2}\\
&\phantom{======} \frac{(tq^{2} q^{-2\omega} q^{2k}; q^{-2\omega}, q^{2k})(t^{-1} q^2 q^{-2\omega}; q^{-2\omega}, q^{2k})}{(tq^{-2} q^{-2\omega} q^{2k}; q^{-2\omega}, q^{2k})(t^{-1} q^{-2} q^{-2\omega}; q^{-2\omega}, q^{2k})} \frac{(tq^{2\lambda} q^{-2\omega}; q^{-2\omega})(t^{-1} q^{-2\lambda} q^{-2\omega}; q^{-2\omega})}{(tq^{-2} q^{-2\omega}; q^{-2\omega})(t^{-1} q^2 q^{-2\omega}; q^{-2\omega})}\\
&\phantom{==}= q^4 \frac{(q^{-4}; q^{2k})(q^{2k}; q^{2k})}{(q^{-2\mu - 2}; q^{2k})(q^{2\mu - 2 + 2k}; q^{2k})}\oint_{\cC_t} \frac{dt}{2 \pi i t} \Omega_{q^{-2}}(t; q^{-2\omega}, q^{2k}) \frac{\theta_0(tq^{-2\mu}; q^{2k})}{\theta_0(tq^2; q^{2k})} \frac{\theta_0(t q^{2\lambda}; q^{-2\omega})}{\theta_0(tq^{2}; q^{-2\omega})},
\end{align*}
which implies the desired.
\end{proof}

\begin{prop} \label{prop:qa-fv-2}
The quasi-analytic continuation of
\[
\frac{(q^{-2k};q^{-2k})(q^4; q^{-2k})}{(q^{2\mu + 2}; q^{-2k})(q^{-2\mu + 2 - 2k}; q^{-2k})}u(q, \lambda, \omega, \mu, k)
\]
to $|q| > 1$, $|q^{-2\omega}| < 1$, and $|q^{-2k}| > 1$ is
\[
q^{-\lambda\mu - \lambda - \mu + 2} \frac{(q^{-4}; q^{2k})(q^{2k}; q^{2k})}{(q^{-2\mu - 2}; q^{2k})(q^{2\mu - 2 + 2k}; q^{2k})}\oint_{|t| = 1} \frac{dt}{2 \pi i t} \Omega_{q^{-2}}(t; q^{-2\omega}, q^{2k}) \frac{\theta_0(tq^{-2\mu}; q^{2k})}{\theta_0(tq^2; q^{2k})} \frac{\theta_0(t q^{2\lambda}; q^{-2\omega})}{\theta_0(tq^{2}; q^{-2\omega})}.
\]
\end{prop}
\begin{proof}
If we view $q$, $q^{-2\omega}$, $q^{-2k}$, $q^{2\lambda}$, and $q^{2\mu}$ as algebraically independent variables, the integrand of Proposition \ref{prop:qa-fv} depends meromorphically on $q$ and $t$ and verifies the conditions of \cite[Theorem 10.2]{Rai}, hence our modification of the integration cycle yields the desired meromorphic extension of the quasi-analytic continuation of Proposition \ref{prop:qa-fv} to $|q| > 1$.
\end{proof}

\section{Free field realization and $q$-Wakimoto modules for $U_q(\asl_2)$} \label{sec:ff}

Our approach to computing the traces of intertwiners is to realize Verma modules for $U_q(\asl_2)$ as $q$-Wakimoto modules.  In the free-field realization of \cite{Mat}, we apply the method of coherent states to the expression for intertwiners given by \cite{Kon} to obtain contour integral formulas for the traces.  In this section we formulate the necessary vertex operators for this free field realization.

\subsection{Fock modules for $U_q(\asl_2)$}

Fix a level $k$.  For $\star \in \{\alpha, \ba, \beta\}$, define the Heisenberg algebras $H_\star$ to be generated by $\{\star_n \mid n \in \ZZ\}$ with $\star_0$ central and relations
\[
[\alpha_m, \alpha_n] = \delta_{n + m, 0} \frac{[2m][km]}{m}, \qquad [\ba_m, \ba_n] = - \delta_{n + m, 0} \frac{[2m][km]}{m}, \qquad [\beta_m, \beta_n] = \delta_{n + m, 0} \frac{[2m][(k + 2)m]}{m}.
\]
Each Heisenberg algebra is the direct sum of the subalgebras $H_{\star, n}$ generated by $\star_{-n}$ and $\star_n$ for $n \geq 0$.  Define the bosonic Fock space $\FF_{\star, a}$ to be the highest weight $H_\star$-module generated by a vector $v_{\star, a}$ so that
\[
\star_0 v_{\star, a} = 2a v_{\star, a} \text{ for $\star \in \{\alpha, \beta\}$} \qquad \text{ and } \qquad \ba_0 v_{\ba, a} = -2a v_{\ba, a}.
\]
We define also the action of $q^\delta$ so that 
\[
q^\delta \star_m q^{-\delta} = q^m \star_m \text{ for } \star \in \{\alpha, \ba, \beta\}
\]
and $q^\delta v_{\star, a} = v_{\star, a}$ for all $a$. Each bosonic Fock space admits the tensor decomposition
\[
\FF_{\star, a} = \bigotimes_{n \geq 0} \FF_{\star, a, n},
\]
where $\FF_{\star, a, 0} = \CC v_{\star, a}$ and $\FF_{\star, a, n}$ for $n > 0$ is the Fock space for $H_{\star, n}$.  Define tensor products of these Fock spaces by
\[
\FF_{a, b_1, b_2} := \FF_{\beta, a} \otimes \FF_{\alpha, b_1} \otimes \FF_{\ba, b_2} \qquad \text{ and } \qquad \FF_{a, b} := \FF_{a, b, b}.
\]

\subsection{$q$-Wakimoto modules for $U_q(\asl_2)$ and the free field construction}

The $q$-Wakimoto module $W_{\mu, k}$ of level $k$ is defined by
\[
W_{\mu, k} := \ker\Big(\eta_0 : \bigoplus_s \FF_{\mu, s}  \to \bigoplus_s \FF_{\mu, s + 1}\Big),
\]
where $\eta_0$ is the zero mode of the operator $\eta(w_0) = \sum_{m \in \ZZ} \eta_m w_0^{-m - 1}$ of Subsection \ref{sec:vert-op}. In \cite{Mat}, a free field representation of $U_q(\asl_2)$ on $W_{\mu, k}$ was given in terms of the vertex operators listed in Subsection \ref{sec:vert-op}, and $W_{\mu, k}$ was identified with a Verma module for $U_q(\asl_2)$ at generic $\mu$ and $k$.

\begin{prop}[{\cite[Proposition 3.1]{Mat}}] \label{prop:mat-ff}
There is a $U_q(\wsl_2)$-action $\pi'_\mu$ on $W_{\mu, k}$ given by
\begin{align*}
\pi'_\mu(e_1) &= x_0^+, \qquad \pi'_\mu(f_1) = x_0^-, \qquad \pi'_\mu(q^{h_1}) = q^{\alpha_0},\\
\pi'_\mu(e_0) &= x_1^{-} q^{-\alpha_0}, \qquad \pi'_\mu(f_0) = q^{\alpha_0} x_{-1}^+, \qquad \pi'_\mu(q^{h_0}) = q^k q^{-\alpha_0}, \qquad \pi'_\mu(q^d) = q^\delta.
\end{align*}
\end{prop}

\begin{prop}[{\cite[Corollary 4.4]{Mat}}] \label{prop:waki-verma}
For generic $\mu$ and $k$, we have an isomorphism of $U_q(\asl_2)$ modules
\[
W_{\mu, k} \simeq M_{2 \mu \rho + k \Lambda_0} = M_{2\mu, k}.
\]
\end{prop}

The following $q$-Sugawara construction for $L_0$ gives the action of $q^\delta$ on Fock space in terms of the free field construction.

\begin{prop}[{\cite[Equation 3.13]{Mat}}] \label{prop:q-sug}
For the operator $L_0$ defined by
\[
L_0 = \sum_{m > 0} \frac{m^2}{[2m][km]}(\alpha_{-m} \alpha_m - \bar{\alpha}_{-m} \bar{\alpha}_m) + \frac{1}{4k} (\alpha_0^2 - \bar{\alpha}_0^2) + \sum_{m > 0} \frac{m^2}{[2m][(k+2)m]} \beta_{-m} \beta_m + \frac{1}{4(k+2)}(\beta_0^2 + 2 \beta_0),
\]
we have $\delta + L_0 = \frac{\mu(\mu+1)}{k + 2}$ on $\FF_{\mu, m_1, m_2}$.
\end{prop}

Finally, to compute traces in the free field realization over $W_{\mu, k}$, we require the following reduction to a trace over the big Fock space.

\begin{prop}[{\cite[Section 6]{Kon2}}] \label{prop:ff-trace-comp}
For an operator $\psi: \bigoplus_s \FF_{\mu, s} \to \bigoplus_s \FF_{\mu, s}$ which preserves $W_{\mu, k}$, we have $\Tr|_{W_{\mu, k}}(\psi) = \Tr|_{\bigoplus_s \FF_{\mu, s}}(\cP \circ \psi)$, where $\cP$ is the projection onto $W_{\mu, k}$ defined by
\[
\cP := \oint \oint \frac{dw dz}{(2\pi i)^2} \frac{1}{z} \eta(w) \xi(z),
\]
where the contours are loops enclosing $w = 0$ and $z = 0$.
\end{prop}

\subsection{Vertex operators and Jackson integrals for intertwiners}

We use the vertex operator expression for intertwiners of \cite{Mat}, which we summarize here. For the map $\phi_j(z): \FF_{r, m_1, m_2} \to \FF_{r + j, m_1 + j, m_2 + j}$ defined by (\ref{eq:phi-def}), define maps $\phi_{j, m}(z)$ by $\phi_{j, j}(z) = \phi_j(z)$ and 
\[
\phi_{j, m - 1}(z) := \frac{1}{[j - m + 1]} [\phi_{j, m}(z), x_0^-]_{q^{2m}}.
\]
Let also $\Delta_\mu = \frac{\mu(\mu + 1)}{k + 2}$.  The Jackson integral of a function $f(t)$ for a period $p$ on the cycle $s$ is the formal sum
\[
\int_0^{s \cdot \infty} f(t) \frac{d_p t}{t} := \sum_{n \in \ZZ} f(s p^n).
\]
Recall that a quasi-meromorphic function is a function of the form $f(t) = t^a g(t)$ for some $a \in \CC$; if $g(t)$ is defined on an open domain and $\wtilde{g}(t)$ is a meromorphic continuation of $g(t)$, outside that domain, we say that $\wtilde{f}(t) = t^a \wtilde{g}(t)$ is the meromorphic continuation\footnote{We use the term meromorphic continuation instead of quasi-meromorphic continuation to avoid confusion with quasi-analytic continuation.} of $f(t)$.  We will sometimes consider the Jackson integral of a quasi-meromorphic function $f(t)$ which is defined only on a subset of the range of the Jackson cycle.  In these cases, we interpret this notation with $f(t)$ replaced by its meromorphic continuation to a branch of $t^a$ along the Jackson cycle.

\begin{prop}[{\cite[Theorem 5.4]{Mat}}] \label{prop:mat-ff-inter}
For $p = q^{2k + 4}$, any $\tau \leq \nu$, and Jackson cycles $s_1, \ldots, s_\tau$ for which matrix elements on $M_{2\mu, k}[-a\delta]$ of the Jackson integral expression
\[
\wPhi^{\mu + \nu - \tau}_{\mu, \nu}(z) := z^{\Delta_\mu + \Delta_\nu - \Delta_{\mu + \nu - \tau}} \int_{0}^{s_1 \cdot \infty} \cdots \int_0^{s_\tau \cdot \infty} \sum_{m \geq 0} S(t_1) \cdots S(t_\tau) \phi_{\nu, \nu - m}(z) \otimes w_{2m - 2\nu}\, d_pt_1 \cdots d_pt_\tau
\]
converge for $a \leq A$, they coincide with matrix elements of an intertwining operator $M_{2\mu, k} \to M_{2\mu + 2\nu - 2\tau, k} \hotimes L_{2\nu}(z)$.
\end{prop}

\begin{remark}
In \cite[Theorem 5.4]{Mat}, the content of Proposition \ref{prop:mat-ff-inter} is stated with $M_{-2\nu}^\vee(z)$ in place of $L_{2\nu}(z)$ and with $2\nu$ not an integer.  To obtain Proposition \ref{prop:mat-ff-inter}, note first that the proof of \cite{Mat} via \cite[Proposition 5.3]{Mat} applies for $M_{-2\nu}^\vee(z)$ for all values of $2\nu$, meaning that matrix elements of $\wPhi^{\mu + \nu - \tau}_{\mu, \nu}(z)$ coincide with matrix elements of $M_{2\mu, k} \to W_{2\mu + 2\nu - 2\tau, k} \hotimes M_{-2\nu}^\vee(z)$ when they converge.  The degree zero part therefore consists of matrix elements of an intertwiner
\[
\psi: M_{2\mu} \to M_{2\mu + 2\nu - 2\tau} \hotimes M_{-2\nu}^\vee
\]
of $U_q(\sl_2)$-modules.  Now, we may find another intertwiner 
\[
\psi': M_{2\mu} \to M_{2\mu + 2\nu - 2\tau} \hotimes L_{2\nu} \into M_{2\mu + 2\nu - 2\tau} \hotimes M_{-2\nu}^\vee
\]
with the same highest weight matrix element as $\psi$.  As intertwiners $M_{2\mu} \to M_{2\mu + 2\nu - 2\tau} \hotimes M_{-2\nu}^\vee$, we then have $\psi' = \psi$, hence matrix elements of $\wPhi^{\mu + \nu - \tau}_{\mu, \nu}(z)$ coincide with those of the intertwiner associated to $\psi'$ by \cite[Theorem 9.3.1]{EFK}.  In particular, it factors through the submodule $L_{2\nu}(z)$ of $M_{-2\nu}^\vee(z)$, yielding Proposition \ref{prop:mat-ff-inter} as stated here.
\end{remark}

\begin{remark}
In \cite[Theorem 5.4]{Mat}, it is shown that the operator $S(t_1) \cdots S(t_\tau) \phi_{\nu, \nu - m}(z)$ converges and commutes with the operators implementing the $U_q(\asl_2)$-action of Proposition \ref{prop:mat-ff} up to total $q^{2k + 4}$-difference in an open region for $t_1, \ldots, t_\tau$.  This implies that the same is true of its meromorphic continuation along the full Jackson cycle, allow us to interpret Proposition \ref{prop:mat-ff} with this meromorphic continuation.
\end{remark}

\subsection{Convergence and normalization of vertex operator expression for intertwiners}

We now show that the Jackson integral in the vertex operator expression for the intertwiner of Proposition \ref{prop:mat-ff-inter} converges in a certain region of parameters when $\nu = \tau = 1$ and the Jackson cycle is chosen to be $s_1 = zq^{-2}$.  Consider the parameter region
\begin{equation} \label{eq:me-conv-region}
0 < |q^{-2\mu}| \ll |q^{-2k}| \ll |q|, |q|^{-1}.
\end{equation}
For the rest of the paper, we use $\wPhi^\mu_{\mu, 1}(z)$ to denote the intertwiner of Proposition \ref{prop:mat-ff-inter} with this Jackson cycle. We first show in Proposition \ref{prop:inter-mat-elt} that the matrix element of the highest weight vector converges and compute its value.  For convenience, we use the notation $\kappa := k + 2$.

\begin{prop} \label{prop:inter-mat-elt}
In the region of parameters (\ref{eq:me-conv-region}), if $\nu = \tau = 1$ and $s_1 = zq^{-2}$, the diagonal matrix element of the Jackson integral for $\wPhi^\mu_{\mu, 1}(z)$ on the highest weight vector of $M_{2\mu, k}$ converges and equals
\[
C_{\mu, 1} := -(1 + q^2)q^{-2\mu-3} q^{\frac{4\mu + 4}{\kappa}}  \frac{(q^{-2\kappa};q^{-2\kappa})}{(q^{-4\mu-4}; q^{-2\kappa})} \frac{(q^{-4\mu};q^{-2\kappa})}{(q^4; q^{-2\kappa})}.
\]
\end{prop}
\begin{proof}
For $|t| > |z q^{4k}|$, we may choose contours $\cC_{w, c}$ around $w = 0$ so that for $w \in \cC_{w, c}$, we have $|t| > |wq^{b \kappa}|$ for all $c$, $|w| < |zq^k|, |z q^{k+4}|$ for $c = 1$, and $|w| > |zq^k|, |zq^{k+4}|$ for $c = -1$.  Define also a cycle $\cC_w = \cC_{w, -1} - \cC_{w, 1}$ around $zq^k$ and $zq^{k+4}$ but not $0$.  Applying Proposition \ref{prop:mat-ff} and substituting (\ref{eq:phi-x-comm}), we find for $p = q^{2k + 4}$ that
\begin{align*}
&C_{\mu, 1} = \left\langle v_{\mu, \mu, \mu}^*, \wPhi^\mu_{\mu, 1}(z) v_{\mu, \mu, \mu}\right\rangle \\
&= -\frac{z^{\Delta_1}}{(q - q^{-1})} \sum_{b, c \in \{\pm 1\}} \!\!\!(-1)^{\frac{2 - b - c}{2}} \!\!\!\int_0^{zq^{-2}\cdot \infty} \!\!\!\!\!\!\!\!d_pt \oint_{\cC_{w, c}} \frac{dw}{2\pi iw} q^{2b} \frac{w q^{-2c} - z q^{k + 2}}{w - zq^{k + 2 + 2b}} \left\langle v_{\mu, \mu, \mu}^*, S(t) :\phi_1(z) X_b^-(w): v_{\mu, \mu, \mu}\right\rangle\\
&= -\frac{z^{\Delta_1}}{(q - q^{-1})} \sum_{b \in \{\pm 1\}} (-1)^{\frac{1 - b}{2}} \int_0^{zq^{-2}\cdot \infty} \!\!\!\!\!\!d_pt \Big( - \oint_{\cC_w} \frac{dw}{2\pi iw} q^{2b} \frac{w q^{2} - z q^{k + 2}}{w - zq^{k + 2 + 2b}} \left\langle v_{\mu, \mu, \mu}^*, S(t) :\phi_1(z) X_b^-(w): v_{\mu, \mu, \mu}\right\rangle\\
&\phantom{===} + \oint_{\cC_{w, 1}} \frac{dw}{2\pi iw} q^{2b} \frac{w (q^{-2} - q^{2})}{w - zq^{k + 2 + 2b}} \left\langle v_{\mu, \mu, \mu}^*, S(t) :\phi_1(z) X_b^-(w): v_{\mu, \mu, \mu}\right\rangle\Big).
\end{align*}
Denote the first term by $C_{\mu, 1, 1}$ and the second by $C_{\mu, 1, 2}$.  Noting that $|t| > |zq^2|, |zq^{-2}|, |wq^{\kappa b}|$ for all $w \in \cC_w$, we have by the OPE's in Subsection \ref{sec:ope-comp} that 
\[
\left\langle v_{\mu, \mu, \mu}^*, S_a(t) :\phi_1(z) X_b^-(w): v_{\mu, \mu, \mu}\right\rangle = t^{-\frac{2}{\kappa}} \frac{(zt^{-1} q^{-2}; q^{-2\kappa})}{(zt^{-1}q^2; q^{-2\kappa})}  q^a \frac{t - w q^{(k + 1) b - a}}{t - w q^{\kappa b}} q^{2b} \frac{w q^{2} - z q^{\kappa}}{w - zq^{\kappa + 2b}} q^{2b\mu}z^{\frac{2\mu}{k + 2}} t^{-\frac{2\mu}{\kappa}}.
\]
Applying meromorphic continuation to all $t$ on the Jackson cycle, we may substitute and take the poles at $w = zq^{\kappa + 2b}$ to find that 
\begin{align*}
C_{\mu, 1, 1} &= -\frac{z^{\frac{2\mu + 2}{\kappa}}}{(q - q^{-1})^2} \sum_{a, b \in \{\pm 1\}} (-1)^{\frac{2 - a - b}{2}} \!\!\int_0^{zq^{-2}\cdot \infty}\! \frac{d_pt}{t} \oint_{\cC_w}\frac{dw}{2\pi i w} t^{-\frac{2\mu + 2}{\kappa}} q^{a + 2b + 2 b\mu}\\
&\phantom{=====================}\frac{(zt^{-1} q^{-2}; q^{-2\kappa})}{(zt^{-1}q^2; q^{-2\kappa})} \frac{t - w q^{(k + 1) b - a}}{t - w q^{\kappa b}} \frac{w q^{2} - z q^{\kappa}}{w - zq^{\kappa + 2b}}\\
&= -\frac{z^{\frac{2\mu + 2}{\kappa}}}{q - q^{-1}} \sum_{b \in \{\pm 1\}} (-1)^{\frac{1 - b}{2}} \int_0^{zq^{-2}\cdot \infty} \frac{d_pt}{t} t^{-\frac{2\mu + 2}{\kappa}} \frac{(zt^{-1} q^{-2}; q^{-2\kappa})}{(zt^{-1}q^2; q^{-2\kappa})}\frac{t q^{2b + 2 b\mu}}{t - z q^{\kappa + (\kappa + 2)b}}(q^{2} - q^{-2b})\\
&= -(q + q^{-1})q^{2\mu + 2} z^{\frac{2\mu + 2}{\kappa}} \int_0^{zq^{-2}\cdot \infty} \frac{d_pt}{t} t^{-\frac{2\mu + 2}{\kappa}} \frac{(zt^{-1} q^{- 2}; q^{-2\kappa})}{(zt^{-1}q^{2\kappa + 2}; q^{-2\kappa})}\\
&= -(q + q^{-1})q^{2\mu + 2} q^{\frac{4\mu + 4}{\kappa}} \sum_{n \geq 1} q^{(-4\mu - 4)n} \frac{(q^{-2\kappa n}; q^{-2\kappa})}{(q^4 q^{-2\kappa n} q^{2\kappa}; q^{-2\kappa})} \\ \nonumber
 &= -(q + q^{-1})q^{-2\mu - 2} q^{\frac{4\mu + 4}{\kappa}} \frac{(q^{-2\kappa};q^{-2\kappa})}{(q^4; q^{-2\kappa})} \sum_{n \geq 0} q^{(-4\mu - 4)n} \frac{(q^4; q^{-2\kappa})_n}{(q^{-2\kappa}; q^{-2\kappa})_n}\\ 
&= -(1 + q^2)q^{-2\mu-3} q^{\frac{4\mu + 4}{\kappa}}  \frac{(q^{-2\kappa};q^{-2\kappa})}{(q^{-4\mu-4}; q^{-2\kappa})} \frac{(q^{-4\mu};q^{-2\kappa})}{(q^4; q^{-2\kappa})},
\end{align*}
where convergence holds because we are in the region (\ref{eq:me-conv-region}).  For $C_{\mu, 1, 2}$, the contour $\cC_{w, 1}$ contains only poles at $w = 0$, which vanish by the manipulations
\begin{align*}
C_{\mu, 1, 2} &= -\frac{z^{\frac{2\mu + 2}{\kappa}}}{(q - q^{-1})^2} \sum_{a, b \in \{\pm 1\}} (-1)^{\frac{2 - a - b}{2}}\! \int_0^{zq^{-2}\cdot \infty} \frac{d_pt}{t} \oint_{\cC_{w, 1}} \frac{dw}{2\pi i} t^{-\frac{2\mu + 2}{\kappa}} q^{a + 2b + 2 b\mu}\\
&\phantom{==========================}\frac{(zt^{-1} q^{-2}; q^{-2\kappa})}{(zt^{-1}q^2; q^{-2\kappa})} \frac{t - w q^{(k + 1) b - a}}{t - w q^{\kappa b}} \frac{(q^{-2} - q^2)}{w - zq^{\kappa + 2b}}\\
&= (q + q^{-1})z^{\frac{2\mu + 2}{\kappa}} \sum_{b \in \{\pm 1\}} (-1)^{\frac{1 - b}{2}}\! \int_0^{zq^{-2}\cdot \infty}\frac{d_pt}{t} \oint_{\cC_{w, 1}} \frac{dw}{2\pi i} \frac{(zt^{-1} q^{-2}; q^{-2\kappa})}{(zt^{-1}q^2; q^{-2\kappa})} \frac{t}{t - w q^{\kappa b}} \frac{t^{-\frac{2\mu + 2}{\kappa}} q^{2b + 2 b\mu}}{w - zq^{\kappa + 2b}}\\
&= 0.
\end{align*}
The result follows because $C_{\mu, 1} = C_{\mu, 1, 1} + C_{\mu, 1, 2} = C_{\mu, 1, 1}$.
\end{proof}

We would like now to relate $\wPhi^{\mu}_{\mu, 1}(z)$ to a multiple of $\Phi^{w_0}_{2\mu, k}(z)$.  Because the space of intertwiners has dimension $1$, the constant of proportionality is constrained to be $C_{\mu, 1}$.  Unfortunately, it is not true that all matrix elements of $\wPhi^{\mu}_{\mu, 1}(z)$ converge simultaneously in any open neighborhood of $0$ in $q^{-2\mu}$.  Instead, we show in Proposition \ref{prop:jack-conv} that in each degree, matrix elements of the the Jackson integral converge on a open neighborhood of $0$ in $q^{-2\mu}$ dependent on degree and coincide with matrix elements of the intertwiner.  This will suffice for our later computations.

\begin{prop} \label{prop:jack-conv}
If $\nu = \tau = 1$ and $s_1 = zq^{-2}$, for each $A \geq 0$, there exists an open neighborhood of $0$ in $q^{-2\mu}$ so that in the region of parameters (\ref{eq:me-conv-region}), matrix elements of the operator $\wPhi^{\mu}_{\mu, 1}(z)$ on the space 
\[
\bigoplus_{a \leq A} M_{\mu, k}[-a \delta]
\]
of vectors of degree at least $-A$ converge and equal matrix elements of $C_{\mu, 1} \Phi^{w_0}_{2\mu, k}(z)$.
\end{prop}
\begin{proof}
By Propositions \ref{prop:mat-ff} and \ref{prop:waki-verma}, every matrix element is a finite linear combination of expressions of the form
\begin{equation} \label{eq:mat-elt-form}
\int_0^{z q^{-2} \cdot \infty} \!\!\frac{d_pt}{t} \langle v_{\mu, \mu, \mu}^*, \prod_{i = 1}^m x^{c_i}_{a_i} S(t) [\phi_1(z), x^-_0]_{q^2} \prod_{j = 1}^l x^{d_j}_{b_j} v_{\mu, \mu, \mu}\rangle.
\end{equation}
for some choice of $m$, $l$, and $a_i, b_j, c_i, d_j$.  Notice that
\begin{align*}
S(t) [\phi_1(z), x^-_0]_{q^2} &= S(t) \phi_1(z) x^-_0 - q^2 S(t) x_0^- \phi_1(z)\\
&= S(t) \phi_1(z) x^-_0 - q^2 x_0^- S(t)\phi_1(z) - q^2 [x_0^-, S(t)] \phi_1(z).
\end{align*}
By the OPE's of Subsection \ref{sec:ope-comp}, we have
\[
S(t) \phi_1(z) = t^{-\frac{2}{\kappa}} \frac{(zt^{-1} q^{-2}; q^{-2\kappa})}{(zt^{-1}q^2; q^{-2\kappa})} : S(t) \phi_1(z): \qquad |t| > |zq^2|, |zq^{-2}|
\]
and by Lemma \ref{lem:x0comm}, we have 
\begin{multline*}
:[x_0^-, S(t)] \phi_1(z): = - \frac{1}{(q - q^{-1})t} \Big(t^{-\frac{2}{\kappa}}\frac{(\frac{z}{t}q^{-2} q^{2\kappa}; q^{-2\kappa})}{(\frac{z}{t} q^2 q^{2\kappa}; q^{-2\kappa})} q^2  : U(t) Y^-(tq^{-\kappa}) W_+(t q^{-\frac{k}{2} - 2})^{-1} W_+(t q^{-\frac{k}{2}})^{-1}\phi_1(z): \\
- t^{-\frac{2}{\kappa}}\frac{(\frac{z}{t}q^{-2}; q^{-2\kappa})}{(\frac{z}{t} q^2; q^{-2\kappa})} q^{-2} : U(t) Y^-(tq^{\kappa}) W_-(t q^{\frac{k}{2} + 2})^{-1} W_-(t q^{\frac{k}{2}})^{-1}\phi_1(z):\Big).
\end{multline*}
These computations imply that $\prod_{i = 1}^m x^{c_i}_{a_i} S(t) [\phi_1(z), x^-_0]_{q^2} \prod_{j = 1}^l x^{d_j}_{b_j}$ is a linear combination of expressions of the form 
\[
t^{-\frac{2}{\kappa}}\frac{(\frac{z}{t}q^{-2}q^{2\kappa s}; q^{-2\kappa})}{(\frac{z}{t} q^2q^{2\kappa s}; q^{-2\kappa})} \prod_{i = 1}^m x^{c_i}_{a_i} :U(t) F(t)\phi_1(z): \prod_{j = 1}^l x^{d_j}_{b_j}
\]
for $s \geq 0$ and $F(t)$ a $q$-vertex operator whose degree $0$ term does not contain $t^{\beta_0}$, $t^{\alpha_0}$, or $t^{\ba_0}$.  Observe now that
\[
\langle v_{\mu, \mu, \mu}^*, t^{-\frac{2}{\kappa}}\frac{(\frac{z}{t}q^{-2}q^{2\kappa s}; q^{-2\kappa})}{(\frac{z}{t} q^2q^{2\kappa s}; q^{-2\kappa})} \prod_{i = 1}^m x^{c_i}_{a_i} :U(t) F(t)\phi_1(z): \prod_{j = 1}^l x^{d_j}_{b_j} v_{\mu, \mu, \mu} \rangle = t^{-\frac{2\mu + 2}{\kappa}} \frac{(\frac{z}{t}q^{-2}q^{2\kappa s}; q^{-2\kappa})}{(\frac{z}{t} q^2q^{2\kappa s}; q^{-2\kappa})} g(t)
\]
for some Laurent polynomial $g(t)$.  The Jackson integral of this expression converges on a neighborhood of $0$ in $q^{-2\mu}$ depending on the maximum degree of $g(t)$.  By Proposition \ref{prop:mat-ff}, matrix elements in degree at least $-A$ are computed with $\prod_{j = 1}^l x^{d_j}_{b_j} v_{\mu, \mu, \mu}$ of degree at most $A + 1$ in $\FF_\mu$, meaning that in degree $A$, we have $\deg g \leq A + 1$.  This implies that $\deg g(t)$ is bounded for matrix elements of degree bounded below, hence the neighborhood of $0$ for $q^{-2\mu}$ may be chosen uniformly in degree.

Because the space of intertwiners is one-dimensional, the matrix elements of $\Phi^{w_0}_{2\mu, k}(z)$ are determined as the unique solutions to a system of linear equations expressing the intertwining relations.  Therefore, if matrix elements of $C_{\mu, 1}^{-1} \wPhi^\mu_{\mu, 1}(z)$ of at least a fixed degree converge in some region, they coincide with the corresponding matrix elements of $\Phi^{w_0}_{2\mu, k}(z)$, giving the claim.
\end{proof}

\section{Contour integral formula for the trace in the three-dimensional representation} \label{sec:ci}

In this section, we combine the tools we have assembled to give a contour integral formula for the trace function in Theorem \ref{thm:int-trace}.  First, we use the free field realization of Section \ref{sec:ff} to represent the trace function as a Jackson integral of an iterated contour integral via the method of coherent states.  We then simplify the integral and identify the formal expansion of the Jackson integral with the expansion of a renormalization of the Felder-Varchenko function computed in Section \ref{sec:fv-fn}.  The contour integral formula for the Felder-Varchenko function then yields convergence of the trace function and our desired integral formula.

\subsection{Statement of the result}

We will compute the trace $T^{w_0}(q, \lambda, \omega, \mu, k)$ of (\ref{eq:sym-tr-def}) when $V = L_2(z)$ is the irreducible three-dimensional evaluation representation of $U_q(\asl_2)$.  Recall it is defined by 
\[
T^{w_0}(q, \lambda, \omega, \mu, k) := \Tr|_{M_{(\mu - 1)\rho + (k-2)\Lambda_0}}\Big(\Phi_{\mu - 1, k - 2}^{w_0}(z) q^{2\lambda\rho + 2\omega d}\Big).
\]
Define the \textit{good region of parameters} to be the region with
\begin{equation} \label{eq:good-region}
0 < |q^{-2\omega}| \ll |q^{-2\mu}| \ll |q^{-2\lambda}| \ll |q^{-2k}| \ll |q|, |q|^{-1}.
\end{equation}
Notice that this region includes the region (\ref{eq:me-conv-region}) on which we have shown convergence of the vertex operator expression for the intertwiner.  Our main result is the following computation of the trace function.  

\begin{thm} \label{thm:int-trace}
For $q^{-2\mu}$ and then $q^{-2\omega}$ sufficiently close to $0$ in the good region of parameters (\ref{eq:good-region}), the trace function converges and has value
\begin{align*}
T^{w_0}(q, \lambda, \omega, \mu, k) &= \frac{q^{\lambda\mu - \lambda + 2}(q^{-4}; q^{-2\omega})}{\theta_0(q^{2\lambda}; q^{-2\omega})(q^{2\lambda - 2} q^{-2\omega}; q^{-2\omega})(q^{-2\lambda - 2}; q^{-2\omega})} \frac{ (q^{-2k};q^{-2k})(q^4 q^{-2k}; q^{-2k})}{(q^{-2\mu + 2};q^{-2k})(q^{2\mu + 2} q^{-2k}; q^{-2k})}\\
&\phantom{==}\frac{(q^{-2 \omega + 2}; q^{-2\omega}, q^{-2 k})^2}{(q^{-2 \omega - 2}; q^{-2 \omega}, q^{-2k})^2} \oint_{\cC_t} \frac{dt}{2\pi it} \Omega_{q^2}(t; q^{-2\omega}, q^{-2k}) \frac{\theta_0(tq^{-2\mu}; q^{-2k})}{\theta_0(tq^{-2}; q^{-2k})} \frac{\theta_0(tq^{2\lambda}; q^{-2\omega})}{\theta_0(tq^{-2}; q^{-2\omega})},
\end{align*}
where the integration cycle $\cC_t$ is the unit circle.
\end{thm}

\subsection{Structure of the proof of Theorem \ref{thm:int-trace}}

We now outline the proof of Theorem \ref{thm:int-trace}, which occupies the remainder of Section \ref{sec:ci}.  We will compute formally the trace of the free field intertwiner $\wPhi^\mu_{\mu, 1}(z)$ in two steps.  First, we compute it as a formal series in $q^{-2\omega}$ and then $q^{-2\mu}$ in Proposition \ref{prop:ff-comp}, for which we work in the \textit{doubly formal good region} of parameters 
\begin{equation} \label{eq:dformal-good-region}
0 < |q^{-2\lambda}| \ll |q^{-2k}| \ll |q|, |q|^{-1}.
\end{equation}
This computation proceeds through the method of coherent states applied to the free field construction summarized in Section \ref{sec:ff} and some intricate contour integral manipulations which hold for parameter values in (\ref{eq:dformal-good-region}).  We then show in Proposition \ref{prop:int-formal} that the result of Proposition \ref{prop:ff-comp} converges as a formal series in $q^{-2\omega}$ and then $q^{-2\mu}$.  The proof of Theorem \ref{thm:int-trace} follows by matching this expansion with that of the Felder-Varchenko function in Proposition \ref{prop:fv-series} and then applying Proposition \ref{prop:rat-fn-trace} and the integral expression for the Felder-Varchenko function to show that the resulting equality holds for small numerical values of $q^{-2\mu}$ and $q^{-2\omega}$ in the good region (\ref{eq:good-region}).

Define the trace of the Jackson integral by
\begin{equation} \label{eq:un-shifted-tr}
\Xi(q, \lambda, \omega, \mu, k) := \Tr|_{M_{2\mu \rho + k\Lambda_0}}\Big(\wPhi^\mu_{\mu, 1}(z) q^{2\lambda\rho + 2\omega d}\Big),
\end{equation}
where the quantity on the right does not depend on $z$.  It is related to the trace function in the doubly formal good region as follows.

\begin{lemma}\label{lem:trace-relation}
In the doubly formal region, as formal power series in $q^{-2\omega}$ and then $q^{-2\mu}$, we have the equality
\[
T^{w_0}(q, \lambda, \omega, \mu, k) = C_{\frac{\mu - 1}{2}, 1}^{-1}\, \Xi\Big(q, \lambda, \omega, \frac{\mu - 1}{2}, k - 2\Big),
\]
where $C_{\mu, 1}$ is the normalizing constant of Proposition \ref{prop:inter-mat-elt}.
\end{lemma}
\begin{proof}
This follows from the identification $M_{2\mu \rho + k\Lambda_0} = W_{\mu, k}$ between the $q$-Wakimoto module and the Verma module and the fact that Proposition \ref{prop:jack-conv} shows $\wPhi^\mu_{\mu, 1}(z) = C_{\mu, 1} \Phi^{w_0}_{2\mu, k}(z)$ as formal series in $q^{-2\omega}$ and then $q^{-2\mu}$ in the doubly formal good region.
\end{proof}

We now state and give integration cycles for Proposition \ref{prop:ff-comp}, which gives a formal computation of the trace function.  Its proof occupies Subsections \ref{sec:int-cycle} to \ref{sec:sub-int}.  Recall the notation $\kappa = k + 2$.

\begin{prop} \label{prop:ff-comp}
In the doubly formal good region of parameters (\ref{eq:dformal-good-region}), as a formal series in $q^{-2\omega}$, the trace function $\Xi(q, \lambda, \omega, \mu, k)$ has formal Jackson integral expansion
\[
\Xi(q, \lambda, \omega, \mu, k) = C(\lambda, \mu) \int_0^{q^{-2} \cdot \infty} \frac{d_pt}{t} t^{- \frac{2(\mu + 1)}{\kappa}}  \Omega_{q^2}(t; q^{-2\omega}, q^{-2\kappa}) \frac{\theta_0(tq^{2}; q^{-2\kappa})\theta_0(tq^{2\lambda}; q^{-2\omega})}{\theta_0(tq^{-2};q^{-2\kappa}) \theta_0(tq^{-2}; q^{-2\omega})}
\]
for $p = q^{2\kappa}$ and 
\[
C(\lambda, \mu) = \frac{1}{q - q^{-1}} \frac{q^{2\lambda\mu - 2\mu - 2}(q^{-4}; q^{-2\omega})}{\theta_0(q^{2\lambda}; q^{-2\omega})(q^{2\lambda - 2} q^{-2\omega}; q^{-2\omega})(q^{-2\lambda - 2}; q^{-2\omega})}\frac{(q^{-2 \omega + 2}; q^{-2\omega}, q^{-2\kappa})^2}{(q^{-2 \omega - 2}; q^{-2 \omega}, q^{-2\kappa})^2}.
\]
\end{prop}

\begin{remark}
Although the Jackson integral involves summation of the integrand at values of $t$ where the one loop correlation functions from the method of coherent states may not converge, we may replace them by their meromorphic continuations in these regions.
\end{remark}

\subsection{Definition of integration cycles} \label{sec:int-cycle}

We now define integration cycles for parameters in the doubly formal good region of parameters (\ref{eq:dformal-good-region}).  First, for numerical $q^{-2\omega}$, define the \textit{good spectral region} $\cS_c$ by 
\begin{equation} \label{eq:spec1}
|w_0| \gg |z_0| \gg |tq^{k}|, |tq^{-k}| \gg |zq^k|, |zq^{k+4}|, |w| > |z| \gg |w_0 q^{-2\omega}|
\end{equation}
and
\begin{equation} \label{eq:spec3}
\begin{cases} |w| < |zq^{k}|, |zq^{k+4}| & c = 1 \\ |w| > |zq^k|, |zq^{k+4}| & c = -1 \end{cases}.
\end{equation}
By $a \gg b$, we mean that $\frac{a}{b} > |q^{100 \lambda}|, |q^{100k}|, |q^{100}|$ so that for each instance of $q^{C\lambda}$, $q^{Ck}$, or $q^C$ which appears in our formulas, we have $\frac{a}{b} > |q^{C\lambda}|, |q^{Ck}|, |q^C|$.  Notice that if $(t, z_0, w_0, w)$ lies in the good spectral region, then all requirements for convergence of one loop correlation functions in Table \ref{tab:two-point} are satisfied.  Now, define the \textit{formal good spectral region} $\wS_c$ to be the region of spectral parameters satisfying (\ref{eq:spec3}) and 
\begin{equation} \label{eq:fspec1}
|w_0| \gg |z_0| \gg |tq^{k}|, |tq^{-k}| \gg |zq^k|, |zq^{k+4}|, |w| > |z|,
\end{equation}
where we use $\gg$ in the same way as in (\ref{eq:spec1}). If $(t, z_0, w_0, w)$ lies in the formal good spectral region $\wS_c$, then for $q^{-2\omega}$ sufficiently close to $0$ it lies in the good spectral region $\cS_c$.  Therefore, on a formal neighborhood of $0$ in $q^{-2\omega}$, if the spectral parameters lie in the formal good spectral region, then the requirements for convergence of one loop correlation functions in Table \ref{tab:two-point} are satisfied.

\begin{lemma} \label{lem:good-ineq}
In the formal good spectral region and the doubly formal good region of parameters, we have
\[
|z_0| > |wq^{b(k + 1)}|, |tq^{-a}| \qquad \text{ and } \qquad |w_0 q^{-2\lambda - b + a}| > |wq^{b(k + 1)}|.
\]
\end{lemma}
\begin{proof}
This follows from (\ref{eq:fspec1}) and the definition of $\gg$.
\end{proof}

\begin{lemma} \label{lem:cont-choice}
There is some $R > 0$, so that for $t$ lying in the open region $\cC_t = \{|t| > R|z|\}$ and any $q, q^{-2\lambda}, q^{-2k}$ in the doubly formal region of parameters (\ref{eq:dformal-good-region}), we may choose circular contours $\cC_{w_0}$, $\cC_{z_0}$, and $\cC_{w, c}$ enclosing $0$ so that for $z_0 \in \cC_{z_0}$, $w_0 \in \cC_{w_0}$, and $w \in \cC_{w, c}$, the spectral parameters $(t, z_0, w_0, w)$ lie in the formal good spectral region.
\end{lemma}
\begin{proof}
We take the contours to be nested circles around the origin with radii obeying (\ref{eq:fspec1}).  The only obstacle is that the relations $|tq^k|, |tq^{-k}| \gg |w|$ and $|w| > |zq^k|, |zq^{k+4}|$ must simultaneously hold on $\cC_{w, -1}$, which is possible for $|t/z|$ sufficiently large, as needed.
\end{proof}

\begin{remark}
For the rest of the proof of Proposition \ref{prop:ff-comp}, we will compute the integrand of the Jackson integral by meromorphic continuation from the region $\cC_t$. 
\end{remark}

\subsection{Reducing to a trace over Fock space}

Assume now that all parameters lie in the doubly formal good region of parameters. For $t$ in the region $\cC_t$ of Lemma \ref{lem:cont-choice}, fix spectral variables on the cycles $\cC_{w_0}$, $\cC_{z_0}$, and $\cC_{w, c}$ of Lemma \ref{lem:cont-choice} so that they lie in the formal good spectral region.  By Propositions \ref{prop:waki-verma} and \ref{prop:ff-trace-comp}, we have
\begin{align*}
\Xi(q, \lambda, \omega, \mu, k) &= \sum_{s\in \mu + \ZZ} \Tr|_{\FF_{\mu, s}} \Big(\oint_{\cC_{w_0}} \oint_{\cC_{z_0}} \frac{dw_0 dz_0}{(2 \pi i)^2} \frac{1}{z_0} \eta(w_0) \xi(z_0) \wPhi^\mu_{\mu, 1}(z)q^{2\lambda\rho + 2 \omega d}\Big) \\
&= \sum_{s \in \mu + \ZZ} \oint_{\cC_{w_0}} \oint_{\cC_{z_0}} \frac{dw_0 dz_0}{(2 \pi i)^2} \frac{1}{z_0} \Tr|_{\FF_{\mu, s}}\Big(\eta(w_0) \xi(z_0) \wPhi^\mu_{\mu, 1}(z)q^{2\lambda\rho + 2\omega d}\Big)
\end{align*}
since cycles $\cC_{w_0}$ and $\cC_{z_0}$ enclose $w_0 = 0$ and $z_0 = 0$.

\begin{remark}
The interchange of sum and integral is valid when
\[
\Tr|_{\FF_{\mu, s}}\Big(\eta(w_0) \xi(z_0) \wPhi^\mu_{\mu, 1}(z)q^{2\lambda \rho+ 2\omega d}\Big)
\]
is convergent, which holds for our choice of contours.
\end{remark}

Recall that we identified $\CC \cdot w_0 \simeq \CC$.  Using $S_{\pm 1}$ and $X^-_{\pm 1}$ to denote $S_{\pm}$ and $X^-_{\pm}$, we obtain by the free field realization of Proposition \ref{prop:mat-ff} that for $p = q^{2\kappa}$, we have
\begin{multline*}
\Tr|_{\FF_{\mu, s}}\Big(\eta(w_0) \xi(z_0) \wPhi^\mu_{\mu, 1}(z)q^{2\lambda\rho + 2\omega d}\Big) = \frac{z^{\Delta_1}}{(q - q^{-1})^2z} \sum_{a, b \in \{\pm 1\}} (-1)^{\frac{2 - a - b}{2}}\\
\int_0^{z q^{-2} \cdot \infty} \frac{d_pt}{t} \oint_{\cC_{w, c}} \frac{dw}{2\pi i w} \Tr|_{\FF_{\mu, s}}\Big(\eta(w_0)\xi(z_0) S_a(t) [\phi_1(z), X^-_b(w)]_{q^2}q^{2\lambda\rho + 2\omega d}\Big),
\end{multline*}
where $\cC_{w, 1}$ applies for $\phi_1(z) X^-_b(w)$ and $\cC_{w, -1}$ applies for $X^-_b(w) \phi_1(z)$.  Putting these observations together and noting that (\ref{eq:phi-x-comm}) applies for parameters on the chosen contours, we obtain
\begin{align*}
\Xi(q, \lambda, \omega, \mu, k) &= \frac{z^{\Delta_1}}{(q - q^{-1})^2} \sum_{a, b, c \in \{\pm1\}} (-1)^{\frac{3 - a  - b - c}{2}} \sum_{s \in \mu + \ZZ} \frac{d_pt}{t}\oint_{\cC_{w, c}} \oint_{\cC_{w_0}} \oint_{\cC_{z_0}}  \frac{dw_0 dz_0 dw}{(2 \pi i)^3 z_0 z w } q^{2b} \frac{w q^{-2c} - z q^{k + 2}}{w - zq^{k + 2 + 2b}}\\
&\phantom{===========}\Tr|_{\FF_{\mu, s}}\Big(\eta(w_0)\xi(z_0) S_a(t) :\phi_1(z) X^-_b(w): q^{2\lambda \rho+ 2\omega d}\Big),
\end{align*}
where the OPE for $\phi_1(z)$ and $X^-_b(w)$ is valid on a formal neighborhood of $0$ in $q^{-2\omega}$ because the parameters lie in the formal good spectral region.

\subsection{Applying the method of coherent states}

We now use the method of coherent states as applied in \cite{KT} to compute 
\[
T_{a, b} := \Tr|_{\FF_{\mu, s}}\Big(\eta(w_0)\xi(z_0) S_a(t) :\phi_1(z) X^-_b(w):q^{2\lambda\rho +2 \omega d}\Big).
\]
Note that each Fock space takes the form 
\[
\FF_{\mu, s} = \FF_{\beta, \mu} \otimes \FF_{\alpha, s} \otimes \FF_{\bar{\alpha}, s} = \bigotimes_{m \geq 0} \FF_{\beta, \mu, m} \otimes \FF_{\alpha, s, m} \otimes \FF_{\bar{\alpha}, s, m}
\]
and that each mode involved in our vertex operators acts only in a single tensor component.  Define 
\[
\FF_{\mu, s}^0 := \FF_{\beta, \mu, 0} \otimes \FF_{\alpha, s, 0} \otimes \FF_{\ba, s, 0} \qquad \text{ and } \FF_{\mu, s}^{>0} := \bigotimes_{m > 0} \FF_{\beta, \mu, m} \otimes \FF_{\alpha, s, m} \otimes \FF_{\bar{\alpha}, s, m}
\]
so that $\FF_{\mu, s} = \FF_{\mu, s}^0 \otimes \FF_{\mu, s}^{>0}$ and $T_{a, b}$ is the product of traces over $\FF_{\mu, s}^0$ and $\FF_{\mu, s}^{>0}$.  We compute the trace $\Tr|_{\FF_{\mu, s}^0}$ over $\FF_{\mu, s}^0$ in Subsection \ref{sec:zero-mode-trace}.  By Corollary \ref{corr:scale-trace}, the trace over $\FF_{\mu, s}^{>0}$ is a product of one loop correlation functions
\[
T_{PQ} := (q^{-2\omega}; q^{-2\omega})^3 \prod_{m \geq 1} \Tr|_{\FF_{\beta, \mu, m} \otimes \FF_{\alpha, s, m} \otimes \FF_{\bar{\alpha}_s, m}} \Big(P(z) Q(w) q^{2\omega d}\Big)
\]
for each pair of vertex operators $P(z)$ and $Q(w)$ appearing in the free field realization and a global factor $T_g = (q^{-2\omega}; q^{-2\omega})^{-3}$ coming from collecting the prefactors in Corollary \ref{corr:scale-trace} over $m \geq 1$ for each of $\alpha$, $\ba$, and $\beta$.  In Subsection \ref{sec:one-loop-comp}, we compute these one loop correlation functions and their regions of convergence, with results recorded in Table \ref{tab:two-point}.  From these computations, we conclude that $T_{a, b}$ is convergent with value
\[
T_{a, b} = \Tr|_{\FF_{\mu, s}^0} T_g T_{\eta\eta} T_{\xi\xi} T_{S_aS_a} T_{\phi\phi} T_{X^-_bX^-_b} T_{\eta\xi} T_{\eta S_a} T_{\eta X^-_b} T_{\xi S_a} T_{\xi \phi} T_{\xi X^-_b} T_{S_a \phi} T_{S_a X^-_b} T_{:\phi X^-_b:}
\]
for $w_0, z_0, w$ lying on the specified cycles and $t \in \cC_t$.  Substituting the values of the degree $0$ trace and the one loop correlation functions from Subsections \ref{sec:zero-mode-trace} and \ref{sec:one-loop-comp} to compute $T_{a, b}$, we obtain
\begin{align*}
&\Xi(q, \lambda, \omega, \mu, k)\\
&= \frac{z^{\Delta_1}}{(q - q^{-1})^2} \sum_{a, b, c \in \{\pm1\}} (-1)^{\frac{3 - a - b - c}{2}} \sum_{s\in \mu + \ZZ}\int_0^{z q^{-2} \cdot \infty}\frac{d_pt}{t} \oint_{\cC_{w, c}} \oint_{\cC_{w_0}} \oint_{\cC_{z_0}} \frac{dw_0 dz_0 dw}{(2 \pi i)^3 z_0 z w}  q^{2b} \frac{w q^{-2c} - z q^{k + 2}}{w - zq^{k + 2 + 2b}} T_{a, b}\\
&= \frac{z^{\Delta_1}}{(q - q^{-1})^2}  \frac{(q^{-2 \omega + 2}; q^{-2\omega}, q^{-2\kappa})^2}{(q^{-2 \omega - 2}; q^{-2 \omega}, q^{-2\kappa})^2} \frac{(q^{-2\omega}; q^{-2\omega})}{(q^{-2\omega - 2}; q^{-2\omega})(q^{-2\omega+2}; q^{-2\omega})}\sum_{a, b, c \in \{\pm1\}} (-1)^{\frac{3 - a - b - c}{2}} \\
&\phantom{===} \sum_{s \in \mu + \ZZ}\int_0^{z q^{-2} \cdot \infty} \frac{d_pt}{t} \oint_{\cC_{w, c}}  \oint_{\cC_{w_0}} \oint_{\cC_{z_0}}\frac{dw_0 dz_0 dw}{(2 \pi i)^3 z_0 z w} q^{(2\lambda + b - a)s + (a + b)\mu + a + 2b} z^{\frac{2\mu}{\kappa}} t^{-\frac{2(\mu + 1)}{\kappa}} z_0^{-\mu + s} w_0^{\mu - s - 1} \\
&\phantom{===} \theta_0\Big(\frac{z_0}{w_0}; q^{-2\omega}\Big)^{-1} \frac{\theta_0\Big(\frac{t}{z_0} q^{-a}; q^{-2\omega}\Big)}{\theta_0\Big(\frac{t}{w_0} q^{-a}; q^{-2\omega}\Big)} \frac{\theta_0\Big(\frac{w}{w_0} q^{b(k + 1)}; q^{-2 \omega}\Big)}{\theta_0\Big(\frac{w}{z_0} q^{b(k + 1)}; q^{-2 \omega}\Big)}\Omega_{q^2}(\frac{t}{z}; q^{-2\omega}, q^{-2\kappa}) \frac{\theta_0(\frac{t}{z} q^{2}; q^{-2\kappa})\theta_0(\frac{z}{t} q^{-2};q^{-2\omega})}{\theta_0(\frac{t}{z}q^{-2};q^{-2\kappa}) \theta_0(\frac{z}{t}q^2; q^{-2\omega})} \\
& \phantom{===}
\begin{cases}  \frac{\theta_0(\frac{w}{t}q^{ak}; q^{-2\omega})}{\theta_0(\frac{w}{t}q^{a(k+2)}; q^{-2\omega})} & a = b \\ 1 & a \neq b \end{cases}
\begin{cases} \frac{(\frac{z}{w} q^k q^{-2\omega}; q^{-2\omega})}{(\frac{z}{w} q^{k+4} q^{-2\omega}; q^{-2\omega})} & b = 1 \\ \frac{(\frac{w}{z}q^{-k-4}q^{-2\omega}; q^{-2\omega})}{(\frac{w}{z} q^{-k}q^{-2\omega}; q^{-2\omega})} & b = -1 \end{cases} \frac{w q^{-2c} - z q^{k + 2}}{w - zq^{k + 2 + 2b}}.
\end{align*}
Changing the index of summation to $r = \mu - s$, we can simplify this expression to 
\begin{align*}
&\Xi(q, \lambda, \omega, \mu, k) \\
&= \frac{1}{(q - q^{-1})^2} \frac{(q^{-2 \omega + 2}; q^{-2\omega}, q^{-2\kappa})^2}{(q^{-2 \omega - 2}; q^{-2 \omega}, q^{-2\kappa})^2} \frac{(q^{-2\omega}; q^{-2\omega})}{(q^{-2\omega - 2}; q^{-2\omega})(q^{-2\omega+2}; q^{-2\omega})} q^{2\mu \lambda} z^{\frac{2\mu + 2}{\kappa}}\sum_{a, b, c \in \{\pm1\}} (-1)^{\frac{3 - a - b - c}{2}}  \\
&\phantom{===}  \sum_{r \in \ZZ} \int_0^{z q^{-2} \cdot \infty} \frac{d_pt}{t} \oint_{\cC_{w, c}}  \oint_{\cC_{w_0}} \oint_{\cC_{z_0}} \frac{dw_0 dz_0 dw}{(2 \pi i)^3 z_0 z w}q^{2b \mu + a + 2b}  t^{- \frac{2(\mu + 1)}{\kappa}} q^{(-2\lambda - b + a)r} z_0^{-r} w_0^{r - 1} \theta_0\Big(\frac{z_0}{w_0}; q^{-2\omega}\Big)^{-1} \\
&\phantom{===} \frac{\theta_0\Big(\frac{t}{z_0} q^{-a}; q^{-2\omega}\Big)}{\theta_0\Big(\frac{t}{w_0} q^{-a}; q^{-2\omega}\Big)} \frac{\theta_0\Big(\frac{w}{w_0} q^{b(k + 1)}; q^{-2 \omega}\Big)}{\theta_0\Big(\frac{w}{z_0} q^{b(k + 1)}; q^{-2 \omega}\Big)}\Omega_{q^2}(\frac{t}{z}; q^{-2\omega}, q^{-2\kappa}) \frac{\theta_0(\frac{t}{z} q^{2}; q^{-2\kappa})\theta_0(\frac{z}{t} q^{-2};q^{-2\omega})}{\theta_0(\frac{t}{z}q^{-2};q^{-2\kappa}) \theta_0(\frac{z}{t}q^2; q^{-2\omega})}\\
&\phantom{===} \begin{cases}  \frac{\theta_0(\frac{w}{t}q^{ak}; q^{-2\omega})}{\theta_0(\frac{w}{t}q^{a(k+2)}; q^{-2\omega})} & a = b \\ 1 & a \neq b \end{cases} \begin{cases} \frac{(\frac{z}{w} q^k q^{-2\omega}; q^{-2\omega})}{(\frac{z}{w} q^{k+4} q^{-2\omega}; q^{-2\omega})} & b = 1 \\ \frac{(\frac{w}{z}q^{-k-4}q^{-2\omega}; q^{-2\omega})}{(\frac{w}{z} q^{-k}q^{-2\omega}; q^{-2\omega})} & b = -1 \end{cases} \frac{w q^{-2c} - z q^{k + 2}}{w - zq^{k + 2 + 2b}}.
\end{align*}
We now simplify the expression by performing the integrals over $z_0$, $w_0$, and $w$ in that order. 

\begin{center}
\begin{table}  \caption{Values and regions of convergence of one loop correlation functions}\label{tab:two-point}
\begin{tabular}{llll} \hline
Corr. fn. & Value & Region of convergence \\ \hline
$T_{\eta\xi}$ & $\theta_0\Big(\frac{z_0}{w_0}; q^{-2\omega}\Big)^{-1}$ & $|w_0 q^{-2\omega}| < |z_0| < |w_0|$\\
$T_{\eta S_a}$ & $\theta_0\Big(\frac{t}{w_0} q^{-a}; q^{-2 \omega}\Big)^{-1}$ & $|w_0 q^{-2\omega}| < |tq^{-a}| < |w_0|$\\
$T_{\eta \phi}$ & $1$ & all\\
$T_{\eta X^-_b}$ & $\theta_0\Big(\frac{w}{w_0} q^{b(k + 1)}; q^{-2 \omega}\Big)$ & $|w_0 q^{-2\omega}| < |w q^{(k + 1)b}| < |w_0|$\\
$T_{\xi S_a}$ & $\theta_0\Big(\frac{t}{z_0} q^{-a}; q^{-2 \omega}\Big)$ & $|z_0 q^{-2\omega}| < |tq^{-a}| < |z_0|$\\
$T_{\xi \phi}$ & $1$ & all\\
$T_{\xi X^-_b}$ & $\theta_0\Big(\frac{w}{z_0} q^{b(k + 1)}; q^{-2 \omega}\Big)^{-1}$ & $|z_0 q^{-2\omega}| < |w q^{(k + 1)b}| < |z_0|$\\
$T_{S_a \phi}$ & $\Omega_{q^2}(\frac{t}{z}; q^{-2\omega}, q^{-2\kappa}) \frac{\theta_0(\frac{t}{z} q^{2}; q^{-2\kappa})\theta_0(\frac{z}{t} q^{-2};q^{-2\omega})}{\theta_0(\frac{t}{z}q^{-2};q^{-2\kappa}) \theta_0(\frac{z}{t}q^2; q^{-2\omega})}$ & $|tq^{-2} q^{-2 \omega} q^{-2\kappa}|, |tq^2 q^{-2\omega} q^{-2\kappa}| < |z| < |tq^2|, |tq^{-2}|$\\
$T_{S_1 X^-_1}$ & $\frac{\theta_0(\frac{w}{t}q^k; q^{-2\omega})}{\theta_0(\frac{w}{t}q^{k+2}; q^{-2\omega})}$ & $|tq^{-k} q^{-2\omega}|, |tq^{-k-2}q^{-2\omega}| < |w| < |tq^{-k}|, |tq^{-k-2}|$\\
$T_{S_{-1}, X^-_{-1}}$ & $\frac{\theta_0(\frac{w}{t}q^{-k}; q^{-2\omega})}{\theta_0(\frac{w}{t} q^{-k-2}; q^{-2\omega})}$ & $|tq^k q^{-2\omega}|, |tq^{k+2}q^{-2\omega}| < |w| < |tq^k|, |tq^{k+2}|$\\
$T_{S_{a}, X^-_{-a}}$ & $1$ & all\\
$T_{:\phi X^-_1:}$ & $\frac{(\frac{z}{w} q^k q^{-2\omega}; q^{-2\omega})}{(\frac{z}{w} q^{k+4} q^{-2\omega}; q^{-2\omega})}$ & $|zq^k q^{-2\omega}|, |zq^{k+4}q^{-2\omega}| < |w|$\\
$T_{:\phi X^-_{-1}:}$ & $\frac{(\frac{w}{z}q^{-k-4}q^{-2\omega}; q^{-2\omega})}{(\frac{w}{z} q^{-k} q^{-2\omega}; q^{-2\omega})}$ & $|w| < |zq^4 q^{2\omega}|, |zq^{k+4} q^{2\omega}|$\\ \hline
\end{tabular}
\end{table}
\end{center}

\subsection{Computing the $z_0$ contour integral}

We apply the following complex analysis lemma.

\begin{lemma} \label{lem:mero-res-extract}
Let $f(z)$ be a meromorphic function of $z$ with annulus of convergence $r < |z| < R$ near $0$ for some $0 < r < R$.  For $r < r_0, |w| < R$ we have that 
\[
\sum_{m \in \ZZ} \oint_{|z| = r_0} \frac{w^m}{z^{m + 1}} f(z) \frac{dz}{2\pi i} = f(w).
\]
\end{lemma}
\begin{proof}
Because $r < r_0 < R$, the Laurent series expansion of $f(z)$ at $0$ converges uniformly on the compact contour $|z| = r_0$.  Therefore, we may take the term-by-term residue expansion of the integral to obtain a Laurent series expansion which converges to $f(w)$ because $r < |w| < R$. 
\end{proof}

In the formal good spectral region, we have $|w_0| > |z_0| > |wq^{b(k + 1)}|$, which implies that the $q^{-2\omega}$-coefficients of the formal series expansion of 
\[
\frac{\theta_0(\frac{t}{z_0} q^{-a}; q^{-2\omega})}{\theta_0(\frac{z_0}{w_0}; q^{-2\omega})\theta_0(\frac{w}{z_0} q^{b(k + 1)}; q^{-2 \omega})}
\]
admit convergent Laurent series expansion in $z_0$.  In the formal good spectral region and in the doubly formal good region of parameters, we have $|w_0| > |w_0 q^{-2\lambda - b + a}| > |wq^{b(k + 1)}|$ by Lemma \ref{lem:good-ineq}, so we may apply Lemma \ref{lem:mero-res-extract} to each coefficient to obtain
\[
\sum_{r \in \ZZ}\oint_{\cC_{z_0}} \frac{z_0^{-r-1} w_0^{r-1} q^{(-2\lambda - b + a)r}\theta_0(\frac{t}{z_0} q^{-a}; q^{-2\omega})}{\theta_0(\frac{z_0}{w_0}; q^{-2\omega})\theta_0(\frac{w}{z_0} q^{b(k + 1)}; q^{-2 \omega})} \frac{dz_0}{2 \pi i} = \frac{1}{w_0} \frac{\theta_0(\frac{t}{w_0} q^{2\lambda + b - 2a}; q^{-2\omega})}{\theta_0(\frac{w}{w_0} q^{2\lambda + b \kappa - a}; q^{-2\omega}) \theta_0(q^{-2\lambda - b + a}; q^{-2\omega})}
\]
as formal series in $q^{-2\omega}$.  Substitution into the original expression yields
\begin{align*}
\Xi(q, \lambda, \omega,& \mu, k) = \frac{1}{(q - q^{-1})^2} \frac{(q^{-2 \omega + 2}; q^{-2\omega}, q^{-2\kappa})^2}{(q^{-2 \omega - 2}; q^{-2 \omega}, q^{-2\kappa})^2} \frac{(q^{-2\omega}; q^{-2\omega})}{(q^{-2\omega - 2}; q^{-2\omega})(q^{-2\omega+2}; q^{-2\omega})} q^{2\mu \lambda} z^{\frac{2\mu + 2}{\kappa}}  \\
&\phantom{===} \sum_{a, b, c \in \{\pm1\}} (-1)^{\frac{3 - a - b - c}{2}} \int_0^{z q^{-2} \cdot \infty} \frac{d_pt}{t} \oint_{\cC_{w, c}}  \oint_{\cC_{w_0}} \frac{dw_0 dw}{(2 \pi i)^2 w_0 z w}
\frac{q^{2b \mu + a + 2b}  t^{- \frac{2(\mu + 1)}{\kappa}}}{\theta_0(q^{-2\lambda - b + a}; q^{-2\omega})} \\
&\phantom{===}\frac{\theta_0(\frac{t}{w_0} q^{2\lambda + b-2a}; q^{-2\omega})}{\theta_0(\frac{t}{w_0} q^{-a}; q^{-2\omega})} \frac{\theta_0(\frac{w}{w_0} q^{b(k + 1)}; q^{-2 \omega})}{\theta_0(\frac{w}{w_0} q^{2\lambda + b\kappa -a}; q^{-2 \omega})}\Omega_{q^2}(\frac{t}{z}; q^{-2\omega}, q^{-2\kappa}) \frac{\theta_0(\frac{t}{z} q^{2}; q^{-2\kappa})\theta_0(\frac{z}{t} q^{-2};q^{-2\omega})}{\theta_0(\frac{t}{z}q^{-2};q^{-2\kappa}) \theta_0(\frac{z}{t}q^2; q^{-2\omega})}\\
&\phantom{===}
\begin{cases}  \frac{\theta_0(\frac{w}{t}q^{ak}; q^{-2\omega})}{\theta_0(\frac{w}{t}q^{a(k+2)}; q^{-2\omega})} & a = b \\ 1 & a \neq b \end{cases}
\begin{cases} \frac{(\frac{z}{w} q^k q^{-2\omega}; q^{-2\omega})}{(\frac{z}{w} q^{k+4} q^{-2\omega}; q^{-2\omega})} & b = 1 \\ \frac{(\frac{w}{z}q^{-k-4}q^{-2\omega}; q^{-2\omega})}{(\frac{w}{z} q^{-k}q^{-2\omega}; q^{-2\omega})} & b = -1 \end{cases} \frac{w q^{-2c} - z q^{k + 2}}{w - zq^{k + 2 + 2b}}.
\end{align*}

\subsection{Computing the $w_0$ contour integral}

We compute the integral of the convergent formal series in $q^{-2\omega}$ by computing it for sufficiently small $q^{-2\omega}$.  In this subsection, take $q^{-2\omega}$ sufficiently small so that the spectral parameters lie in the good spectral region $\cS_c$.  We first compute the integral of an elliptic function.

\begin{lemma} \label{lem:w0-ell}
For constants $c_1, c_2, c_3$ so that $|c_1|, |c_2|, |c_2|/|c_3|, |c_1|/|c_3| \neq |q|^{-2\omega n}$ for any $n$ and a contour $|w_0| = r$ satisfying $r|q^{-2\omega}| < |c_2|, |c_1|/|c_3| < r$, we have
\begin{multline*}
\int_{|w_0| = r} \frac{\theta_0(\frac{c_1}{w_0}; q^{-2\omega}) \theta_0(\frac{c_2}{c_3 w_0}; q^{-2\omega})}{\theta_0(\frac{c_2}{w_0}; q^{-2\omega})\theta_0(\frac{c_1}{c_3 w_0}; q^{-2\omega})} \frac{dw_0}{2 \pi i w_0}\\
 = \frac{\theta_0(\frac{c_1}{c_2}; q^{-2\omega}) \theta_0(\frac{1}{c_3}; q^{-2\omega})}{\theta_0(\frac{c_1}{c_2c_3}; q^{-2\omega}) (q^{-2\omega}; q^{-2\omega})^2}\Big(2 \eta_1 (C_1 - C_2 - C_3) - \zeta(C_1 - C_2) + \zeta(C_3)\Big),
\end{multline*}
where $C_i = \frac{1}{2\pi i} \log(c_i)$, $\zeta$ and $\sigma$ are the Weierstrass zeta and sigma functions with periods $1$ and $\frac{1}{2\pi i}\log(q^{-2\omega})$, nome $q^{-\omega}$, and $\eta_1 = \zeta(\frac{1}{2})$. 
\end{lemma}
\begin{proof}
Fix a branch of $\log(-)$ and change variables to the additive coordinate $W_0 = \frac{1}{2 \pi i} \log \frac{w_0}{r}$.  Define $R = \frac{1}{2\pi i} \log r$.  By matching zeroes, poles, and residues of functions elliptic in $W_0$, the integrand is given by
\begin{multline} \label{eq:ell-int}
\frac{\theta_0(\frac{c_1}{w_0}; q^{-2\omega}) \theta_0(\frac{c_2}{c_3 w_0}; q^{-2\omega})}{\theta_0(\frac{c_2}{w_0}; q^{-2\omega})\theta_0(\frac{c_1}{c_3 w_0}; q^{-2\omega})} \\ = \frac{\theta_0(\frac{c_1}{c_2}; q^{-2\omega}) \theta_0(\frac{1}{c_3}; q^{-2\omega})}{\theta_0(\frac{c_1}{c_2c_3}; q^{-2\omega}) (q^{-2\omega}; q^{-2\omega})^2}\Big(\zeta(W_0 - C_2 + R) - \zeta(W_0 - C_1 + C_3 + R) - \zeta(C_1 - C_2) + \zeta(C_3)\Big).
\end{multline}
Because $r|q^{-2\omega}| < |c_2|, |c_1|/|c_3| < r$, we have 
\begin{equation} \label{eq:rc-region}
-\Imm\Big(\frac{1}{2 \pi i} \log q^{-2\omega}\Big) < \Imm(R - C_2), \Imm(R - C_1 + C_3) < 0,
\end{equation}
which means that 
\begin{align*}
\int_0^1 \Big(\zeta(W_0 - C_2 + R) - \zeta(W_0 - C_1 + C_3 + R)\Big) dW_0 &= \log \frac{\sigma(1 - C_2 + R)}{\sigma(- C_2 + R)} - \log \frac{\sigma(1 - C_1 + C_3 + R)}{\sigma(- C_1 + C_3)} \\
&= \Big(\eta_1 + 2 \eta_1(-C_2 + R)\Big) - \Big(\eta_1 + 2\eta_1(-C_1 + C_3 + R)\Big)\\
&= 2\eta_1 (C_1 - C_2 - C_3),
\end{align*}
where we may take the same branch of $\log(\sigma(-))$ in both terms of the first equality by (\ref{eq:rc-region}) and we apply $\frac{\sigma(1 + a)}{\sigma(a)} = - e^{\eta_1 + 2\eta_1 a}$ for the second.  Noting that $\int_{|w_0| = r} f(w_0) \frac{dw_0}{2 \pi i w_0} = \int_0^1 f(re^{2\pi i W_0}) dW_0$, integrating each term in (\ref{eq:ell-int}) separately, and substituting, we obtain the desired integral value of
\[
\frac{\theta_0(\frac{c_1}{c_2}; q^{-2\omega}) \theta_0(\frac{1}{c_3}; q^{-2\omega})}{\theta_0(\frac{c_1}{c_2c_3}; q^{-2\omega}) (q^{-2\omega}; q^{-2\omega})^2} \Big(2\eta_1 (C_1 - C_2 - C_3) - \zeta(C_1 - C_2) + \zeta(C_3)\Big). \qedhere
\]
\end{proof}

On our choice of contours, the integral over $w_0$ satisfies the hypotheses of Lemma \ref{lem:w0-ell} with $c_1 = w q^{b(k+1)}$, $c_2 = tq^{-a}$, and $c_3 = q^{-2\lambda - b + a}$ for $q^{-2\omega}$ sufficiently close to $0$ by Lemma \ref{lem:good-ineq}, yielding
\begin{multline*}
\oint_{\cC_{w_0}} \frac{dw_0}{2 \pi i w_0} \frac{\theta_0(\frac{t}{w_0} q^{2\lambda + b-2a}; q^{-2\omega})\theta_0(\frac{w}{w_0} q^{b(k + 1)}; q^{-2 \omega})}{\theta_0(\frac{t}{w_0} q^{-a}; q^{-2\omega})\theta_0(\frac{w}{w_0} q^{2\lambda + b\kappa -a}; q^{-2 \omega})} \\
= \frac{\theta_0(\frac{w}{t} q^{b(k+1) + a}; q^{-2\omega}) \theta_0(q^{2\lambda + b - a}; q^{-2\omega})}{\theta_0(\frac{w}{t} q^{2\lambda + b\kappa}; q^{-2\omega}) (q^{-2\omega}; q^{-2\omega})^2} \frac{2\eta_1 \frac{w}{t} q^{2\lambda + b \kappa} - \zeta(\log(\frac{w}{t} q^{b(k+1) + a})) + \zeta(\log(q^{-2\lambda - b + a}))}{2\pi i}.
\end{multline*}
Substituting in the result and noting that $(q^{-2\omega - 2}; q^{-2\omega})(q^{-2\omega + 2}; q^{-2\omega}) = \frac{\theta_0(q^2; q^{-2\omega})}{1 - q^2}$, we find that
\begin{align*}
\Xi(q, &\lambda, \omega, \mu, k)
= -\frac{1}{q - q^{-1}} \frac{(q^{-2 \omega + 2}; q^{-2\omega}, q^{-2\kappa})^2}{(q^{-2\omega - 2}; q^{-2\omega}, q^{-2\kappa})^2} \frac{1}{\theta_0(q^2; q^{-2\omega})(q^{-2\omega}; q^{-2\omega})} q^{2\lambda(\mu + 1) + 1} z^{\frac{2\mu + 2}{\kappa}}  \\
&\phantom{=} \sum_{a, b, c \in \{\pm1\}} (-1)^{\frac{3 - a - b - c}{2}}  \int_0^{z q^{-2} \cdot \infty} \frac{d_pt}{t}\oint_{\cC_{w, c}} \frac{dw}{2 \pi i z w } q^{2b \mu + 3b}  t^{- \frac{2(\mu + 1)}{\kappa}}\\
&\phantom{=} \frac{\theta_0(\frac{w}{t} q^{bk}; q^{-2\omega})}{\theta_0(\frac{w}{t} q^{2\lambda + b(k+2)}; q^{-2\omega})}\Big(2\eta_1\frac{w}{t} q^{2\lambda + b\kappa} - \zeta(\log(\frac{w}{t} q^{b(k+1) + a})) + \zeta(\log(q^{-2\lambda - b + a}))\Big)\\
& \phantom{=}
 \Omega_{q^2}(\frac{t}{z}; q^{-2\omega}, q^{-2\kappa}) \frac{\theta_0(\frac{t}{z} q^{2}; q^{-2\kappa})\theta_0(\frac{z}{t} q^{-2};q^{-2\omega})}{\theta_0(\frac{t}{z}q^{-2};q^{-2\kappa}) \theta_0(\frac{z}{t}q^2; q^{-2\omega})} \begin{cases} \frac{(\frac{z}{w} q^k q^{-2\omega}; q^{-2\omega})}{(\frac{z}{w} q^{k+4} q^{-2\omega}; q^{-2\omega})} & b = 1 \\ \frac{(\frac{w}{z}q^{-k-4}q^{-2\omega}; q^{-2\omega})}{(\frac{w}{z} q^{-k}q^{-2\omega}; q^{-2\omega})} & b = -1 \end{cases} \frac{w q^{-2c} - z q^{k + 2}}{w - zq^{k + 2 + 2b}}.
\end{align*}
Summing over $a \in \{\pm1\}$, we obtain
\begin{align*}
\Xi(q, &\lambda, \omega, \mu, k) = -\frac{1}{q - q^{-1}} \frac{(q^{-2 \omega + 2}; q^{-2\omega}, q^{-2\kappa})^2}{(q^{-2\omega - 2}; q^{-2\omega}, q^{-2\kappa})^2} \frac{1}{\theta_0(q^2; q^{-2\omega})(q^{-2\omega}; q^{-2\omega})} q^{2\lambda(\mu + 1) + 1} z^{\frac{2\mu + 2}{\kappa}}  \\
&\phantom{=} \sum_{b, c \in \{\pm1\}} (-1)^{\frac{2 - b - c}{2}}  \int_0^{z q^{-2} \cdot \infty} \frac{d_pt}{t} \oint_{\cC_{w, c}} \frac{dw}{2 \pi i z w }q^{2b \mu + 3b}  t^{- \frac{2(\mu + 1)}{\kappa}} \frac{\theta_0(\frac{w}{t} q^{bk}; q^{-2\omega})}{\theta_0(\frac{w}{t} q^{2\lambda + b\kappa}; q^{-2\omega})}\\
&\phantom{===}\Big(\zeta(\log(\frac{w}{t} q^{b(k+1) - 1}))-\zeta(\log(\frac{w}{t} q^{b(k+1) + 1})) + \zeta(\log(q^{-2\lambda - b + 1}))- \zeta(\log(q^{-2\lambda - b - 1}))\Big)\\
& \phantom{===}
 \Omega_{q^2}(\frac{t}{z}; q^{-2\omega}, q^{-2\kappa}) \frac{\theta_0(\frac{t}{z} q^{2}; q^{-2\kappa})\theta_0(\frac{z}{t} q^{-2};q^{-2\omega})}{\theta_0(\frac{t}{z}q^{-2};q^{-2\kappa}) \theta_0(\frac{z}{t}q^2; q^{-2\omega})}\begin{cases} \frac{(\frac{z}{w} q^k q^{-2\omega}; q^{-2\omega})}{(\frac{z}{w} q^{k+4} q^{-2\omega}; q^{-2\omega})} & b = 1 \\ \frac{(\frac{w}{z}q^{-k-4}q^{-2\omega}; q^{-2\omega})}{(\frac{w}{z} q^{-k}q^{-2\omega}; q^{-2\omega})} & b = -1 \end{cases} \frac{w q^{-2c} - z q^{k + 2}}{w - zq^{k + 2 + 2b}}.
\end{align*}
By matching zeroes, poles, and residues of elliptic functions in $\log t$, we see that 
\begin{multline*}
\frac{1}{2\pi i} \Big(\zeta(\log \frac{w}{t} q^{b(k+1) -1}) - \zeta(\log(\frac{w}{t} q^{b(k+1) + 1})) + \zeta(\log(q^{-2\lambda + 1 - b})) - \zeta(\log(q^{-2\lambda - 1 - b}))\Big)\\
 = \frac{\theta_0(\frac{w}{t} q^{b\kappa + 2 \lambda}; q^{-2\omega}) \theta_0(\frac{w}{t} q^{bk - 2\lambda}; q^{-2\omega})}{\theta_0(\frac{w}{t} q^{b(k+1) - 1}; q^{-2\omega}) \theta_0(\frac{w}{t} q^{b(k+1) + 1}; q^{-2\omega})}
\frac{\theta_0(q^2; q^{-2\omega}) (q^{-2\omega}; q^{-2\omega})^2}{\theta_0(q^{-b + 1 - 2\lambda}; q^{-2\omega}) \theta_0(q^{b + 1 + 2\lambda}; q^{-2\omega})}.
\end{multline*}
Substituting in yields
\begin{align*}
\Xi(q,& \lambda, \omega, \mu, k) = -\frac{1}{q - q^{-1}} \frac{(q^{-2 \omega + 2}; q^{-2\omega}, q^{-2\kappa})^2}{(q^{-2\omega - 2}; q^{-2\omega}, q^{-2\kappa})^2} (q^{-2\omega}; q^{-2\omega}) q^{2\lambda(\mu + 1) + 1} z^{\frac{2\mu + 2}{\kappa}} \sum_{b, c \in \{\pm1\}} (-1)^{\frac{2 - b - c}{2}} \int_0^{z q^{-2} \cdot \infty} \frac{d_pt}{t} \\
&\phantom{=}  \oint_{\cC_{w, c}} \frac{dw}{2 \pi i z w}  \frac{q^{2b \mu + 3b}  t^{- \frac{2(\mu + 1)}{\kappa}}}{\theta_0(q^{-b + 1 - 2\lambda}; q^{-2\omega}) \theta_0(q^{b + 1 + 2\lambda}; q^{-2\omega})} \Omega_{q^2}(\frac{t}{z}; q^{-2\omega}, q^{-2\kappa}) \frac{\theta_0(\frac{t}{z} q^{2}; q^{-2\kappa})\theta_0(\frac{z}{t} q^{-2};q^{-2\omega})}{\theta_0(\frac{t}{z}q^{-2};q^{-2\kappa}) \theta_0(\frac{z}{t}q^2; q^{-2\omega})}\\
&\phantom{=} \frac{\theta_0(\frac{w}{t} q^{bk}; q^{-2\omega}) \theta_0(\frac{w}{t} q^{bk - 2\lambda}; q^{-2\omega})}{\theta_0(\frac{w}{t} q^{b(k+1) - 1}; q^{-2\omega}) \theta_0(\frac{w}{t} q^{b(k+1) + 1}; q^{-2\omega})}\begin{cases} \frac{(\frac{z}{w} q^k q^{-2\omega}; q^{-2\omega})}{(\frac{z}{w} q^{k+4} q^{-2\omega}; q^{-2\omega})} & b = 1 \\ \frac{(\frac{w}{z}q^{-k-4}q^{-2\omega}; q^{-2\omega})}{(\frac{w}{z} q^{-k}q^{-2\omega}; q^{-2\omega})} & b = -1 \end{cases} \frac{w q^{-2c} - z q^{k + 2}}{w - zq^{k + 2 + 2b}},
\end{align*}
where we may again interpret the expression as a formal series in $q^{-2\omega}$. 

\subsection{Computing the $w$ contour integrals}

We now evaluate the contour integrals over $w$. We again will compute the integral for numerical $q^{-2\omega}$ sufficiently close to $0$ that the parameters lie in the good spectral region.  Define
\[
I_{w, 1, c}(t) := \oint_{\cC_{w, c}} \frac{dw}{2 \pi i w} \frac{\theta_0(\frac{w}{t} q^{k - 2 \lambda}; q^{-2\omega})}{\theta_0(\frac{w}{t} q^{k+2}; q^{-2\omega})} \frac{(\frac{z}{w} q^k q^{-2\omega}; q^{-2\omega})}{(\frac{z}{w} q^{k+4}; q^{-2\omega})} (q^{-2c} - z w^{-1}q^{k + 2})
\]
and
\[
I_{w, -1, c}(t) := \oint_{\cC_{w, c}} \frac{dw}{2\pi i w}\frac{\theta_0(\frac{w}{t} q^{-k - 2 \lambda}; q^{-2\omega})}{\theta_0(\frac{w}{t} q^{-k - 2}; q^{-2\omega})} \frac{(\frac{w}{z} q^{-k-4} q^{-2\omega}; q^{-2\omega})}{(\frac{w}{z} q^{-k}; q^{-2\omega})} (q^2 - wz^{-1} q^{-k-2c}).
\]

\begin{lemma} \label{lem:w-int-vals-new}
As formal series in $q^{-2\omega}$, the differences of integrals $I_{w, 1}(t) := I_{w, 1, 1}(t) - I_{w, 1, -1}(t)$ and $I_{w, -1}(t) := I_{w, -1, 1}(t) - I_{w, -1, -1}(t)$ are given by 
\[
I_{w, 1}(t) = I^1_{w, 1}(t) + I^2_{w, 1}(t) \qquad \text{ and } \qquad I_{w, -1}(t) = I^1_{w, -1}(t),
\]
where 
\begin{multline*}
I^1_{w, 1}(t) = (q^{-2} - q^{- 2\lambda})\frac{\theta_0(q^{-2-2\lambda} q^{-2\omega}; q^{-2\omega})}{(q^{-2\omega}; q^{-2\omega})^2} \frac{(\frac{z}{t} q^{2k + 2}; q^{-2\omega})}{(\frac{z}{t}q^{2k + 6}; q^{-2\omega})}\\\hgf(\frac{t}{z} q^{-2k-2} q^{-2\omega}, q^{-2\omega}; \frac{t}{z} q^{-2k-6} q^{-2\omega}; q^{-2\omega}, q^{2\lambda - 2})\\
+ q^{- 2\lambda} \frac{\theta_0(q^{-2-2\lambda} q^{-2\omega}; q^{-2\omega})}{(q^{-2\omega}; q^{-2\omega})^2} \frac{(\frac{z}{t} q^{2k + 2}; q^{-2\omega})}{(\frac{z}{t}q^{2k + 6}; q^{-2\omega})}
\end{multline*}
and
\[
I^2_{w, 1}(t) = (q^{2\lambda} - q^2) \frac{\theta_0(\frac{z}{t} q^{2k + 4 - 2\lambda}; q^{-2\omega})}{\theta_0(\frac{z}{t} q^{2k + 6}; q^{-2\omega})} \frac{(q^{-4}; q^{-2\omega})}{(q^{-2\omega}; q^{-2\omega})} \frac{(q^{2 \lambda + 2} q^{-2\omega}; q^{-2\omega})}{(q^{2\lambda - 2}; q^{-2\omega})}
\]
and 
\begin{multline*}
I^1_{w, -1}(t) = (q^2 - q^{4 - 2\lambda}) \frac{\theta_0(q^{2 - 2\lambda}q^{-2\omega}; q^{-2\omega})}{(q^{-2\omega}; q^{-2\omega})^2} \frac{(\frac{t}{z}q^{-2} q^{-2\omega}; q^{-2\omega})}{(\frac{t}{z} q^2 q^{-2\omega}; q^{-2\omega})} \\\hgf(\frac{t}{z} q^2 q^{-2\omega}, q^{-2\omega}; \frac{t}{z} q^{-2} q^{-2\omega}; q^{-2\omega}, q^{2\lambda - 2}) \\
+ q^{4 - 2\lambda} \frac{\theta_0(q^{2 - 2\lambda}q^{-2\omega}; q^{-2\omega})}{(q^{-2\omega}; q^{-2\omega})^2} \frac{(\frac{t}{z}q^{-2} q^{-2\omega}; q^{-2\omega})}{(\frac{t}{z} q^2 q^{-2\omega}; q^{-2\omega})}
\end{multline*}
and $\hgf(a_1, a_2; b_1; q, z)$ denotes the $q$-hypergeometric function
\[
\hgf(a_1, a_2; b_1; q, z) := \sum_{n \geq 0} \frac{(a_1; q)_n (a_2; q)_n}{(b_1; q)_n(q;q)_n} z^n.
\]
\end{lemma}
\begin{proof}
We compute the integrals by deforming contours to $0$.  By Corollary \ref{corr:theta-ratio}, we have the estimates
\begin{align*}
\left|\frac{\theta_0(\frac{w}{t} q^{k - 2 \lambda}; q^{-2\omega})}{\theta_0(\frac{w}{t} q^{k+2}; q^{-2\omega})}\right| &\leq D_1(q^{-2\omega}, \eps) |q|^{-\frac{2\lambda - 2}{2}}|\frac{w^2}{t^2}q^{2k - 2\lambda + 2}|^{-\frac{\lambda + 1}{2\omega}}\\
\left|\frac{\theta_0(\frac{w}{t} q^{-k - 2 \lambda}; q^{-2\omega})}{\theta_0(\frac{w}{t} q^{-k - 2}; q^{-2\omega})}\right| &\leq D_1(q^{-2\omega}, \eps) |q|^{- \frac{2\lambda - 2}{2}}|\frac{w^2}{t^2} q^{-2k - 2\lambda - 2}|^{- \frac{\lambda - 1}{2 \omega}}.
\end{align*}
For $q^{-2\omega}$ sufficiently small, we see that $\frac{\lambda \pm 1}{2 \omega} < \frac{1}{2}$, meaning that we may compute $I_{w, 1, c}(t)$ and $I_{w, -1, c}(t)$ by deforming $\cC_{w, c}$ to $0$.  We now perform the deformations one by one.

\noindent \textit{Computing $I_{w, 1, 1}(t)$:} For $I_{w, 1, 1}(t)$, we have $|w| < |zq^k|, |z q^{k+4}|, |tq^{-k-2}|$, so we wish to sum residues at $w = t q^{-k - 2} q^{-2\omega(n+1)}$ and $w = zq^{k+4}q^{-2\omega(n + 1)}$.  The first set has residues
\begin{multline*}
q^{-2} \frac{\theta_0(q^{-2 - 2\lambda} q^{-2\omega(n + 1)}; q^{-2\omega})}{(q^{-2\omega(n + 1)}; q^{-2\omega}) (q^{-2\omega}; q^{-2\omega}) (q^{2\omega n}; q^{-2\omega})_n} \frac{(\frac{z}{t} q^{2k + 2} q^{2\omega n}; q^{-2\omega})}{(\frac{z}{t} q^{2k + 6} q^{2\omega n}; q^{-2\omega})}\\
= q^{-2} q^{(2\lambda - 2)n} \frac{\theta_0(q^{-2 - 2\lambda} q^{-2\omega}; q^{-2\omega})}{(q^{-2\omega}; q^{-2\omega})^2} \frac{(\frac{z}{t}q^{2k + 2}; q^{-2\omega})}{(\frac{z}{t} q^{2k + 6}; q^{-2\omega})} \frac{(\frac{t}{z} q^{-2k-2} q^{-2\omega}; q^{-2\omega})_n}{(\frac{t}{z} q^{-2k-6} q^{-2\omega}; q^{-2\omega})_n}
\end{multline*}
The second set has residues
\begin{multline*}
q^{-2} \frac{\theta_0(\frac{z}{t} q^{2k + 4 - 2\lambda} q^{-2\omega (n + 1)}; q^{-2\omega})}{\theta_0(\frac{z}{t} q^{2k + 6} q^{-2\omega(n + 1)}; q^{-2\omega})} \frac{(q^{-4}q^{2\omega n}; q^{-2\omega})}{(q^{2\omega n}; q^{-2\omega})_n (q^{-2\omega}; q^{-2\omega})}\\
= q^{2\lambda} q^{(2\lambda - 2)n} \frac{\theta_0(\frac{z}{t} q^{2k + 4 - 2\lambda}; q^{-2\omega})}{\theta_0(\frac{z}{t} q^{2k + 6}; q^{-2\omega})} \frac{(q^4q^{-2\omega}; q^{-2\omega})_n (q^{-4}; q^{-2\omega})}{(q^{-2\omega}; q^{-2\omega})_n (q^{-2\omega}; q^{-2\omega})}.
\end{multline*}
Summing over these residues yields
\begin{multline*}
I_{w, 1, 1}(t) = q^{-2}\frac{\theta_0(q^{-2 - 2\lambda} q^{-2\omega}; q^{-2\omega})}{(q^{-2\omega}; q^{-2\omega})^2} \frac{(\frac{z}{t}q^{2k + 2}; q^{-2\omega})}{(\frac{z}{t} q^{2k + 6}; q^{-2\omega})} \hgf(\frac{t}{z} q^{-2k-2} q^{-2\omega}, q^{-2\omega}; \frac{t}{z} q^{-2k-6} q^{-2\omega}; q^{-2\omega}, q^{(2\lambda - 2)}) \\
+ q^{2\lambda} \frac{\theta_0(\frac{z}{t} q^{2k + 4 - 2\lambda}; q^{-2\omega})}{\theta_0(\frac{z}{t} q^{2k + 6}; q^{-2\omega})} \frac{(q^{-4}; q^{-2\omega})}{(q^{-2\omega}; q^{-2\omega})} \frac{(q^{2\lambda + 2} q^{-2\omega}; q^{-2\omega})}{(q^{2\lambda - 2}; q^{-2\omega})}. 
\end{multline*}

\noindent \textit{Computing $I_{w, 1, -1}(t)$:} We have $|tq^{-k-2}| > |w| > |zq^k|, |zq^{k+4}|$, so we wish to sum residues at $w = t q^{-k - 2} q^{-2\omega(n + 1)}$ and $w = z q^{k + 4} q^{-2n\omega}$.  The first set has residues
\begin{multline*}
q^2 \frac{\theta_0(q^{-2 - 2\lambda} q^{-2\omega (n + 1)}; q^{-2\omega})}{(q^{-2\omega}; q^{-2\omega}) (q^{2\omega n}; q^{-2\omega})_{n} (q^{-2\omega(n + 1)}; q^{-2\omega})} \frac{(\frac{z}{t} q^{2k + 2} q^{2\omega(n + 1)}; q^{-2\omega})}{(\frac{z}{t} q^{2k + 6} q^{2\omega (n + 1)}; q^{-2\omega})} \\
= q^{-2} q^{(2\lambda - 2)n} \frac{\theta_0(q^{-2-2\lambda} q^{-2\omega}; q^{-2\omega})}{(q^{-2\omega}; q^{-2\omega})^2} \frac{(\frac{z}{t} q^{2k + 2}; q^{-2\omega})}{(\frac{z}{t}q^{2k + 6}; q^{-2\omega})}\frac{(\frac{t}{z} q^{-2k-2} q^{-2\omega}; q^{-2\omega})_{n + 1}}{(\frac{t}{z} q^{-2k-6} q^{-2\omega}; q^{-2\omega})_{n + 1}}.
\end{multline*}
The second set has residues
\begin{multline*}
q^2 \frac{\theta_0(\frac{z}{t} q^{2k + 4 - 2\lambda} q^{-2\omega n}; q^{-2\omega})}{\theta_0(\frac{z}{t} q^{2k + 6} q^{-2\omega n}; q^{-2\omega})} \frac{(q^{-4}q^{2\omega n}; q^{-2\omega})}{(q^{2\omega n}; q^{-2\omega})_n(q^{-2\omega}; q^{-2\omega})}\\
= q^2 q^{(2\lambda - 2)n}\frac{\theta_0(\frac{z}{t} q^{2k + 4 - 2\lambda}; q^{-2\omega})}{\theta_0(\frac{z}{t} q^{2k + 6}; q^{-2\omega})} \frac{(q^4 q^{-2\omega}; q^{-2\omega})_n (q^{-4}; q^{-2\omega})}{(q^{-2\omega}; q^{-2\omega})_n(q^{-2\omega}; q^{-2\omega})}.
\end{multline*}
Summing over these residues yields
\begin{multline*}
I_{w, 1, -1}(t) = q^{- 2\lambda}\frac{\theta_0(q^{-2-2\lambda} q^{-2\omega}; q^{-2\omega})}{(q^{-2\omega}; q^{-2\omega})^2} \frac{(\frac{z}{t} q^{2k + 2}; q^{-2\omega})}{(\frac{z}{t}q^{2k + 6}; q^{-2\omega})}\\
(\hgf(\frac{t}{z} q^{-2k-2} q^{-2\omega}, q^{-2\omega}; \frac{t}{z} q^{-2k-6} q^{-2\omega}; q^{-2\omega}, q^{2\lambda - 2}) - 1) \\
+ q^2 \frac{\theta_0(\frac{z}{t} q^{2k + 4 - 2\lambda}; q^{-2\omega})}{\theta_0(\frac{z}{t} q^{2k + 6}; q^{-2\omega})} \frac{(q^{-4}; q^{-2\omega})}{(q^{-2\omega}; q^{-2\omega})} \frac{(q^{2 \lambda + 2} q^{-2\omega}; q^{-2\omega})}{(q^{2\lambda - 2}; q^{-2\omega})}.
\end{multline*}

\noindent \textit{Computing $I_{w, 1}(t)$:} Combining these terms, we find that 
\begin{multline*}
I_{w, 1}(t) = (q^{-2} - q^{- 2\lambda})\frac{\theta_0(q^{-2-2\lambda} q^{-2\omega}; q^{-2\omega})}{(q^{-2\omega}; q^{-2\omega})^2} \frac{(\frac{z}{t} q^{2k + 2}; q^{-2\omega})}{(\frac{z}{t}q^{2k + 6}; q^{-2\omega})}\\\hgf(\frac{t}{z} q^{-2k-2} q^{-2\omega}, q^{-2\omega}; \frac{t}{z} q^{-2k-6} q^{-2\omega}; q^{-2\omega}, q^{2\lambda - 2})\\
+ q^{- 2\lambda} \frac{\theta_0(q^{-2-2\lambda} q^{-2\omega}; q^{-2\omega})}{(q^{-2\omega}; q^{-2\omega})^2} \frac{(\frac{z}{t} q^{2k + 2}; q^{-2\omega})}{(\frac{z}{t}q^{2k + 6}; q^{-2\omega})}\\
+ (q^{2\lambda} - q^2) \frac{\theta_0(\frac{z}{t} q^{2k + 4 - 2\lambda}; q^{-2\omega})}{\theta_0(\frac{z}{t} q^{2k + 6}; q^{-2\omega})} \frac{(q^{-4}; q^{-2\omega})}{(q^{-2\omega}; q^{-2\omega})} \frac{(q^{2 \lambda + 2} q^{-2\omega}; q^{-2\omega})}{(q^{2\lambda - 2}; q^{-2\omega})}.
\end{multline*}

\noindent \textit{Computing $I_{w, -1, 1}(t)$:} In this case, we have $|w| < |zq^k|, |z q^{k + 4}|, |tq^{k + 2}|$ on $\cC_{w, 1}$, so it suffices to consider the poles at $w = tq^{k + 2} q^{-2\omega (n + 1)}$ for $n \geq 0$.  They have residues
\begin{multline*}
q^2\frac{\theta_0(q^{2 - 2\lambda}q^{-2\omega(n + 1)}; q^{-2\omega})}{(q^{-2\omega}; q^{-2\omega})(q^{2\omega n}; q^{-2\omega})_{n} (q^{-2\omega(n + 1)}; q^{-2\omega})} \frac{(\frac{t}{z} q^{-2} q^{-2\omega(n + 1)}; q^{-2\omega})}{(\frac{t}{z}q^2 q^{-2\omega(n + 1)}; q^{-2\omega})}\\
= q^{2} q^{(2\lambda - 2)n} \frac{\theta_0(q^{2 - 2\lambda}q^{-2\omega}; q^{-2\omega})}{(q^{-2\omega}; q^{-2\omega})^2} \frac{(\frac{t}{z}q^{-2} q^{-2\omega}; q^{-2\omega})}{(\frac{t}{z} q^2 q^{-2\omega}; q^{-2\omega})} \frac{(\frac{t}{z} q^2 q^{-2\omega}; q^{-2\omega})_n}{(\frac{t}{z} q^{-2} q^{-2\omega}; q^{-2\omega})_n}.
\end{multline*}
Summing over these yields
\[
I_{w, -1, 1}(t) = q^{2} \frac{\theta_0(q^{2 - 2\lambda}q^{-2\omega}; q^{-2\omega})}{(q^{-2\omega}; q^{-2\omega})^2} \frac{(\frac{t}{z}q^{-2} q^{-2\omega}; q^{-2\omega})}{(\frac{t}{z} q^2 q^{-2\omega}; q^{-2\omega})} \hgf(\frac{t}{z} q^2 q^{-2\omega}, q^{-2\omega}; \frac{t}{z} q^{-2} q^{-2\omega}; q^{-2\omega}, q^{2\lambda - 2}).
\]

\noindent \textit{Computing $I_{w, -1, -1}(t)$:} In this case, we have $|w| > |zq^k|, |z q^{k + 4}|$ and $w < |tq^{k + 2}|$ on $\cC_{w, -1}$, so it suffices to consider the poles at $w = tq^{k + 2} q^{-2\omega (n + 1)}$ for $n \geq 0$. They have residues
\begin{multline*}
q^2 \frac{\theta_0(q^{2 - 2\lambda}q^{-2\omega(n + 1)}; q^{-2\omega})}{(q^{-2\omega}; q^{-2\omega})(q^{2\omega n}; q^{-2\omega})_{n} (q^{-2\omega(n + 1)}; q^{-2\omega})} \frac{(\frac{t}{z} q^{-2} q^{-2\omega(n + 2)}; q^{-2\omega})}{(\frac{t}{z}q^2 q^{-2\omega(n + 2)}; q^{-2\omega})}\\
=  q^{2} q^{(2\lambda - 2)n} \frac{\theta_0(q^{2 - 2\lambda}q^{-2\omega}; q^{-2\omega})}{(q^{-2\omega}; q^{-2\omega})^2} \frac{(\frac{t}{z}q^{-2} q^{-2\omega}; q^{-2\omega})}{(\frac{t}{z} q^2 q^{-2\omega}; q^{-2\omega})} \frac{(\frac{t}{z} q^2 q^{-2\omega}; q^{-2\omega})_{n+1}}{(\frac{t}{z} q^{-2} q^{-2\omega}; q^{-2\omega})_{n + 1}}.
\end{multline*}
Summing over these residues yields
\[
I_{w, -1, -1}(t) = q^{4 - 2\lambda} \frac{\theta_0(q^{2 - 2\lambda}q^{-2\omega}; q^{-2\omega})}{(q^{-2\omega}; q^{-2\omega})^2} \frac{(\frac{t}{z}q^{-2} q^{-2\omega}; q^{-2\omega})}{(\frac{t}{z} q^2 q^{-2\omega}; q^{-2\omega})} (\hgf(\frac{t}{z}q^2 q^{-2\omega}, q^{-2\omega}; \frac{t}{z} q^{-2} q^{-2\omega}; q^{-2\omega}, q^{2\lambda - 2}) - 1).
\]

\noindent \textit{Computing $I_{w, -1}(t)$:} Combining these terms, we find that
\begin{multline*}
I_{w, -1}(t) = (q^2 - q^{4 - 2\lambda}) \frac{\theta_0(q^{2 - 2\lambda}q^{-2\omega}; q^{-2\omega})}{(q^{-2\omega}; q^{-2\omega})^2} \frac{(\frac{t}{z}q^{-2} q^{-2\omega}; q^{-2\omega})}{(\frac{t}{z} q^2 q^{-2\omega}; q^{-2\omega})} \hgf(\frac{t}{z} q^2 q^{-2\omega}, q^{-2\omega}; \frac{t}{z} q^{-2} q^{-2\omega}; q^{-2\omega}, q^{2\lambda - 2}) \\
+ q^{4 - 2\lambda} \frac{\theta_0(q^{2 - 2\lambda}q^{-2\omega}; q^{-2\omega})}{(q^{-2\omega}; q^{-2\omega})^2} \frac{(\frac{t}{z}q^{-2} q^{-2\omega}; q^{-2\omega})}{(\frac{t}{z} q^2 q^{-2\omega}; q^{-2\omega})}. \qedhere
\end{multline*}
\end{proof}

\subsection{Completing the proof of Proposition \ref{prop:ff-comp}} \label{sec:sub-int}

We now rearrange the results of our integrals to obtain the desired result.  All expressions are now formal series in $q^{-2\omega}$ and then $q^{-2\mu}$ and meromorphically continued in $t$ to the full Jackson cycle from $\cC_t$.  For $i \in \{1, 2\}$, define $K^i_{b}(t) := L(t) J^i_{w, b}(t)$ with
\begin{align*}
J^i_{w, b}(t) &= q^{2b \mu + 3b} \frac{I^i_{w, b}}{\theta_0(q^{-b + 1 - 2\lambda}; q^{-2\omega}) \theta_0(q^{b + 1 + 2\lambda}; q^{-2\omega})}
\end{align*}
and
\[
L(t) = t^{- \frac{2(\mu + 1)}{k + 2}} \Omega_{q^2}(t/z; q^{-2\omega}, q^{-2k-4}) \frac{\theta_0(t/z q^{2}; q^{-2k-4})\theta_0(z/t q^{-2};q^{-2\omega})}{\theta_0(t/zq^{-2};q^{-2k-4}) \theta_0(z/tq^2; q^{-2\omega})}.
\]
In this notation, we have
\begin{multline} \label{eq:int-k}
\Xi(q, \lambda, \omega, \mu, k)= -\frac{1}{q - q^{-1}} \frac{(q^{-2 \omega + 2}; q^{-2\omega}, q^{-2 k - 4})^2}{(q^{-2 \omega - 2}; q^{-2 \omega}, q^{-2k - 4})^2} (q^{-2\omega}; q^{-2\omega}) q^{2\lambda(\mu + 1) + 1}z^{\frac{2\mu + 2}{k + 2}}\\ \int_0^{z q^{-2} \cdot \infty} \frac{d_pt}{t} (K^1_{1}(t) + K^2_{1}(t) - K^1_{-1}(t)).
\end{multline}

\begin{lemma} \label{eq:k-cancel}
We have $K^1_1(pt) = K^1_{-1}(t)$.
\end{lemma}
\begin{proof}
By Lemma \ref{lem:phase-trans}, we have
\[
\Omega_{q^2}(p\frac{t}{z}; q^{-2\omega}, q^{-2\kappa}) = \frac{\theta_0(\frac{t}{z} q^{2\kappa - 2};q^{-2\omega})}{\theta_0(\frac{t}{z} q^{2\kappa + 2};q^{-2\omega})} \Omega_{q^2}(\frac{t}{z}; q^{-2\omega}, q^{-2\kappa}),
\]
and we also have that 
\begin{align*}
\frac{\theta_0(p\frac{t}{z} q^{2}; q^{-2\kappa})\theta_0(\frac{z}{t} p^{-1}q^{-2};q^{-2\omega})}{\theta_0(p\frac{t}{z}q^{-2};q^{-2\kappa}) \theta_0(\frac{z}{t} p^{-1}q^2; q^{-2\omega})} &= \frac{\theta_0(p\frac{t}{z} q^{2}; q^{-2\kappa})\theta_0(\frac{t}{z} q^{2\kappa + 2} q^{-2\omega};q^{-2\omega})}{\theta_0(\frac{t}{z}p q^{-2};q^{-2\kappa}) \theta_0(\frac{t}{z} q^{2\kappa - 2} q^{-2\omega}; q^{-2\omega})}\\& = \frac{\theta_0(\frac{t}{z} q^{2}; q^{-2\kappa})\theta_0(\frac{t}{z} q^{2\kappa + 2};q^{-2\omega})}{\theta_0(\frac{t}{z}q^{-2};q^{-2\kappa}) \theta_0(\frac{t}{z} q^{2\kappa - 2}; q^{-2\omega})}.
\end{align*}
Together, these imply that
\[
L(pt) = q^{-4 \mu - 4}\frac{\theta_0(\frac{z}{t}q^2; q^{-2\omega})}{\theta_0(\frac{z}{t} q^{-2};q^{-2\omega})} L(t)
\]
and therefore by Lemma \ref{lem:w-int-vals-new} and the definition of $J^1_{w, \pm}(t)$ we have
\begin{multline*}
J^1_{w, 1}(pt) = -\frac{q^{2\mu + 2\lambda + 3} (q^{-2} - q^{- 2\lambda})}{\theta_0(q^{2\lambda};q^{-2\omega})(q^{-2\omega}; q^{-2\omega})^2} \frac{(\frac{z}{t} q^{-2}; q^{-2\omega})}{(\frac{z}{t}q^{2}; q^{-2\omega})}\hgf(\frac{t}{z} q^{2} q^{-2\omega}, q^{-2\omega}; \frac{t}{z} q^{-2} q^{-2\omega}; q^{-2\omega}, q^{2\lambda - 2})\\
-  \frac{q^{2\mu + 3}}{\theta_0(q^{2\lambda};q^{-2\omega})(q^{-2\omega}; q^{-2\omega})^2} \frac{(\frac{z}{t} q^{-2} q^{2\omega}; q^{-2\omega})}{(\frac{z}{t}q^{2} q^{2\omega}; q^{-2\omega})}
\end{multline*}
and
\begin{multline*}
J^1_{w, -1}(t) = - \frac{q^{-2\mu + 2\lambda - 1}(q^{-2} - q^{- 2\lambda})}{\theta_0(q^{2\lambda}; q^{-2\omega})(q^{-2\omega}; q^{-2\omega})^2} \frac{(\frac{t}{z}q^{-2} q^{-2\omega}; q^{-2\omega})}{(\frac{t}{z} q^2 q^{-2\omega}; q^{-2\omega})} \hgf(\frac{t}{z} q^2 q^{-2\omega}, q^{-2\omega}; \frac{t}{z} q^{-2} q^{-2\omega}; q^{-2\omega}, q^{2\lambda - 2}) \\
-\frac{q^{-2\mu - 1}}{\theta_0(q^{2\lambda};q^{-2\omega})(q^{-2\omega}; q^{-2\omega})^2} \frac{(\frac{t}{z}q^{-2} q^{-2\omega}; q^{-2\omega})}{(\frac{t}{z} q^2 q^{-2\omega}; q^{-2\omega})}
\end{multline*}
We conclude that
\begin{align*}
K^1_{1}(p t) &= L(t) q^{-4 \mu - 4}\frac{\theta_0(\frac{z}{t}q^2; q^{-2\omega})}{\theta_0(\frac{z}{t} q^{-2};q^{-2\omega})} J^1_{w, 1}(pt) = L(t) J^1_{w, -1}(t) = K^1_{-1}(t). \qedhere
\end{align*}
\end{proof}

We are now ready to deduce Proposition \ref{prop:ff-comp}.  By (\ref{eq:int-k}), invariance of the formal Jackson integral under $p$-shifts, and Lemma \ref{eq:k-cancel}, we obtain
\begin{align*}
\Xi(q, \lambda, \omega, \mu, k) &= -\frac{1}{q - q^{-1}} \frac{(q^{-2 \omega + 2}; q^{-2\omega}, q^{-2 \kappa})^2}{(q^{-2 \omega - 2}; q^{-2 \omega}, q^{-2\kappa})^2} (q^{-2\omega}; q^{-2\omega}) q^{2\lambda(\mu + 1) + 1}z^{\frac{2\mu + 2}{\kappa}} \\
&\phantom{=========================} \int_0^{z q^{-2} \cdot \infty} \frac{d_pt}{t} (K^1_{1}(pt) + K^2_{1}(pt) - K^1_{-1}(t))\\
&= - \frac{1}{q - q^{-1}} \frac{(q^{-2 \omega + 2}; q^{-2\omega}, q^{-2 \kappa})^2}{(q^{-2 \omega - 2}; q^{-2 \omega}, q^{-2\kappa})^2} (q^{-2\omega}; q^{-2\omega}) q^{2\lambda(\mu + 1) + 1}z^{\frac{2\mu + 2}{\kappa}}  \int_0^{z q^{-2} \cdot \infty} \frac{d_pt}{t} K^2_{1}(pt).
\end{align*}
Now, by Lemma \ref{lem:w-int-vals-new}, we find that 
\begin{align*}
J^2_{w, 1}(pt)&= - \frac{q^{2\mu - 2\lambda + 1}(q^{2\lambda} - q^2)}{\theta_0(q^{-2\lambda}; q^{-2\omega})(q^{2\lambda - 2}; q^{-2\omega}) (q^{-2\lambda - 2}; q^{-2\omega})} \frac{(q^{-4}; q^{-2\omega})}{(q^{-2\omega}; q^{-2\omega})}\frac{\theta_0(\frac{z}{t} q^{- 2\lambda}; q^{-2\omega})}{\theta_0(\frac{z}{t} q^{2}; q^{-2\omega})} 
\end{align*}
and therefore that 
\begin{align*}
K^2_{1}(t) = -C_1(\lambda, \mu)t^{- \frac{2(\mu + 1)}{\kappa}} \Omega_{q^2}(\frac{t}{z}; q^{-2\omega}, q^{-2\kappa}) \frac{\theta_0(\frac{t}{z} q^{2}; q^{-2\kappa})}{\theta_0(\frac{t}{z}q^{-2};q^{-2\kappa})}\frac{\theta_0(\frac{z}{t} q^{-2\lambda}; q^{-2\omega})}{\theta_0(\frac{z}{t} q^{2}; q^{-2\omega})}
\end{align*}
for the constant
\[
C_1(\lambda, \mu) = \frac{q^{-2\mu - 1}}{\theta_0(q^{2\lambda}; q^{-2\omega})(q^{2\lambda - 2} q^{-2\omega}; q^{-2\omega})(q^{-2\lambda - 2}; q^{-2\omega})} \frac{(q^{-4}; q^{-2\omega})}{(q^{-2\omega}; q^{-2\omega})}.
\]
Substituting back into our computation yields
\begin{align*}
\Xi(q, \lambda, \omega, \mu, k) &= \frac{1}{q - q^{-1}} \frac{(q^{-2 \omega + 2}; q^{-2\omega}, q^{-2\kappa})^2}{(q^{-2 \omega - 2}; q^{-2 \omega}, q^{-2\kappa})^2} (q^{-2\omega}; q^{-2\omega}) q^{2\lambda(\mu + 1) + 1}z^{\frac{2\mu + 2}{\kappa}} C_1(\lambda, \mu)\\
&\phantom{===} \int_0^{z q^{-2} \cdot \infty} \frac{d_pt}{t} t^{- \frac{2(\mu + 1)}{k + 2}} \Omega_{q^2}(\frac{t}{z}; q^{-2\omega}, q^{-2\kappa}) \frac{\theta_0(\frac{t}{z} q^{2}; q^{-2\kappa})}{\theta_0(\frac{t}{z}q^{-2};q^{-2\kappa})}\frac{\theta_0(\frac{z}{t} q^{-2\lambda}; q^{-2\omega})}{\theta_0(\frac{z}{t} q^{2}; q^{-2\omega})},
\end{align*}
which upon noting that
\begin{align*}
C(\lambda, \mu) = \frac{1}{q - q^{-1}} \frac{(q^{-2 \omega + 2}; q^{-2\omega}, q^{-2\kappa})^2}{(q^{-2 \omega - 2}; q^{-2 \omega}, q^{-2\kappa})^2} (q^{-2\omega}; q^{-2\omega}) q^{2\lambda \mu - 1} C_1(\lambda, \mu)
\end{align*}
simplifies to
\begin{align*}
\Xi(q, \lambda, \omega, \mu, k) &=q^{2\lambda + 2} C(\lambda, \mu) z^{\frac{2\mu + 2}{\kappa}} \int_0^{z q^{-2} \cdot \infty} \frac{d_pt}{t} t^{- \frac{2(\mu + 1)}{\kappa}}  \Omega_{q^2}(\frac{t}{z}; q^{-2\omega}, q^{-2\kappa}) \frac{\theta_0(\frac{t}{z} q^{2}; q^{-2\kappa})\theta_0(\frac{z}{t} q^{- 2\lambda}; q^{-2\omega})}{\theta_0(\frac{t}{z}q^{-2};q^{-2\kappa}) \theta_0(z/tq^2; q^{-2\omega})} \\
&= C(\lambda, \mu) z^{\frac{2\mu + 2}{\kappa}}\int_0^{z q^{-2} \cdot \infty} \frac{d_pt}{t} t^{- \frac{2(\mu + 1)}{\kappa}}  \Omega_{q^2}(\frac{t}{z}; q^{-2\omega}, q^{-2\kappa}) \frac{\theta_0(\frac{t}{z}q^{2}; q^{-2\kappa})\theta_0(\frac{t}{z} q^{2\lambda}; q^{-2\omega})}{\theta_0(\frac{t}{z}q^{-2};q^{-2\kappa}) \theta_0(\frac{t}{z}q^{-2}; q^{-2\omega})}\\
&= C(\lambda, \mu) \int_0^{q^{-2} \cdot \infty} \frac{d_pt}{t} t^{- \frac{2(\mu + 1)}{\kappa}} \Omega_{q^2}(t; q^{-2\omega}, q^{-2\kappa}) \frac{\theta_0(tq^{2}; q^{-2\kappa})\theta_0(t q^{2\lambda}; q^{-2\omega})}{\theta_0(t q^{-2};q^{-2\kappa}) \theta_0(tq^{-2}; q^{-2\omega})},
\end{align*}
where in the last step we change variables to eliminate $z$. This completes the proof of Proposition \ref{prop:ff-comp}.

\subsection{Jackson integral expression for the trace function}

We now apply Lemma \ref{lem:trace-relation} to obtain a Jackson integral expression for $T^{w_0}(q, \lambda, \omega, \mu, k)$.  

\begin{prop} \label{prop:jack-int-t}
In the doubly formal good region of parameters (\ref{eq:dformal-good-region}), as a formal Jackson integral and a formal series in $q^{-2\omega}$ we have 
\[
T^{w_0}(q, \lambda, \omega, \mu, k) = D(\lambda, \mu) \int_0^{q^{-2} \cdot \infty} \frac{d_{q^{-2k}}t}{t}\Omega_{q^2}(t; q^{-2\omega}, q^{-2k})  t^{-\frac{\mu + 1}{k}} \frac{\theta_0(tq^2; q^{-2k})}{\theta_0(tq^{-2}; q^{-2k})} \frac{\theta_0(tq^{2\lambda}; q^{-2\omega})}{\theta_0(tq^{-2}; q^{-2\omega})},
\]
where the constant $D(\lambda, \mu)$ is defined by
\begin{multline*}
D(\lambda, \mu) =
 \frac{(q^4 q^{-2k}; q^{-2k})(q^{-4}; q^{-2\omega})}{(q^{-2k};q^{-2k})} \frac{(q^{-2 \omega + 2}; q^{-2\omega}, q^{-2 k})^2}{(q^{-2 \omega - 2}; q^{-2 \omega}, q^{-2k})^2}\\
 \frac{q^{\lambda\mu - \lambda + 2 - \frac{2\mu+2}{k}} }{\theta_0(q^{2\lambda}; q^{-2\omega}) (q^{2\lambda - 2} q^{-2\omega}; q^{-2\omega})(q^{-2\lambda - 2}; q^{-2\omega})} \frac{(q^{-2\mu-2}; q^{-2k})}{(q^{-2\mu + 2};q^{-2k})}.
\end{multline*}
\end{prop}

\begin{remark}
In the statement of Proposition \ref{prop:jack-int-t}, the functions 
\[
f(t) = t^{-\frac{\mu + 1}{k}} \frac{\theta_0(tq^2; q^{-2k})}{\theta_0(tq^{-2}; q^{-2k})} \qquad \text{ and } \qquad g(t) = \frac{\theta_0(tq^{-2\mu}; q^{-2k})}{\theta_0(t q^{-2}; q^{-2k})}
\]
have the same transformation properties under $q^{-2k}$-shifts, meaning that $\frac{f(q^{-2k}t)}{f(t)} = \frac{g(q^{-2k}t)}{g(t)} = q^{2\mu - 2}$.
\end{remark}

\begin{proof}[Proof of Proposition \ref{prop:jack-int-t}]
By Lemma \ref{lem:trace-relation}, we have 
\begin{equation} \label{eq:trace-const-norm}
T^{w_0}(q, \lambda, \omega, \mu, k) = C_{\frac{\mu - 1}{2}, 1}^{-1} \Xi\Big(q, \lambda, \omega, \frac{\mu - 1}{2}, k - 2\Big).
\end{equation}
By Proposition \ref{prop:ff-comp}, we find that 
\[
\Xi\Big(q, \lambda, \omega, \frac{\mu - 1}{2}, k - 2\Big) = D_1(\lambda, \mu) \int_0^{q^{-2} \cdot \infty} \frac{d_{q^{-2k}}t}{t}\Omega_{q^2}(t; q^{-2\omega}, q^{-2k})  t^{-\frac{\mu + 1}{k}} \frac{\theta_0(tq^2; q^{-2k})}{\theta_0(tq^{-2}; q^{-2k})} \frac{\theta_0(tq^{2\lambda}; q^{-2\omega})}{\theta_0(tq^{-2}; q^{-2\omega})},
\]
where the constant $D_1(\lambda, \mu)$ is defined by 
\begin{align*}
D_1(\lambda, \mu) &:= C(\lambda, \frac{\mu - 1}{2}) = \frac{1}{q - q^{-1}} \frac{q^{\lambda\mu - \lambda - \mu - 1}(q^{-4}; q^{-2\omega}) }{\theta_0(q^{2\lambda}; q^{-2\omega})(q^{2\lambda - 2}q^{-2\omega}; q^{-2\omega})(q^{-2\lambda - 2}; q^{-2\omega})}\frac{(q^{-2 \omega + 2}; q^{-2\omega}, q^{-2 k})^2}{(q^{-2 \omega - 2}; q^{-2 \omega}, q^{-2k})^2}.
\end{align*}
Further, by Proposition \ref{prop:inter-mat-elt}, we have that 
\[
C_{\frac{\mu - 1}{2}, 1} = -(1 + q^2)q^{\frac{2\mu + 2}{k} - 2 - \mu}  \frac{(q^{-2k};q^{-2k})}{(q^{-2\mu-2}; q^{-2k})} \frac{(q^{-2\mu + 2};q^{-2k})}{(q^4; q^{-2k})}.
\]
Substituting these into (\ref{eq:trace-const-norm}) yields the desired.
\end{proof}

\subsection{Relating the Jackson integral to the Felder-Varchenko function}

In this section, we convert the Jackson integral expression of Proposition \ref{prop:jack-int-t} to a contour integral expression to give a proof of Theorem \ref{thm:int-trace}.  For this, we show in Proposition \ref{prop:int-formal} that when formally expanded in $q^{-2\omega}$, the coefficients of both the Jackson and contour integrals converge in a formal neighborhood of $0$ as functions of $q^{-2\mu}$.

\begin{prop} \label{prop:int-formal}
In the doubly formal good region of parameters (\ref{eq:dformal-good-region}), as a formal power series in $q^{-2\omega}$, the Jackson integral for $T^{w_0}(q, \lambda, \omega, \mu, k)$ converges in a formal neighborhood of $0$ in the variable $q^{-2\mu}$ and equals
\begin{align*}
T^{w_0}(q, \lambda, \omega, \mu, k)
&= \frac{(q^{-4}; q^{-2\omega})}{(q^{-2\omega}; q^{-2\omega})} \frac{(q^{-2 \omega + 2}; q^{-2\omega}, q^{-2 k})^2}{(q^{-2 \omega - 2}; q^{-2 \omega}, q^{-2k})^2}\frac{(q^{-4}q^{-2\omega}; q^{-2\omega}, q^{-2k})}{(q^4 q^{-2\omega} q^{-2k}; q^{-2k}, q^{-2\omega})}\\
&\phantom{==} \frac{q^{\lambda\mu - \lambda + 2}}{\theta_0(q^{2\lambda}; q^{-2\omega})(q^{2\lambda - 2} q^{-2\omega}; q^{-2\omega})(q^{-2\lambda - 2}; q^{-2\omega})} \frac{(q^{-2\mu-2}; q^{-2k})}{(q^{-2\mu + 2};q^{-2k})}  \\
&\phantom{==}
\sum_{n \geq 0} q^{(-2\mu + 2)n} \frac{\theta_0(q^{2\lambda - 2} q^{2kn}; q^{-2\omega})}{\theta_0(q^{-4}q^{2kn}; q^{-2\omega})}\prod_{l = 1}^n \frac{\theta_0(q^{-4} q^{2kl}; q^{-2\omega})}{\theta_0(q^{2kl}; q^{-2\omega})}.
\end{align*}
\end{prop}
\begin{proof}
Define the integral expression
\[
I = \int_0^{q^{-2} \cdot \infty} \frac{d_{q^{-2k}}t}{t}\Omega_{q^2}(t; q^{-2\omega}, q^{-2k})  t^{-\frac{\mu + 1}{k}} \frac{\theta_0(tq^2; q^{-2k})}{\theta_0(tq^{-2}; q^{-2k})} \frac{\theta_0(tq^{2\lambda}; q^{-2\omega})}{\theta_0(tq^{-2}; q^{-2\omega})}
\]
so that $T^{w_0}(q, \lambda, \omega, \mu, k) = D(\lambda, \mu) I$ by Proposition \ref{prop:jack-int-t}.  Denote the integrand by 
\[
J(t) = \Omega_{q^2}(t; q^{-2\omega}, q^{-2k})  t^{-\frac{\mu + 1}{k}} \frac{\theta_0(tq^2; q^{-2k})}{\theta_0(tq^{-2}; q^{-2k})} \frac{\theta_0(tq^{2\lambda}; q^{-2\omega})}{\theta_0(tq^{-2}; q^{-2\omega})}
\]
and let its formal power series expansion in $q^{-2\omega}$ be $J(t) = \sum_{n \geq 0} J_n(t) q^{-2\omega n}$.  We have that 
\begin{align*}
J_0(t) &= \frac{(t q^{-2}; q^{-2k})}{(tq^2; q^{-2k})}  t^{-\frac{\mu + 1}{k}} \frac{\theta_0(tq^2; q^{-2k})}{\theta_0(tq^{-2}; q^{-2k})} \frac{(1 - tq^{2\lambda})}{(1 - tq^{-2})} = t^{- \frac{\mu + 1}{k}} \frac{(t^{-1}q^{-2} q^{-2k}; q^{-2k})}{(t^{-1}q^2 q^{-2k}; q^{-2k})} \frac{(1 - tq^{2\lambda})}{(1 - tq^{-2})}.
\end{align*}
We conclude that $\int_0^{q^{-2} \cdot \infty} \frac{d_{q^{-2k}} t}{t} J_0(t) t^n$ converges for $\frac{\mu + 1}{k} > n$, which holds in the doubly formal good region of parameters (\ref{eq:dformal-good-region}) on a formal neighborhood of $0$ in $q^{-2\mu}$.  Defining the formal power series $\wJ(t)$ so that $J(t) = J_0(t) \wJ(t)$, we see that $\wJ(t)$ has expansion of the form
\[
\wJ(t) = 1 + \sum_{n > 0} \wJ_n(t) q^{-2\omega n},
\]
where $\wJ_n(t)$ is a Laurent series in $t$ of degree at most $n$.  This implies that $\int_0^{q^{-2} \cdot \infty} \frac{d_{q^{-2k}} t}{t} J_n(t)$ converges for each $n$.  We conclude that each coefficient in the formal power series expansion of $I = \int_0^{q^{-2} \cdot \infty} \frac{d_{q^{-2k}} t}{t} J(t)$ converges in a formal neighborhood of $0$, so as a formal power series in $q^{-2\omega}$ it equals
\begin{align*}
I &= \sum_{n \geq 0} J(q^{-2} q^{2kn})\\
&= \sum_{n \geq 0} (q^{-2} q^{2kn})^{-\frac{\mu +1}{k}} \frac{(q^{-4} q^{2kn}; q^{-2\omega}, q^{-2k})(q^{-2kn}q^{-2\omega} q^{-2k}; q^{-2\omega}, q^{-2k})}{(q^{2kn}q^{-2\omega}; q^{-2\omega}, q^{-2k})(q^4q^{-2kn} q^{-2\omega}q^{-2k}; q^{-2\omega}, q^{-2k})}\\
&\phantom{=================================} \frac{(q^{-2kn} q^{-2k}; q^{-2k})}{\theta_0(q^{-4}q^{2kn}; q^{-2k})} \frac{\theta_0(q^{2\lambda - 2} q^{2kn}; q^{-2\omega})}{\theta_0(q^{-4}q^{2kn}; q^{-2\omega})}\\
&= q^{\frac{2\mu + 2}{k}}\frac{(q^{-4}; q^{-2\omega}, q^{-2k})}{(q^{-2\omega}; q^{-2\omega})(q^4 q^{-2\omega} q^{-2k}; q^{-2k}, q^{-2\omega})} \frac{(q^{-2k}; q^{-2k})}{\theta_0(q^{-4}; q^{-2k})} \sum_{n \geq 0} q^{(-2\mu + 2)n}\\
&\phantom{=================================} \frac{\theta_0(q^{2\lambda - 2} q^{2kn}; q^{-2\omega})}{\theta_0(q^{-4}q^{2kn}; q^{-2\omega})}\prod_{l = 1}^n \frac{\theta_0(q^{-4} q^{2kl}; q^{-2\omega})}{\theta_0(q^{2kl}; q^{-2\omega})},
\end{align*} 
which yields the desired upon substitution of $D(\lambda, \mu)$. 
\end{proof}

We now prove Theorem \ref{thm:int-trace} and thus connect the trace function and the Felder-Varchenko function.

\begin{proof}[Proof of Theorem \ref{thm:int-trace}]

Comparing Proposition \ref{prop:fv-series} in the formal neighborhood of $\infty$ for $q^{-2\mu}$ and Proposition \ref{prop:int-formal} in the formal neighborhood of $0$ for $q^{-2\mu}$, we find as formal power series in $q^{-2\omega}$ and then $q^{-2\mu}$, in the doubly formal good region of parameters (\ref{eq:dformal-good-region}) we have
\begin{multline} \label{eq:fv-iden}
T^{w_0}(q, \lambda, \omega, \mu, k) = \frac{q^{-\mu + 4} (q^{-4}; q^{-2\omega})}{\theta_0(q^{2\lambda}; q^{-2\omega})(q^{2\lambda - 2} q^{-2\omega}; q^{-2\omega})(q^{-2\lambda - 2}; q^{-2\omega})} \\
 \frac{(q^{-2 \omega + 2}; q^{-2\omega}, q^{-2 k})^2}{(q^{-2 \omega - 2}; q^{-2 \omega}, q^{-2k})^2} \frac{ (q^{-2k};q^{-2k})(q^4 q^{-2k}; q^{-2k})}{(q^{-2\mu + 2};q^{-2k})(q^{2\mu + 2} q^{-2k}; q^{-2k})}u(q, \lambda, \omega, -\mu, k). 
\end{multline}
Recalling the definition (\ref{eq:fv-def2}) of the Felder-Varchenko function, we obtain 
\begin{multline*}
T^{w_0}(q, \lambda, \omega, \mu, k) = (q^{-4}; q^{-2\omega}) \frac{(q^{-2 \omega + 2}; q^{-2\omega}, q^{-2 k})^2}{(q^{-2 \omega - 2}; q^{-2 \omega}, q^{-2k})^2} \frac{q^{\lambda\mu - \lambda + 2}}{\theta_0(q^{2\lambda}; q^{-2\omega})(q^{2\lambda - 2} q^{-2\omega}; q^{-2\omega})(q^{-2\lambda - 2}; q^{-2\omega})} \\
 \frac{(q^{-2k};q^{-2k})(q^4 q^{-2k}; q^{-2k})}{(q^{-2\mu + 2};q^{-2k})(q^{2\mu + 2} q^{-2k}; q^{-2k})} \oint_{\cC_t} \frac{dt}{2\pi it} \Omega_{q^2}(t; q^{-2\omega}, q^{-2k}) \frac{\theta_0(tq^{-2\mu}; q^{-2k})}{\theta_0(tq^{-2}; q^{-2k})} \frac{\theta_0(tq^{2\lambda}; q^{-2\omega})}{\theta_0(tq^{-2}; q^{-2\omega})}.
\end{multline*}
By Proposition \ref{prop:rat-fn-trace}, each coefficient of $q^{-\lambda \mu} T^{w_0}(q, \lambda, \omega, \mu, k)$ as a formal power series in $q^{-2\omega}$ is a rational function in $q^{-2\mu}$ and $q^{-2k}$.\footnote{Note that this does not follow from the Jackson integral expression for $T^{w_0}(q, \lambda, \omega, \mu, k)$, since as a formal series in $q^{-2\omega}$, that expression holds only on a formal neighborhood of $0$ in $q^{-2\mu}$.  To obtain simultaneous rationality of coefficients of the formal series in $q^{-2\omega}$ on an open neighborhood of $0$ in $q^{-2\mu}$, we require the representation-theoretic input of Proposition \ref{prop:rat-fn-trace}.}  Therefore, the coefficients of the formal power series in $q^{-2\omega}$ in (\ref{eq:fv-iden}) are equal as rational functions because they are equal on a formal neighborhood of $0$ in $q^{-2\mu}$ by (\ref{eq:fv-iden}).  Define now the \textit{formal good region} of parameters
\begin{equation} \label{eq:formal-good-region}
0 < |q^{-2\mu}| \ll |q^{-2\lambda}| \ll |q^{-2k}| \ll |q|, |q|^{-1}.
\end{equation}
We conclude that (\ref{eq:fv-iden}) holds in the formal good region of parameters at the level of formal power series in $q^{-2\omega}$ for numerical $q^{-2\mu}$ sufficiently close to $0$.

It remains only to check that this formal series in $q^{-2\omega}$ converges.  Because the poles of the integrand in the integral expression are bounded uniformly away from $\cC_t$ as $q^{-2\omega} \to 0$, the formal power series expansion of the integrand converges uniformly on the compact cycle $\cC_t$ to the integrand.  We conclude that the integral may be integrated term-wise and therefore that its formal expansion in $q^{-2\omega}$ converges.  Because $T^{w_0}(q, \lambda, \omega, \mu, k)$ shares a formal expansion with the integral, it also converges for numerical $q^{-2\omega}$ sufficiently close to $0$.  We conclude that the desired equality holds for numerical parameters in the good region (\ref{eq:good-region}) with $q^{-2\mu}$ and then $q^{-2\omega}$ sufficiently close to $0$, completing the proof.
\end{proof}

\begin{corr} \label{corr:trace-fv}
For $q^{-2\mu}$ and then $q^{-2\omega}$ sufficiently close to $0$ in the good region of parameters (\ref{eq:good-region}), the trace $T^{w_0}(q, \lambda, \omega, \mu, k)$ is related to the Felder-Varchenko function by 
\begin{multline*}
T^{w_0}(q, \lambda, \omega, \mu, k) =  \frac{q^{-\mu + 4} (q^{-4}; q^{-2\omega})}{\theta_0(q^{2\lambda}; q^{-2\omega})(q^{2\lambda - 2} q^{-2\omega}; q^{-2\omega})(q^{-2\lambda - 2}; q^{-2\omega})} \\
 \frac{(q^{-2 \omega + 2}; q^{-2\omega}, q^{-2 k})^2}{(q^{-2 \omega - 2}; q^{-2 \omega}, q^{-2k})^2} \frac{ (q^{-2k};q^{-2k})(q^4 q^{-2k}; q^{-2k})}{(q^{-2\mu + 2};q^{-2k})(q^{2\mu + 2} q^{-2k}; q^{-2k})}u(q, \lambda, \omega, -\mu, k). 
\end{multline*}
\end{corr}
\begin{proof}
This follows by combining Theorem \ref{thm:int-trace} and the formal equality (\ref{eq:fv-iden}).
\end{proof}

\section{The classical limit} \label{sec:class}

In this section, we take the classical limit of our expression for the trace of a $U_q(\asl_2)$-intertwiner and recover the expression for the trace of a $U(\asl_2)$-intertwiner given in \cite{EK2}.

\subsection{Verma modules, evaluation modules, and intertwiners}

Let $M_{\mu, k}$ denote the Verma module for $U(\asl_2)$ with highest weight $\mu\rho + k\Lambda_0$, and let $L_{\mu}(z)$ denote the finite-dimensional evaluation module with highest weight $\mu$.  The module $L_\mu(z)$ has a basis $w_\mu, w_{\mu - 2}, \ldots, w_{-\mu}$ which coincides with the basis for $L_\mu(z)$ as a $U_q(\asl_2)$-module.  In particular, it satisfies
\[
e_1 w_m \otimes z^n = \frac{\mu - m}{2} w_{m + 2} \otimes z^n.
\]
For $v \in L_{\mu}(z)[0]$, denote by 
\[
\Phi^{v, \cl}_{\mu, k}(z): M_{\mu, k} \to M_{\mu, k} \otimes L_{\mu}(z)
\]
the corresponding $U(\asl_2)$-intertwiner, normalized so that 
\[
\Phi^{v, \cl}_{\mu, k}(z) \cdot v_{\mu, k} = v_{\mu, k} \otimes v + (\text{l.o.t.}).
\]

\subsection{Wakimoto realization of intertwiners}

In \cite[Section 5]{EK2}, a different normalization of this intertwiner is considered via the Wakimoto realization.  Let $A$ be the algebra generated by $\alpha_n, \beta_n, \gamma_n$ for $n \in \ZZ$ subject to the relations
\[
[\alpha_n, \alpha_m] = 2n \delta_{n + m, 0} \qquad \text{ and } \qquad [\beta_n, \gamma_m] = \delta_{n + m, 0}
\]
with all other commutators zero.  Define also the bosonic fields
\[
\alpha(z) = \sum_{n \in \ZZ} \alpha_n z^{-n-1}, \qquad \beta(z) = \sum_{n \in \ZZ} \beta_n z^{-n-1}, \qquad \gamma(z) = \sum_{n \in \ZZ} \gamma_n z^{-n}.
\]
Define the Fock module $H_{\lambda, k}$ to be generated by the vacuum vector $v_{\lambda, k}$ with properties 
\[
\alpha_n v_{\lambda, k} = 0 \text{ for $n > 0$} \qquad \beta_n v_{\lambda, k}= \gamma_{n + 1} v_{\lambda, k} = 0 \text{ for $n \geq 0$} \qquad \alpha_0 v_{\lambda, k} = \frac{\lambda}{\sqrt{k + 2}} v_{\lambda, k}.
\]
A free field realization of the Verma module was given in terms of these Fock modules by Wakimoto in \cite{Wak}.  In our notation, this realization is given as follows.

\begin{prop}[{\cite[Equation 5.7]{EK2}}] \label{prop:class-wakimoto}
The assignments 
\begin{align*}
J_e(z) &= \beta(z)\\
J_h(z) &= - 2 : \beta(z) \gamma(z): + \sqrt{\kappa}\alpha(z)\\
J_f(z) &= - :\gamma(z)^2 \beta(z): + \sqrt{\kappa}\alpha(z)\gamma(z) + (\kappa - 2) \gamma'(z)
\end{align*}
define an action of $U(\asl_2)$ on $H_{\lambda, k}$.  For generic $\lambda$ and $k$, the resulting module is isomorphic to the Verma module $M_{\lambda, k}$ with $v_{\lambda, k}$ the highest weight vector.
\end{prop}

Define now the classical vertex operator
\[
X(c \alpha, z) = \exp\left(c \sum_{n < 0} \frac{\alpha_n}{-n} z^{-n}\right)\exp\left(c \sum_{n > 0} \frac{\alpha_n}{-n} z^{-n}\right) e^{c\alpha} z^{c\alpha_0}
\]
and the screening operator
\[
U(t) = \beta(t) X(-\alpha/\sqrt{\kappa}, t).
\]
These two operators define the intertwining operator which appears in \cite[Section 5]{EK2}.

\begin{prop}[{\cite[Equation 5.8]{EK2}}] \label{prop:class-inter}
For any cycle $\Delta$ on which $\log(t_i)$, $\log(t_i - t_j)$, and $\log(t_i - z)$ are well-defined, the operator $\hPhi_\mu(z) : M_{\mu, k} \to M_{\mu, k} \hotimes L_{2m}(z)$ given by 
\[
\hPhi_\mu(z) v = z^{\frac{m(m+1)}{k+2}}\sum_{n = 0}^{2m + 1} \left(\int_\Delta X(\frac{m}{\sqrt{k + 2}} \alpha, z) (-\gamma(z))^n U(t_1) \cdots U(t_m) dt_1 \cdots dt_m\right) v \otimes \frac{e_1^n}{n!} w_{-2m},
\]
is an intertwining operator of $U(\asl_2)$-representations.
\end{prop}
\begin{remark}
Proposition \ref{prop:class-inter} has a different normalization from \cite[Section 5]{EK2} because we take $d$ to act by $0$ on $M_{\lambda, k}$.
\end{remark}

\subsection{The Etingof-Kirillov~Jr. expression for the classical trace function}

We now apply the computation of \cite[Theorem 5.1]{EK2} to compute in Corollary \ref{corr:norm-tr2} the classical trace function
\[
\Tr|_{M_{\mu - 1, k - 2}}\Big(\Phi^{w_0, \cl}_{\mu - 1, k - 2} q^{2d} e^{\pi i \lambda h}\Big).
\]
We begin by defining the normalized traces which appear in \cite[Theorem 5.1]{EK2}.  For $q = e^{\pi i \tau}$, define the function 
\[
F(\lambda, \tau, \mu, k) = \frac{\Tr|_{M_{\mu, k}}\Big(\hPhi_\mu q^{-2d - \frac{h}{2}} e^{-\pi i \lambda h}\Big)}{\Tr|_{M_{\frac{k}{2}, k}}\Big(q^{-2d - \frac{h}{2}} e^{-\pi i \lambda h}\Big)}
\]
and its normalized version given by 
\[
\wF(\lambda, \tau, \mu, k) := F(-\lambda + \tau/2, -\tau, \mu - 1, k - 2) = \frac{\Tr|_{M_{\mu - 1, k - 2}}\Big(\hPhi_{\mu - 1} q^{2d} e^{\pi i \lambda h}\Big)}{\Tr|_{M_{\frac{k-2}{2}, k - 2}}\Big(q^{2d} e^{\pi i \lambda h}\Big)}.
\]
An integral expression was given in \cite{EK2} for $F(\lambda, \tau, \mu, k)$ using the additive notation
\begin{equation} \label{eq:jac-add}
\theta_1(\zeta \mid \tau) = 2 e^{\pi i \tau/4} \sin(\zeta) \prod_{n \geq 1} (1 - e^{2\pi i n \tau} e^{2i \zeta}) (1 - e^{2\pi i n \tau } e^{-2i \zeta})(1 - e^{2\pi i n\tau}) = \theta(e^{2i \zeta}; e^{2\pi i \tau})
\end{equation}
for the first Jacobi theta function.

\begin{thm}[{\cite[Theorem 5.1]{EK2}}] \label{thm:ek-class}
For $\mu \in \CC$ and $\kappa = k + 2$ with $\kappa \neq 0$, $\Imm(\tau) > 0$, $\Imm(\lambda) > 0$, and $m \geq 0$, we have 
\begin{align*}
F(\lambda, \tau, \mu, k) &= e^{-\pi i \lambda (\mu - \frac{k}{2})} z^{\frac{m(\mu -m)}{\kappa}} \int_\Delta \prod_{i = 1}^m t_i^{(-1 - \mu)/\kappa} \prod_{i = 1}^m \left(\frac{\theta_1(\pi (\zeta_0 - \zeta_i) \mid \tau)}{\theta_1'(0 \mid \tau)}\right)^{-2m/\kappa}\\
& \prod_{i < j} \left(\frac{\theta_1(\pi(\zeta_i - \zeta_j) \mid \tau)}{\theta_1'(0 \mid \tau)}\right)^{2/\kappa} m! \prod_{i = 1}^m G(z, t_i \mid \tau, \lambda + \tau/2) dt_1 \cdots dt_m,
\end{align*}
where $\Delta$ is a cycle chosen so that $\log t_i$, $\log (t_i - t_j)$, and $\log(t_i - z)$ have single-valued branches on $\Delta$, the function $G$ is defined by 
\[ 
G(1, e^{2\pi i \zeta} \mid \tau, \lambda) = \frac{i}{2} e^{-2\pi i \zeta} \frac{\theta_1(\pi(\lambda - \zeta) \mid \tau) \,\theta_1'(0 \mid \tau)}{\theta_1(\pi \lambda \mid \tau) \, \theta_1(\pi \zeta \mid \tau)},
\]
and we take the coordinates $z = e^{2\pi i \zeta_0}$, $t_i = e^{2\pi i \zeta_i}$, and $q = e^{\pi i \tau}$. 
\end{thm}

\begin{remark}
In \cite[Theorem 5.1]{EK2}, $d$ is normalized to act on $v_{\mu, k}$ by $-\frac{\mu^2}{2(k + 2)}$. In this work we normalize $d$ to act on $v_{\mu, k}$ by $0$, hence the statement of Theorem \ref{thm:ek-class} differs in normalization from the statement in \cite[Theorem 5.1]{EK2}.
\end{remark}

We now specialize to $m = 1$ and set $z = 1$ (since $F(\lambda, \tau, \mu, k)$ is independent of $z$).  After computing all relevant normalizations, we obtain in Corollary \ref{corr:norm-tr2}, the classical limit of $T^{w_0}(q, \lambda, \omega, \mu, k)$ from Theorem \ref{thm:int-trace}. 

\begin{corr} \label{corr:ek-class1}
For $\mu, \kappa \in \CC$, $\kappa \neq 0$, $|q| > 1$, $\Imm(\lambda) < 0$, and $m = 1$, we have 
\[
\wF(\lambda, \tau, \mu, k) = 2^{\frac{2}{k}} e^{-\frac{3\pi i}{k}} e^{\pi i \lambda (\mu - \frac{k}{2})}  \frac{(q^{-2}; q^{-2})^{\frac{2k + 4}{k}}}{\theta_0(e^{2\pi i \lambda}; q^{-2})} \int_\Delta t^{-\frac{\mu - 1}{k} - 1} \frac{\theta_0(t e^{2\pi i \lambda}; q^{-2})}{\theta_0(t; q^{-2})^{\frac{k + 2}{k}}}dt,
\]
where $\Delta$ is a Pochhammer loop around $0$ and $1$.
\end{corr}
\begin{proof}
Direct substitution and setting $z = 1$ in Theorem \ref{thm:ek-class} allows us to choose $\Delta$ to be a Pochhammer loop around $0$ and $1$ and yields
\begin{align*}
F(\lambda, \tau, \mu, k) &= e^{-\pi i \lambda (\mu - \frac{k}{2})} \int_\Delta t^{- \frac{\mu + 1}{\kappa}}\left(\frac{\theta_1(-\pi\zeta_1 \mid \tau))}{\theta_1'(0 \mid \tau)}\right)^{-\frac{2}{\kappa}}\frac{i}{2} t^{-1} \frac{\theta_1(\pi(\lambda + \tau/2 - \zeta_1) \mid \tau) \theta_1'(0 \mid \tau)}{\theta_1(\pi (\lambda + \tau/2) \mid \tau) \theta_1(\pi \zeta_1 \mid \tau)}dt\\
&=  e^{-\pi i \frac{2}{\kappa}} e^{-\pi i \lambda (\mu - \frac{k}{2})} \frac{i}{2}\int_\Delta t^{-\frac{\mu + 1}{\kappa} -1} \frac{\theta_1(\pi(\lambda + \tau/2 - \zeta_1) \mid \tau) \, \theta_1'(0 \mid \tau)^{\frac{\kappa + 2}{\kappa}}}{\theta_1(\pi (\lambda + \tau/2) \mid \tau) \, \theta_1(\pi \zeta_1 \mid \tau)^{\frac{\kappa + 2}{\kappa}}}dt\\
&=  e^{-\frac{2\pi i}{\kappa}} e^{-\pi i \lambda (\mu - \frac{k}{2})} \frac{i}{2} \frac{\theta_1'(0 \mid \tau)^{\frac{\kappa + 2}{\kappa}}}{\theta_1(\pi (\lambda + \tau/2) \mid \tau)}\int_\Delta t^{-\frac{\mu + 1}{\kappa} -1} \frac{\theta_1(\pi(\lambda + \tau/2 - \zeta_1) \mid \tau)}{\theta_1(\pi \zeta_1 \mid \tau)^{\frac{\kappa + 2}{\kappa}}}dt\\
&= e^{-\frac{2\pi i}{k + 2}} e^{-\pi i \lambda (\mu - \frac{k}{2})} \frac{i}{2} \frac{\theta_1'(0 \mid \tau)^{\frac{k + 4}{k + 2}}}{\theta_1(\pi (\lambda + \tau/2) \mid \tau)}\int_\Delta t^{-\frac{\mu + 1}{k + 2} -1} \frac{\theta_1(\pi(\lambda + \tau/2 - \zeta_1) \mid \tau)}{\theta_1(\pi \zeta_1 \mid \tau)^{\frac{k + 4}{k + 2}}}dt.
\end{align*}
We now compute the derivative
\[
\theta_1'(0 \mid \tau) = -2 e^{\pi i \tau/4} (q^2; q^2) \theta_0'(1; q^2) = 2 e^{\pi i \tau/4} (q^2; q^2)^3.
\]
Substituting and applying (\ref{eq:jac-add}) and (\ref{eq:theta-def}), we conclude
\begin{align*}
\wF(\lambda, \tau, \mu, k) &= e^{-\frac{2\pi i}{k}} e^{\pi i \lambda (\mu - 1 - \frac{k - 2}{2})} \frac{i}{2} \frac{\theta_1'(0 \mid \tau)^{\frac{k + 2}{k}}}{\theta_1(-\pi \lambda \mid \tau)}\int_\Delta t^{-\frac{\mu}{k} -1} \frac{\theta_1(\pi(-\lambda - \zeta_1) \mid \tau)}{\theta_1(\pi \zeta_1 \mid \tau)^{\frac{k + 2}{k}}}dt\\
&= e^{-\frac{2\pi i}{k}} e^{\pi i \lambda (\mu - \frac{k}{2})} \frac{i}{2} 2^{\frac{k + 2}{k}} e^{-\pi i \frac{k + 2}{2k}} \frac{(q^{-2}; q^{-2})^{\frac{2k + 4}{k}}}{\theta_0(e^{2\pi i \lambda}; q^{-2})} \int_\Delta t^{-\frac{\mu - 1}{k} - 1} \frac{\theta_0(t e^{2\pi i \lambda}; q^{-2})}{\theta_0(t; q^{-2})^{\frac{k + 2}{k}}}dt\\
&= 2^{\frac{2}{k}} e^{-\frac{3\pi i}{k}} e^{\pi i \lambda (\mu - \frac{k}{2})}  \frac{(q^{-2}; q^{-2})^{\frac{2k + 4}{k}}}{\theta_0(e^{2\pi i \lambda}; q^{-2})} \int_\Delta t^{-\frac{\mu - 1}{k} - 1} \frac{\theta_0(t e^{2\pi i \lambda}; q^{-2})}{\theta_0(t; q^{-2})^{\frac{k + 2}{k}}}dt. \qedhere
\end{align*}
\end{proof}

\begin{lemma} \label{lem:tr-norm}
For $|q| > 1$ and $\Imm(\lambda) < 0$, we have that 
\[
\Tr|_{M_{\frac{k}{2}, k}}\Big(q^{2d} e^{\pi i \lambda h}\Big) = e^{\frac{\pi i \lambda k}{2}}(q^2; q^2)^{-1} \theta_0(e^{-2\pi i \lambda}; q^2)^{-1}.
\]
\end{lemma}
\begin{proof}
Since negative roots take the form $-\alpha, \pm \alpha + m \delta, m \delta$ for $m < 0$, we obtain
\[
\Tr|_{M_{\frac{k}{2}, k}}\Big(q^{2d} e^{\pi i \lambda h}\Big) = e^{\pi i (\lambda h + 2 \tau d, \frac{k}{2}\rho + k \Lambda_0)} \prod_{\beta < 0} (1 - e^{\pi i (\lambda h + 2 \tau d, \beta)})^{-1}= e^{\frac{\pi i \lambda k}{2}}(q^{-2}; q^{-2})^{-1} \theta_0(e^{-2\pi i \lambda}; q^{-2})^{-1}. \qedhere
\]
\end{proof}

\begin{corr} \label{corr:norm-tr}
For $|q| > 1$ and $\Imm(\lambda) < 0$, we have that 
\[
\Tr|_{M_{\mu - 1, k - 2}}\Big(\hPhi_{\mu - 1} q^{2d} e^{\pi i \lambda h}\Big) = - 2^{\frac{2}{k}} e^{-\frac{3\pi i}{k}} e^{\pi i \lambda(\mu + 1)} \frac{(q^{-2}; q^{-2})^{\frac{k + 4}{k}}}{\theta_0(e^{2\pi i \lambda}; q^{-2})^2} \int_\Delta t^{-\frac{\mu - 1}{k} - 1} \frac{\theta_0(t e^{2\pi i \lambda}; q^{-2})}{\theta_0(t; q^{-2})^{\frac{k + 2}{k}}}dt,
\]
where $\Delta$ is a Pochhammer loop around $0$ and $1$.
\end{corr}
\begin{proof}
By Corollary \ref{corr:ek-class1} and Lemma \ref{lem:tr-norm}, we may compute 
\begin{align*}
\Tr|_{M_{\mu -1, k - 2}}\Big(\hPhi_{\mu - 1} q^{2d} e^{\pi i \lambda h}\Big) &= \wF(\lambda, \tau, \mu, k)\Tr|_{M_{\frac{k-2}{2}, k - 2}}\Big(q^{2d} e^{\pi i \lambda h}\Big)\\
&= 2^{\frac{2}{k}} e^{-\frac{3\pi i}{k}} e^{\pi i \lambda(\mu - 1)} \frac{(q^{-2}; q^{-2})^{\frac{k + 4}{k}}}{\theta_0(e^{2\pi i \lambda}; q^{-2})\theta_0(e^{-2\pi i \lambda}; q^{-2})} \int_\Delta t^{-\frac{\mu - 1}{k} - 1} \frac{\theta_0(t e^{2\pi i \lambda}; q^{-2})}{\theta_0(t; q^{-2})^{\frac{k + 2}{k}}}dt\\
&= -2^{\frac{2}{k}} e^{-\frac{3\pi i}{k}} e^{\pi i \lambda(\mu + 1)} \frac{(q^{-2}; q^{-2})^{\frac{k + 4}{k}}}{\theta_0(e^{2\pi i \lambda}; q^{-2})^2} \int_\Delta t^{-\frac{\mu - 1}{k} - 1} \frac{\theta_0(t e^{2\pi i \lambda}; q^{-2})}{\theta_0(t; q^{-2})^{\frac{k + 2}{k}}}dt. \qedhere
\end{align*}
\end{proof}

\begin{lemma}[Beta integral] \label{lem:beta-int}
If $\Delta$ is a Pochhammer loop around $0$ and $1$, then
\[
\int_\Delta t^{\alpha - 1}(1 - t)^{\beta - 1} dt = (1 - e^{2\pi i \alpha})(1 - e^{2\pi i \beta}) \frac{\Gamma(\alpha)\Gamma(\beta)}{\Gamma(\alpha + \beta)}.
\]
\end{lemma}

\begin{corr} \label{corr:norm-tr2}
For $|q| > 1$ and $\Imm(\lambda) < 0$, we have that 
\begin{multline*}
\Tr|_{M_{\mu -1, k-2}}\Big(\Phi_{\mu-1, k - 2}^{w_0, \cl} q^{2d} e^{\pi i \lambda h}\Big)\\
 =-\frac{e^{\pi i \lambda(\mu + 1)} e^{\pi i \frac{\mu - 1}{k}}}{1 - e^{-\frac{4\pi i}{k}}} \frac{\Gamma(\frac{\mu - 1}{k})\Gamma(\frac{k - \mu - 1}{k})}{\Gamma(-\frac{2}{k})} \frac{(q^{-2}; q^{-2})^{\frac{k + 4}{k}}}{\theta_0(e^{2\pi i \lambda}; q^{-2})^2} \int_\Delta \frac{1}{2\pi it} t^{-\frac{\mu - 1}{k}} \frac{\theta_0(t e^{2\pi i \lambda}; q^{-2})}{\theta_0(t; q^{-2})^{\frac{k + 2}{k}}}dt,
\end{multline*}
where $\Delta$ is a Pochhammer loop around $0$ and $1$.
\end{corr}
\begin{proof}
Because both $\hPhi_{\mu - 1}(z)$ and $\hPhi_{\mu - 1}^{w_0}(z)$ are intertwiners $M_{\mu - 1, k - 2} \to M_{\mu - 1, k - 2} \hotimes L_{2m}(z)$, $\hPhi_{\mu - 1}(z)$ is a scalar multiple of $\hPhi_{\mu - 1}^{w_0}(z)$.  Applying Lemma \ref{lem:beta-int}, the constant of proportionality is 
\begin{align*}
\langle v_{\mu - 1, k - 2}^*, \hPhi_{\mu - 1}(z) v_{\mu - 1, k - 2}\rangle &= \lim_{q \to \infty} \lim_{\lambda \to - i \infty} e^{-\pi i \lambda (\mu - 1)} \Tr|_{M_{\mu - 1, k - 2}}\Big(\hPhi_{\mu - 1} q^{2d} e^{\pi i \lambda h}\Big)\\
&= 2^{\frac{2}{k}} e^{-\frac{3\pi i}{k}} \int_\Delta t^{-\frac{\mu - 1}{k}} (1 - t)^{-\frac{k + 2}{k}}dt\\
&= 2^{\frac{2}{k}} e^{-\frac{3\pi i}{k}} (1 - e^{- 2\pi i \frac{\mu - 1}{k}})(1 - e^{-\frac{4\pi i}{k}}) \frac{\Gamma(\frac{k - \mu + 1}{k})\Gamma(-\frac{2}{k})}{\Gamma(\frac{k - \mu - 1}{k})}\\
&= 2^{\frac{2}{k}} e^{-\frac{3\pi i}{k}} e^{-\pi i \frac{\mu - 1}{k}} 2\pi i (1 - e^{-\frac{4\pi i}{k}}) \frac{\Gamma(-\frac{2}{k})}{\Gamma(\frac{\mu - 1}{k})\Gamma(\frac{k - \mu - 1}{k})}.
\end{align*}
Normalizing the result of Corollary \ref{corr:norm-tr} we conclude
\begin{multline*}
\Tr|_{M_{\mu -1, k-2}}\Big(\Phi_{\mu-1, k - 2}^{w_0, \cl} q^{2d} e^{\pi i \lambda h}\Big) = \frac{\Tr|_{M_{\mu -1, k-2}}\Big(\hPhi_{\mu-1} q^{2d} e^{\pi i \lambda h}\Big)}{2^{\frac{2}{k}} e^{-\frac{3\pi i}{k}} e^{-\pi i \frac{\mu - 1}{k}} 2\pi i (1 - e^{-\frac{4\pi i}{k}}) \frac{\Gamma(-\frac{2}{k})}{\Gamma(\frac{\mu - 1}{k})\Gamma(\frac{k - \mu - 1}{k})}}\\
= -\frac{e^{\pi i \lambda(\mu + 1)} e^{\pi i \frac{\mu - 1}{k}}}{1 - e^{-\frac{4\pi i}{k}}} \frac{\Gamma(\frac{\mu - 1}{k})\Gamma(\frac{k - \mu - 1}{k})}{\Gamma(-\frac{2}{k})} \frac{(q^{-2}; q^{-2})^{\frac{k + 4}{k}}}{\theta_0(e^{2\pi i \lambda}; q^{-2})^2} \int_\Delta \frac{1}{2\pi it} t^{-\frac{\mu - 1}{k}} \frac{\theta_0(t e^{2\pi i \lambda}; q^{-2})}{\theta_0(t; q^{-2})^{\frac{k + 2}{k}}}dt. \qedhere
\end{multline*}
\end{proof}

\subsection{The classical limit of the trace function}

We now obtain the result of Corollary \ref{corr:norm-tr2} as a consequence of Theorem \ref{thm:int-trace}.  For this, we take the classical limit of $T^{w_0}(q, \lambda, \omega, \mu, k)$, given by $\mu$ and $k$ fixed and 
\[
q = e^\eps \qquad \omega = \eps^{-1} \Omega \qquad \lambda = \eps^{-1} \Lambda
\]
as $\eps \to 0$.  Define the resulting limit of $T^{w_0}(q, \lambda, \omega, \mu, k)$ by 
\[
t(\Lambda, \Omega, \mu, k) := \lim_{\eps \to 0} T^{w_0}(e^\eps, \eps^{-1}\Lambda, \eps^{-1}\Omega, \mu, k).
\]

\begin{thm} \label{thm:class-lim}
If $\Lambda, \Omega, \mu, k$ are real with $-1 < \mu < 1$, the classical limit of $T^{w_0}(q, \lambda, \omega, \mu, k)$ is given by 
\[
t(\Lambda, \Omega, \mu, k)  = \frac{e^{\mu \Lambda + \Lambda} e^{\pi i \frac{\mu - 1}{k}}}{1 - e^{-\frac{4\pi i}{k}}} \frac{(e^{-2\Omega}; e^{-2\Omega})^{\frac{k + 4}{k}}}{\theta_0(e^{2\Lambda}; q^{-2\Omega})^2} \frac{\Gamma(\frac{\mu - 1}{k}) \Gamma(\frac{k - \mu - 1}{k})}{\Gamma(-\frac{2}{k})}\oint_{\Delta} \frac{dt}{2\pi it}\frac{\theta_0(t e^{2\Lambda}; e^{-2\Omega})}{\theta_0(t; e^{-2\Omega})^{\frac{k + 2}{k}}} t^{- \frac{\mu - 1}{k}},
\]
where $\Delta$ is a Pochhammer loop around $0$ and $1$.
\end{thm}

\begin{remark}
Theorem \ref{thm:class-lim} computes the classical trace function
\[
\Tr|_{M_{\mu - 1, k - 2}}\Big(\Phi^{w_0, \cl}_{\mu - 1, k - 2}(z) e^{2\Omega d} e^{2\Lambda \rho}\Big)
\]
and agrees with Corollary \ref{corr:norm-tr2} under the variable substitutions $q = e^\Omega$ and $\lambda = \frac{\Lambda}{\pi i}$ as expected.
\end{remark}

We require some elementary asymptotic lemmas for the proof of Theorem \ref{thm:class-lim}.

\begin{lemma}[{\cite[Corollary 10.3.4]{AAR}}] \label{lem:qgamma}
For $x \in \CC - \{0, -1, \ldots\}$, we have 
\[
\lim_{p \to 1^-} (1 - p)^{1 - x} \frac{(p; p)}{(p^x; p)} = \Gamma(x).
\]
\end{lemma}

\begin{lemma} \label{lem:class-asymp2}
We have 
\[
\lim_{p \to 1^-} \frac{(p^a; p)}{(p^b; p)} (1 - p)^{a - b} = \frac{\Gamma(b)}{\Gamma(a)}
\]
\end{lemma}
\begin{proof}
This follows by applying Lemma \ref{lem:qgamma} twice.
\end{proof}

\begin{lemma} \label{lem:sc-asymp}
If $u$ and $v$ are complex with $|u|, |v| < 1$, we have
\[
\lim_{p \to 1^-} (p^a u; v) = (u; v)
\]
uniformly on compact subsets of $|u| < 1$.
\end{lemma}
\begin{proof}
We evaluate the limit as
\[
\lim_{p \to 1^-} (p^a u; v) = \lim_{p \to 1} \exp\Big(-\sum_{m > 0} \sum_{n \geq 0} \frac{p^{am} u^m v^m}{m}\Big) = \exp\Big(-\sum_{m > 0} \sum_{n \geq 0} \frac{u^m v^{nm}}{m}\Big) = (u; v). \qedhere
\]
\end{proof}

\begin{lemma}[{\cite[Theorem 10.2.4]{AAR}}] \label{lem:class-asymp}
For complex $u$ with $|u| \leq 1$, if $b \geq a$ and $a + b \geq 1$, we have 
\[
\lim_{p \to 1^-} \frac{(p^a u; p)}{(p^b u; p)} = (1 - u)^{b - a}
\]
uniformly on compact subsets of $|u| \leq 1$. If $a + b \geq 1$ does not hold, the convergence is uniform on compact subsets of $\{|u| \leq 1\}$ avoiding $1$.
\end{lemma}

\begin{lemma} \label{lem:class-bound}
For complex $u$, if $b \geq a$ and $b \notin \{0, -1, \ldots\}$, uniformly in $p < 1$ near $1$ and $|u| = 1$ near $1$, we have
\[
\frac{(p^a u; p)}{(p^b u; p)} \leq C (1 - u)^{b - a}
\]
for some constant $C$ uniform in $u$ and $p$.
\end{lemma}
\begin{proof}
Choose $m \geq 0$ so that $a + b + 2m \geq 1$.  For this $m$, we have 
\[
\frac{(p^a u; p)}{(p^b u; p)} = \frac{(p^au; p)_m}{(p^bu; p)_m}\frac{(p^{a + m}u; p)}{(p^{b + m}u; p)},
\]
so by Lemma \ref{lem:class-asymp} it suffices to show that $\frac{1 - p^{a + l}u}{1 - p^{b + l}u}$ is uniformly bounded for $0 \leq l \leq m - 1$.  Let $p = e^\eps$ and $u = e^{2\pi i \delta}$ and notice that for some constant $C_1$ and $p, u$ sufficiently close to $1$ that
\[
\left|\frac{1 - p^{a + l}u}{1 - p^{b + l}u}\right| = \left|1 + \frac{u p^l (p^a - p^b)}{1 - p^{b + l}u}\right| \leq 1 + \frac{|p^a - p^b|}{|1 - p^{b + l}u|} \leq 1 + C_1 \frac{|b - a| \eps}{|(b + l)\eps + 2\pi i \delta|} \leq 1 + C_1 \frac{|b - a|}{|b + l|},
\]
which is uniformly bounded, as needed.
\end{proof}

\begin{proof}[Proof of Theorem \ref{thm:class-lim}]
Recall from Theorem \ref{thm:int-trace} that 
\begin{align*}
T^{w_0}(q, \lambda, \omega, \mu, k) &=  \frac{q^{\lambda\mu - \lambda + 2}(q^{-4}; q^{-2\omega})}{\theta_0(q^{2\lambda}; q^{-2\omega})(q^{2\lambda - 2} q^{-2\omega}; q^{-2\omega})(q^{-2\lambda - 2}; q^{-2\omega})} \frac{ (q^{-2k};q^{-2k})(q^4 q^{-2k}; q^{-2k})}{(q^{-2\mu + 2};q^{-2k})(q^{2\mu + 2} q^{-2k}; q^{-2k})}\\
&\phantom{==}\frac{(q^{-2 \omega + 2}; q^{-2\omega}, q^{-2 k})^2}{(q^{-2 \omega - 2}; q^{-2 \omega}, q^{-2k})^2} \oint_{\cC_t} \frac{dt}{2\pi it} \Omega_{q^2}(t; q^{-2\omega}, q^{-2k}) \frac{\theta_0(tq^{-2\mu}; q^{-2k})}{\theta_0(tq^{-2}; q^{-2k})} \frac{\theta_0(tq^{2\lambda}; q^{-2\omega})}{\theta_0(tq^{-2}; q^{-2\omega})},
\end{align*}
where $\cC_t$ is the unit circle. Applying Lemmas \ref{lem:class-asymp2}, \ref{lem:sc-asymp}, and \ref{lem:class-asymp}, we obtain that 
\begin{multline*}
\lim_{\eps \to 0} (q^{-4}; q^{-2\omega})  \frac{q^{\lambda\mu - \lambda + 2}}{\theta_0(q^{2\lambda}; q^{-2\omega})(q^{2\lambda - 2} q^{-2\omega}; q^{-2\omega})(q^{-2\lambda - 2}; q^{-2\omega})} \frac{ (q^{-2k};q^{-2k})(q^4 q^{-2k}; q^{-2k})}{(q^{-2\mu + 2};q^{-2k})(q^{2\mu + 2} q^{-2k}; q^{-2k})}\\
= -\frac{e^{\mu \Lambda - \Lambda} (e^{-2\Omega}; e^{-2\Omega})}{\theta_0(e^{2\Lambda}; e^{-2\Omega}) (e^{2\Lambda - 2 \Omega}; e^{-2\Omega})(e^{-2\Lambda};e^{-2\Omega})} \frac{\Gamma(\frac{\mu - 1}{k}) \Gamma(\frac{k - \mu - 1}{k})}{\Gamma(1)\Gamma(-\frac{2}{k})}
\end{multline*}
and
\[
\lim_{\eps \to 0} \frac{(q^{-2 \omega + 2}; q^{-2\omega}, q^{-2 k})^2}{(q^{-2 \omega - 2}; q^{-2 \omega}, q^{-2k})^2}= \lim_{\eps \to 0} \prod_{n \geq 0} \frac{(q^{-2 \omega (n+1) + 2}; q^{-2 k})^2}{(q^{-2 \omega(n + 1) - 2}; q^{-2k})^2}= \prod_{n \geq 0} (1 - e^{-2\Omega(n + 1)})^{\frac{4}{k}}= (e^{-2\Omega}; e^{-2\Omega})^{\frac{4}{k}}.
\]
Uniformly on compact subsets of $\cC_t$, we have by Lemma \ref{lem:sc-asymp} that
\[
\lim_{\eps \to 0} \frac{(tq^{2\lambda}q^{-2\omega}; q^{-2\omega})(t^{-1} q^{-2\lambda} q^{-2\omega}; q^{-2\omega})}{(tq^{-2} q^{-2\omega}; q^{-2\omega})(t^{-1} q^2 q^{-2\omega};q^{-2\omega})} = \frac{(t e^{2\Lambda}e^{-2\Omega}; e^{-2\Omega})(t^{-1} e^{-2\Lambda} e^{-2 \Omega}; e^{-2\Omega})}{(te^{-2\Omega}; e^{-2\Omega})(t^{-1} e^{-2\Omega}; e^{-2\Omega})}.
\]
and for sufficiently small $q^{-2\omega}$ by Lemma \ref{lem:class-asymp} that
\[
\lim_{\eps \to 0} \prod_{n \geq 0} \frac{(tq^{-2} q^{-2\omega (n + 1)}; q^{-2k})(t^{-1} q^{-2} q^{-2\omega (n + 1)} q^{-2k}; q^{-2k})}{(tq^2 q^{-2\omega (n+1)}; q^{-2k})(t^{-1} q^{2} q^{-2\omega (n + 1)} q^{-2k}; q^{-2k})} = \prod_{n \geq 0} (1 - te^{-2\Omega (n+1)})^{-\frac{2}{k}}(1 - t^{-1}e^{-2\Omega (n + 1)})^{-\frac{2}{k}}
\]
Finally, by Lemma \ref{lem:class-asymp}, uniformly on compact subsets of $\cC_t$ avoiding $1$ we have 
\[
\lim_{\eps \to 0} \frac{1 - tq^{2\lambda}}{1 - t q^{-2}}\frac{(tq^{-2\mu}; q^{-2k})(t^{-1}q^{2\mu} q^{-2k}; q^{-2k})}{(t q^{2}; q^{-2k})(t^{-1} q^2 q^{-2k}; q^{-2k})} = (1 - t q^{2\Lambda}) \frac{(1 - t)^{-\frac{\mu + 1}{k}}(1 - t^{-1})^{\frac{\mu - 1}{k}}}{1 - t},
\]
where we note that for $-1 < \mu < 1$ both $-\frac{\mu + 1}{k}$ and $\frac{\mu - 1}{k}$ are positive because $k < 0$ for our choice of parameters.  Combining the previous three limits and applying Lemma \ref{lem:class-bound}, on a small enough compact neighborhood of $1$ in $\cC_t$, the integrand has asymptotics bounded uniformly in $\eps$ by a constant multiple of $(1 - t)^{-\frac{2}{k} - 1}$.  Therefore, the limit of the integral may be computed by omitting a shrinking compact neighborhood of $1$ in $\cC_t$, so we conclude by uniformity of the limits on compact subsets of $\cC_t$ avoiding $1$ that 
\begin{align*}
t(\Lambda, \Omega, \mu, k) &= - \frac{e^{\mu \Lambda - \Lambda} (e^{-2\Omega}; e^{-2\Omega})}{\theta_0(e^{2\Lambda}; e^{-2\Omega}) (e^{2\Lambda - 2 \Omega}; e^{-2\Omega})(e^{-2\Lambda};e^{-2\Omega})} \frac{\Gamma(\frac{\mu - 1}{k}) \Gamma(\frac{k - \mu - 1}{k})}{\Gamma(1)\Gamma(-\frac{2}{k})}(e^{-2\Omega}; e^{-2\Omega})^{\frac{4}{k}}\\
&\phantom{==========} \oint_{\cC_t} \frac{dt}{2\pi it} (1 - t)^{-\frac{\mu + 1}{k}}(1 - t^{-1})^{\frac{\mu - 1}{k}} \frac{\theta_0(t e^{2\Lambda}; e^{-2\Omega})}{\theta_0(t; e^{-2\Omega}) (t e^{-2\Omega}; e^{-2\Omega})^{\frac{2}{k}} (t^{-1} e^{-2\Omega}; e^{-2\Omega})^{\frac{2}{k}}}\\
&= \frac{e^{\mu \Lambda + \Lambda} e^{\pi i \frac{\mu - 1}{k}}}{1 - e^{-\frac{4\pi i}{k}}} \frac{(e^{-2\Omega}; e^{-2\Omega})^{\frac{k + 4}{k}}}{\theta_0(e^{2\Lambda}; e^{-2\Omega})^2} \frac{\Gamma(\frac{\mu - 1}{k}) \Gamma(\frac{k - \mu - 1}{k})}{\Gamma(-\frac{2}{k})}\oint_{\Delta} \frac{dt}{2\pi it}\frac{\theta_0(t e^{2\Lambda}; e^{-2\Omega})}{\theta_0(t; e^{-2\Omega})^{\frac{k + 2}{k}}} t^{- \frac{\mu - 1}{k}},
\end{align*}
where in the second equality we recall that $\Delta$ is a Pochhammer loop around $0$ and $1$, note that 
\[
(1 - t)^{-\frac{\mu + 1}{k}}(1 - t^{-1})^{\frac{\mu - 1}{k}}  = e^{\pi i \frac{\mu - 1}{k}} t^{-\frac{\mu - 1}{k}} (1 - t)^{-\frac{2}{k}}
\]
uniformly in $t$ on the complement of the positive real axis, and note that the integral over $\Delta$ is related to the integral over $\cC_t$ by the monodromy around the branch point at $t = 1$.
\end{proof}

\section{The trigonometric limit} \label{sec:trig}

We show that the trigonometric limit of our computation recovers the trace function for $U_q(\sl_2)$.

\subsection{Conventions for $U_q(\sl_2)$}

Recall that $U_q(\sl_2)$ is the Hopf algebra generated by $e, f, q^h$ with relations
\[
q^h e q^{-h} = q^2 e, \qquad q^h f q^{-h} = q^{-2} f, \qquad [e, f] = \frac{q^h - q^{-h}}{q - q^{-1}}
\]
and coproduct and antipode
\begin{align*}
\Delta(e) &= e \otimes 1 + q^h \otimes e, \qquad \Delta(f) = f \otimes q^{-h} + 1 \otimes f, \qquad \Delta(q^h) = q^h \otimes q^h;\\
S(e) &= - q^{-h} e, \qquad S(f) = -f q^h, \qquad S(q^h) = q^{-h}.
\end{align*}
As defined, $U_q(\sl_2)$ is a Hopf subalgebra of $U_q(\asl_2)$.  Let $M_\mu$ denote the Verma module for $U_q(\sl_2)$ with highest weight $\mu$, and let $v_\mu$ be its highest weight vector.  It has a basis $\{f^j v_\mu\}_{j \geq 0}$, on which $U_q(\sl_2)$ acts by
\[
e \cdot f^j v_\mu = [\mu - j + 1][j] f^{j-1} v_\mu, \qquad f \cdot f^j v_\mu = f^{j+1} v_\mu, \qquad q^h\cdot f^j v_\mu = q^{\mu - 2j} f^j v_\mu.
\]
If $\mu$ is a non-negative integer, let $L_\mu$ denote the finite-dimensional irreducible module with highest weight $\mu$.  We pick a basis $\{w_{2j}\}$ for $L_{2\mu}$ so that $U_q(\sl_2)$ acts by
\[
e \cdot w_{2j} = [\mu - j] w_{2j + 2}, \qquad f \cdot w_{2j} = [\mu + j] w_{2j - 2}, \qquad q^h \cdot w_{2j} = q^{2j} w_{2j}.
\]
Note that this basis is compatible with the basis for the evaluation module $L_{2\mu}(z)$ for $U_q(\asl_2)$.  

\subsection{Computing the trace function for $U_q(\sl_2)$}

For $w_0 \in L_{2m}$, let $\Phi^{w_0, \trig}_{\mu}$ be the unique intertwiner 
\[
\Phi^{w_0, \trig}_{\mu}: M_\mu \to M_\mu \hotimes L_{2m}
\]
for which $\Phi^{w_0, \trig}_{\mu}(v_\mu) = v_\mu \otimes w_0 + (\text{l.o.t.})$.  We compute the trace function associated to $\Phi^{w_0, \trig}_{\mu}$ in a manner similar to that used in the opposite coproduct in \cite[Section 7]{EV}.

\begin{lemma} \label{lem:trig-inter-val}
For $w_0 \in L_{2m}$, the intertwiner $\Phi^{w_0, \trig}_\mu$ acts on the highest weight vector $v_\mu \in M_\mu$ by
\[
\Phi^{w_0, \trig}_\mu (v_\mu) = \sum_{j = 0}^m (-1)^{j} q^{\mu j - j (j - 1)} \frac{[m]_j}{[\mu]_j [j]_j} f^j v_\mu \otimes w_{2j}.
\]
\end{lemma}
\begin{proof}
For some constants $c_j(\mu, m)$ with $c_0(\mu, m) = 1$, we have
\[
\Phi^{w_0, \trig}_{\mu}(v_\mu) = \sum_{j = 0}^m c_j(\mu, m) f^j v_\mu \otimes w_{2j}.
\]
Because $\Phi^{w_0, \trig}_{\mu}(v_\mu)$ is killed by $\Delta(e)$, we find that 
\begin{align*}
0 &= \sum_{j = 0}^m c_j(\mu, m) e f^j v_\mu \otimes w_{2j} + \sum_{j = 0}^m c_j(\mu, m) q^h f^j v_\mu \otimes e w_{2j}\\
&= \sum_{j = 0}^m c_j(\mu, m) [\mu - j + 1][j] f^{j-1} v_\mu \otimes w_{2j} + \sum_{j = 0}^m c_j(\mu, m) q^{\mu - 2j} [m - j] f^j v_\mu \otimes w_{2j + 2}.
\end{align*}
Taking the coefficient of $f^j v_\mu \otimes w_{2j + 2}$ yields
\[
c_j(\mu, m) q^{\mu - 2j} [m - j] + c_{j + 1}(\mu, m) [\mu - j][j + 1] = 0
\]
and therefore by induction we find that 
\[
c_j(\mu, m) = (-1)^{j} q^{\mu j - j (j - 1)} \frac{[m]_j}{[\mu]_j [j]_j}. \qedhere
\]
\end{proof}

\begin{prop} \label{prop:trig-trace}
If $|q^{-2\lambda}| \ll 1$, the trace function for $\Phi^{w_0, \trig}_{\mu}$ converges and has value
\[
\Tr|_{M_\mu}(\Phi^{w_0, \trig}_{\mu} q^{2\lambda\rho}) = q^{\lambda \mu}\sum_{l = 0}^m (-1)^{l} q^{-2\lambda l}q^{- l (l - 1)/2} \frac{[m]_l [m + l]_l}{[l]_l} \frac{(q - q^{-1})^l}{\prod_{i = 0}^{l-1} (1 - q^{-2\mu + 2i}) \prod_{i = 0}^l (1 - q^{-2\lambda - 2i})}.
\]
\end{prop}
\begin{proof}
Applying Lemma \ref{lem:trig-inter-val} and the expansion
\[
\Delta(f^k) = \sum_{l = 0}^k \frac{\prod_{i = k - l + 1}^k (1 - q^{-2i})}{\prod_{i = 1}^l (1 - q^{-2i})} f^{k - l} \otimes q^{-(k - l)h} f^l,
\]
we obtain
\begin{align*}
\Phi^{w_0, \trig}_{\mu}(&q^{2\lambda\rho}f^k v_\mu) = q^{\lambda \mu - 2 \lambda k} \Delta(f^k) \Phi^{w_0, \trig}_{\mu}(v_\mu)\\
&= q^{\lambda \mu - 2 \lambda k}  \sum_{l = 0}^k \frac{\prod_{i = k - l + 1}^k (1 - q^{-2i})}{\prod_{i = 1}^l (1 - q^{-2i})} (f^{k - l} \otimes q^{-(k - l)h} f^l) \sum_{j = 0}^m (-1)^{j} q^{\mu j - j (j - 1)} \frac{[m]_j}{[\mu]_j [j]_j} f^j v_\mu \otimes w_{2j}\\
&= q^{\lambda \mu - 2 \lambda k}  \sum_{l = 0}^k \frac{\prod_{i = k - l + 1}^k (1 - q^{-2i})}{\prod_{i = 1}^l (1 - q^{-2i})}\sum_{j = 0}^m (-1)^{j} q^{\mu j - j (j - 1)} \frac{[m]_j}{[\mu]_j [j]_j} f^{k + j - l} v_\mu \otimes q^{-(k-l)h} f^lw_{2j}.
\end{align*}
The diagonal terms correspond to the $j = l$ term of the above summation. Notice that 
\[
q^{-(k-l)h} f^l w_{2l} = [m + l]_l w_0.
\]
We conclude that 
\begin{align*}
\Tr|_{M_\mu}(&\Phi^{w_0,\trig}_{\mu} q^{2\lambda\rho}) = \sum_{k \geq 0} q^{\lambda \mu - 2\lambda k} \sum_{l = 0}^k \frac{\prod_{i = k - l + 1}^k (1 - q^{-2i})}{\prod_{i = 1}^l (1 - q^{-2i})} (-1)^{l} q^{\mu l - l (l - 1)} \frac{[m]_l [m + l]_l}{[\mu]_l [l]_l}\\
&= q^{\lambda \mu}\sum_{l \geq 0} (-1)^{l} q^{-2\lambda l}q^{\mu l - l (l - 1)} \frac{[m]_l [m + l]_l}{[\mu]_l [l]_l}\sum_{k \geq l}  q^{- 2\lambda (k - l)}  \frac{\prod_{i = k - l + 1}^k (1 - q^{-2i})}{\prod_{i = 1}^l (1 - q^{-2i})} \\
&= q^{\lambda \mu}\sum_{l \geq 0} (-1)^{l} q^{-2\lambda l}q^{\mu l - l (l - 1)} \frac{[m]_l [m + l]_l}{[\mu]_l [l]_l} \frac{1}{\prod_{i = 0}^l (1 - q^{-2\lambda - 2i})}\\
&= q^{\lambda \mu}\sum_{l = 0}^m (-1)^{l} q^{-2\lambda l}q^{- l (l - 1)/2} (q - q^{-1})^l \frac{[m]_l [m + l]_l}{[l]_l} \frac{1}{\prod_{i = 0}^{l-1} (1 - q^{-2\mu + 2i}) \prod_{i = 0}^l (1 - q^{-2\lambda - 2i})}. \qedhere
\end{align*}
\end{proof}

We make our previous computation explicit in the special case $m = 1$ corresponding to the three-dimensional representation.  For $w_0 \in L_2$, define 
\[
T^{w_0, \trig}(q, \lambda, \mu) = \Tr|_{M_{\mu - 1}}(\Phi^{w_0, \trig}_{\mu - 1}\,q^{2\lambda\rho}).
\]

\begin{lemma} \label{lem:trig-tr1-val}
We have
\[
T^{w_0, \trig}(q, \lambda, \mu) = \frac{q^{\lambda \mu - \lambda}}{1 - q^{-2\lambda}}\Big(\frac{1 - q^{-2\mu + 2} - q^{-2\lambda + 2} + q^{-2\mu - 2\lambda}}{(1 - q^{-2\mu + 2})(1 - q^{-2\lambda - 2})}\Big).
\]
\end{lemma}
\begin{proof}
This follows specializing Proposition \ref{prop:trig-trace} to the case $m = 1$.
\end{proof}

\subsection{The trigonometric limit of the trace function}

We check that the trigonometric limit $q^{-2\omega} \to 0$ of $T^{w_0}(q, \lambda, \omega, \mu, k)$ corresponds with the trace function for $U_q(\sl_2)$.  

\begin{thm} \label{thm:trig-consist}
We have that 
\[
\lim_{q^{-2\omega} \to 0} T^{w_0}(q, \lambda, \omega, \mu, k) = T^{w_0, \trig}(q, \lambda, \mu).
\]
\end{thm}
\begin{proof}
The limiting expression is the constant coefficient in $q^{-2\omega}$ of $T^{w_0}(q, \lambda, \omega, \mu, k)$, which by Proposition \ref{prop:int-formal} and Lemma \ref{lem:trig-tr1-val} has value
\begin{align*}
\lim_{q^{-2\omega} \to 0}& T^{w_0}(q, \lambda, \omega, \mu, k) = \frac{q^{\lambda\mu - \lambda + 2}(1 - q^{-4})}{(1 - q^{2\lambda})(1 - q^{-2\lambda - 2})} \frac{(q^{-2\mu-2}; q^{-2k})}{(q^{-2\mu + 2};q^{-2k})}  \sum_{n \geq 0} \frac{q^{(-2\mu + 2)n} (1 - q^{2\lambda - 2} q^{2kn})}{(1 - q^{-4}q^{2kn})}\frac{(q^{-4}q^{2kn}; q^{-2k})_n}{(q^{2kn}; q^{-2k})_n}\\
&=-\frac{q^{\lambda\mu - \lambda + 2}}{(1 - q^{2\lambda})(1 - q^{-2\lambda - 2})} \frac{(q^{-2\mu-2}; q^{-2k})}{(q^{-2\mu + 2};q^{-2k})} \frac{q^4 (1 - q^{-4})}{(1 - q^{4})} \sum_{n \geq 0} q^{(-2\mu - 2 - 2k)n} (1 - q^{2\lambda - 2} q^{2kn})\frac{(q^{4}; q^{-2k})_n}{(q^{-2k}; q^{-2k})_n}\\
&=\frac{q^{\lambda\mu - \lambda + 2}}{(1 - q^{2\lambda})(1 - q^{-2\lambda - 2})} \frac{(q^{-2\mu-2}; q^{-2k})}{(q^{-2\mu + 2};q^{-2k})}  \Big(\frac{(q^{-2\mu + 2 - 2k}; q^{-2k})}{(q^{-2\mu - 2 - 2k}; q^{-2k})} - q^{2\lambda - 2}\frac{(q^{-2\mu + 2}; q^{-2k})}{(q^{-2\mu - 2}; q^{-2k})}\Big)\\
&=T^{w_0, \trig}(q, \lambda, \mu). \qedhere
\end{align*}
\end{proof}

\section{Symmetry of the trace function} \label{sec:sym}

In this section, we show that a certain normalization of the trace function is symmetric under interchanging $(\lambda, \omega)$ and $(\mu, k)$.  To state the symmetry which we will show, define the Weyl denominator $\delta_q(\lambda, \omega)$ by 
\[
\delta_q(\lambda, \omega) := \Tr|_{M_{-\wrho}}(q^{2\lambda \rho + 2\omega d})^{-1}
\]
and the normalized trace
\[
\wT^{w_0}(q, \lambda, \omega, \mu, k) := \delta_q(\lambda, \omega) T^{w_0}(q^{-1}, -\lambda, -\omega, \mu, k),
\]
where on the right we consider the quasi-analytic continuation of $T^{w_0}(q, \lambda, \omega, \mu, k)$ to the region of parameters $|q| > 1$, $|q^{-2\omega}| < 1$, and $|q^{-2k}| > 1$.

\begin{thm} \label{thm:tnorm-sym}
The function $\wT^{w_0}(q, \lambda, \omega, \mu, k)$ is symmetric under interchange of $(\lambda, \omega)$ and $(\mu, k)$.
\end{thm}

For Theorem \ref{thm:tnorm-sym}, we first compute the quasi-analytic continuation and the normalization factor $\delta_q(\lambda, \omega)$.

\begin{lemma} \label{lem:tr-qa}
The quasi-analytic continuation of $T^{w_0}(q, \lambda, \omega, \mu, k)$ to $|q| > 1$, $|q^{-2\omega}| < 1$, and $|q^{-2k}| > 1$ is
\begin{multline*}
T^{w_0}(q, \lambda, \omega, \mu, k) \equiv \frac{q^{\lambda\mu - \lambda + 2} (q^{-4}; q^{-2\omega})}{\theta_0(q^{2\lambda}; q^{-2\omega})(q^{2\lambda - 2} q^{-2\omega}; q^{-2\omega})(q^{-2\lambda - 2}; q^{-2\omega})} \frac{1}{1 - q^4}\frac{(q^{-4}; q^{2k})(q^{2k}; q^{2k})}{(q^{2\mu - 2}; q^{2k})(q^{-2\mu - 2 + 2k}; q^{2k})} \\
 \frac{(q^{-2 \omega - 2 + 2k}; q^{-2\omega}, q^{2 k})^2}{(q^{-2 \omega + 2 + 2k}; q^{-2 \omega}, q^{2k})^2} \oint_{|t| = 1} \frac{dt}{2 \pi i t} \Omega_{q^{-2}}(t; q^{-2\omega}, q^{2k}) \frac{\theta_0(tq^{2\mu}; q^{2k})}{\theta_0(tq^2; q^{2k})} \frac{\theta_0(t q^{2\lambda}; q^{-2\omega})}{\theta_0(tq^{2}; q^{-2\omega})}.
\end{multline*}
\end{lemma}
\begin{proof}
This follows from Corollary \ref{corr:trace-fv} and Proposition \ref{prop:qa-fv-2}.
\end{proof}

\begin{prop} \label{prop:tr-delta}
We have
\[
\delta_q(\lambda, \omega) = q^{\lambda} (q^{-2\omega};q^{-2\omega}) \theta_0(q^{-2\lambda}; q^{-2\omega}).
\]
\end{prop}
\begin{proof}
Notice that 
\begin{align} \nonumber
\delta_q(\lambda, \omega) &= \Tr|_{M_{-\rho - 2 \Lambda_0}}\Big(q^{2\lambda\rho + 2\omega d}\Big)^{-1} = q^{(\rho + 2 \Lambda_0, 2 \lambda \rho + 2 \omega d)} \prod_{\beta > 0} (1 - q^{-(\beta, 2\lambda + 2 \omega d)})^{\mult(\beta)} \\ \label{eq:verma-tr}
&= q^{\lambda} (1 - q^{-2\lambda}) \prod_{m > 0} (1 - q^{-2\omega m})(1 - q^{2\lambda - 2\omega m})(1 - q^{-2\lambda - 2 \omega m})= q^{\lambda} (q^{-2\omega};q^{-2\omega}) \theta_0(q^{-2\lambda}; q^{-2\omega}),
\end{align}
where we recall that the positive roots for $U_q(\asl_2)$ have multiplicity $1$ and are given by
\[
\{\alpha, \pm \alpha + m\delta, m\delta\mid m > 0\}. \qedhere
\]
\end{proof}

\begin{proof}[Proof of Theorem \ref{thm:tnorm-sym}]
By Lemma \ref{lem:tr-qa} and Proposition \ref{prop:tr-delta}, after some cancellation it suffices to check that 
\begin{multline*}
q^{-2\lambda} \oint_{|t| = 1} \frac{dt}{2 \pi i t} \Omega_{q^{2}}(t; q^{-2\omega}, q^{-2k}) \frac{\theta_0(tq^{-2\mu}; q^{-2k})}{\theta_0(tq^{-2}; q^{-2k})} \frac{\theta_0(t q^{2\lambda}; q^{-2\omega})}{\theta_0(tq^{-2}; q^{-2\omega})}\\
 = q^{-2\mu} \oint_{|t| = 1} \frac{dt}{2 \pi i t} \Omega_{q^{2}}(t; q^{-2\omega}, q^{-2k}) \frac{\theta_0(tq^{2\mu}; q^{-2k})}{\theta_0(tq^{-2}; q^{-2k})} \frac{\theta_0(t q^{-2\lambda}; q^{-2\omega})}{\theta_0(tq^{-2}; q^{-2\omega})}.
\end{multline*}
By Lemma \ref{lem:phase-trans}, we have that 
\begin{align*}
\Omega_{q^2}(t; q^{-2\omega}, q^{-2k}) &= \Omega_{q^2}(t^{-1}; q^{-2\omega}, q^{-2k}) \frac{\theta_0(t^{-1}q^2; q^{-2k}) \theta_0(tq^{-2}; q^{-2\omega})}{\theta_0(t^{-1} q^{-2}; q^{-2k}) \theta_0(tq^2; q^{-2\omega})}\\
&= q^{4} \Omega_{q^2}(t^{-1}; q^{-2\omega}, q^{-2k}) \frac{\theta_0(tq^{-2}; q^{-2k}) \theta_0(tq^{-2}; q^{-2\omega})}{\theta_0(t q^{2}; q^{-2k}) \theta_0(tq^2; q^{-2\omega})}.
\end{align*}
Upon substitution, this implies that 
\begin{align*}
q^{-2\lambda} \oint_{|t| = 1} &\frac{dt}{2 \pi i t} \Omega_{q^{2}}(t; q^{-2\omega}, q^{-2k}) \frac{\theta_0(tq^{-2\mu}; q^{-2k})}{\theta_0(tq^{-2}; q^{-2k})} \frac{\theta_0(t q^{2\lambda}; q^{-2\omega})}{\theta_0(tq^{-2}; q^{-2\omega})}\\
&= q^{-2\lambda + 4} \oint_{|t| = 1} \frac{dt}{2 \pi i t} \Omega_{q^{2}}(t^{-1}; q^{-2\omega}, q^{-2k}) \frac{\theta_0(tq^{-2\mu}; q^{-2k})}{\theta_0(tq^{2}; q^{-2k})} \frac{\theta_0(t q^{2\lambda}; q^{-2\omega})}{\theta_0(tq^{2}; q^{-2\omega})}\\
&= q^{-2\mu} \oint_{|t| = 1} \frac{dt}{2 \pi i t} \Omega_{q^{2}}(t^{-1}; q^{-2\omega}, q^{-2k}) \frac{\theta_0(t^{-1}q^{2\mu}; q^{-2k})}{\theta_0(t^{-1}q^{-2}; q^{-2k})} \frac{\theta_0(t^{-1} q^{-2\lambda}; q^{-2\omega})}{\theta_0(t^{-1}q^{-2}; q^{-2\omega})}\\
&= q^{-2\mu} \oint_{|t| = 1} \frac{dt}{2 \pi i t} \Omega_{q^{2}}(t; q^{-2\omega}, q^{-2k}) \frac{\theta_0(tq^{2\mu}; q^{-2k})}{\theta_0(tq^{-2}; q^{-2k})} \frac{\theta_0(t q^{-2\lambda}; q^{-2\omega})}{\theta_0(tq^{-2}; q^{-2\omega})},
\end{align*}
where in the last step we make the change of variables $t \mapsto t^{-1}$.
\end{proof}

\section{Application to affine Macdonald polynomials} \label{sec:aff-mac}

In this section we explain how our trace functions relate to the affine Macdonald polynomials for $\asl_2$ defined by Etingof-Kirillov~Jr. in \cite{EK3}.  We use them to prove Felder-Varchenko's conjecture that their definition of affine Macdonald polynomials via hypergeometric theta functions in \cite{FV4} coincides with that of \cite{EK3}.  The classical degeneration of this section is related to the study of conformal blocks in \cite{FSV1}.

\subsection{Affine Macdonald polynomials as traces of intertwiners} \label{sec:aff-mac-def}

Fix an integer $\mk \geq 0$.  For integers $k \geq \mu \geq 0$, let $L_{\mu, k}$ denote the irreducible integrable module for $U_q(\asl_2)$ with highest weight $\mu \rho + k\Lambda_0$ and highest weight vector $v_{\mu, k}$.  For $v \in \Sym^{2(\mk - 1)}L_1[0]$, we have an intertwiner 
\[
\Upsilon_{\mu, k, \mk}^v(z): L_{\mu\rho + k \Lambda_0+ (\mk - 1)\wrho} \to L_{\mu\rho + k \Lambda_0 + (\mk - 1)\wrho} \hotimes \Sym^{2(\mk - 1)}L_1(z)
\]
such that $\Upsilon^v_{\mu, k, \mk}(z) v_{\mu\rho + k \Lambda_0 + (\mk - 1)\wrho} = v_{\mu\rho + k \Lambda_0 + (\mk - 1)\wrho} \otimes v + (\text{l.o.t.})$. Define the trace function
\[
\chi_{\mu, k, \mk}(q, \lambda, \omega) = \Tr|_{L_{\mu\rho + k \Lambda_0 + (\mk - 1)\wrho}}\Big(\Upsilon^{w_0}_{\mu, k, \mk}(z) q^{2\lambda} q^{2\omega d}\Big),
\]
where the trace is independent of $z$.  In \cite{EK3}, the affine Macdonald polynomial for $\asl_2$ was defined to be
\[
J_{\mu, k, \mk}(q, \lambda, \omega) := \frac{\chi_{\mu, k, \mk}(q, \lambda, \omega)}{\chi_{0, 0, \mk}(q, \lambda, \omega)}.
\]
It is a symmetric Laurent polynomial with highest term $e^{(\mu \rho + k \Lambda_0, 2\lambda \rho + 2\omega d)}$. 

\begin{remark}
To avoid conflict with the use of $k$ for level, we use $\mk$ to denote the parameter of the Macdonald polynomial.  This corresponds to the variable $k$ in \cite{EK3} and $m$ in \cite{FV4}.
\end{remark}

\subsection{Elliptic Macdonald polynomials as hypergeometric theta functions}

In \cite{FV4}, the elliptic Macdonald polynomial was defined by Felder-Varchenko in terms of hypergeometric theta functions.  Fix parameters satisfying $|q| > 1$, $|q^{-2k}| < 1$, and $|q^{-2\omega}| < 1$. Define
\[
Q(q, \mu, k) := -\frac{2iq^{2\mu - 2}(q^{-2k}; q^{-2k})^2\theta_0(q^4; q^{-2k})}{\theta_0(q^{2\mu - 2}; q^{-2k})\theta_0(q^{2\mu + 2}; q^{-2k})}.
\]
In terms of $Q$, define the $l^\text{th}$ non-symmetric hypergeometric theta function of level $\kappa + 2$ by the convergent series 
\[
\wDelta_{\mu, \kappa}(q, \lambda, \omega) := q^{\frac{2\mu^2}{\kappa}}Q(q, \mu, \kappa)\sum_{j \in 2\kappa \ZZ + \mu} u(q, \lambda, \omega, j, \kappa) q^{-\frac{\omega + 2}{2\kappa}j^2}
\]
and the $\mu^\text{th}$ hypergeometric theta function of level $\kappa + 2$ by 
\[
\Delta_{\mu, \kappa}(q, \lambda, \omega) := \wDelta_{\mu, \kappa}(q, \lambda, \omega) - \wDelta_{\mu, \kappa}(q, -\lambda, \omega).
\]
Then the elliptic Macdonald polynomial is given by
\[
\wJ_{\mu, \kappa}(q, \lambda, \omega) := \frac{i q^{-\frac{(\mu + 2)^2}{\kappa} + \frac{\omega}{2\kappa}(\mu + 2)^2 + 3\lambda}}{(q^{-2\omega}; q^{-2\omega})^3} \frac{\Delta_{\mu + 2, \kappa}(q, \lambda, \omega)}{\theta_0(q^{2\lambda - 2}; q^{-2\omega}) \theta_0(q^{2\lambda}; q^{-2\omega})\theta_0(q^{2\lambda + 2}; q^{-2\omega})}.
\]
In \cite{FV4}, Felder-Varchenko conjectured that $\wJ_{\mu, \kappa}(q, \lambda, \omega)$ is related to the affine Macdonald polynomial of \cite{EK3}. Define the quantity
\[
\wJ^0_{\mu, \kappa}(q, \lambda, \omega) := \frac{2 q^{2\mu + 3\lambda + 2} (q^{-2\kappa}; q^{-2\kappa})^2 \theta_0(q^4; q^{-2\kappa})}{\theta_0(q^{2\mu + 2}; q^{-2\kappa})\theta_0(q^{2\mu + 6}; q^{-2\kappa})(q^{-2\omega}; q^{-2\omega})^3\theta_0(q^{2\lambda - 2}; q^{-2\omega}) \theta_0(q^{2\lambda}; q^{-2\omega})\theta_0(q^{2\lambda + 2}; q^{-2\omega})}.
\]

\begin{lemma} \label{lem:ell-mac-trans}
For $q^{-2\mu}$ and then $q^{-2\omega}$ sufficiently close to $0$ in the good region of parameters (\ref{eq:good-region}), the elliptic Macdonald polynomial may be expressed in the form
\[
\wJ_{\mu, \kappa}(q, \lambda, \omega) = \wJ^0_{\mu, \kappa}(q, \lambda, \omega) \sum_{j \in \ZZ} q^{-2j(\kappa j + \mu + 2)(\omega + 2)} \Big(u(q, \lambda, \omega, \mu + 2 + 2\kappa j, \kappa) - u(q, -\lambda, \omega, \mu + 2 + 2\kappa j, \kappa)\Big).
\]
\end{lemma}
\begin{proof}
Applying the definitions, we find that 
\begin{align*}
\wJ_{\mu, \kappa}(q, \lambda, \omega) &= \frac{i q^{-\frac{(\mu + 2)^2}{\kappa} + \frac{\omega}{2\kappa}(\mu + 2)^2 + 3\lambda}}{(q^{-2\omega}; q^{-2\omega})^3} \frac{q^{\frac{2(\mu + 2)^2}{\kappa}}Q(q, \mu + 2, \kappa)}{\theta_0(q^{2\lambda - 2}; q^{-2\omega}) \theta_0(q^{2\lambda}; q^{-2\omega})\theta_0(q^{2\lambda + 2}; q^{-2\omega})}\\
&\phantom{===} \sum_{j \in 2\kappa \ZZ + \mu + 2} (u(q, \lambda, \omega, j, \kappa) - u(q, -\lambda, \omega, j, \kappa))q^{-\frac{\omega + 2}{2\kappa}j^2}\\
&= \frac{2 q^{2\mu + 2} q^{\frac{(\mu + 2)^2}{\kappa} + \frac{\omega}{2\kappa}(\mu + 2)^2 + 3\lambda} (q^{-2\kappa}; q^{-2\kappa})^2 \theta_0(q^4; q^{-2\kappa})}{\theta_0(q^{2\mu + 2}; q^{-2\kappa})\theta_0(q^{2\mu + 6}; q^{-2\kappa})(q^{-2\omega}; q^{-2\omega})^3\theta_0(q^{2\lambda - 2}; q^{-2\omega}) \theta_0(q^{2\lambda}; q^{-2\omega})\theta_0(q^{2\lambda + 2}; q^{-2\omega})}\\
&\phantom{===} \sum_{j \in 2\kappa \ZZ + \mu + 2} (u(q, \lambda, \omega, j, \kappa) - u(q, -\lambda, \omega, j, \kappa)) q^{-\frac{\omega + 2}{2\kappa} j^2}\\
&= \wJ^0_{\mu, \kappa}(q, \lambda, \omega) \sum_{j \in \ZZ} (u(q, \lambda, \omega, \mu + 2 + 2\kappa j, \kappa) - u(q, -\lambda, \omega, \mu + 2 + 2\kappa j, \kappa)) q^{-2j(\kappa j + \mu + 2)(\omega + 2)}. \qedhere
\end{align*}
\end{proof}

\subsection{BGG resolution for $U_q(\asl_2)$-modules with integral highest weight}

We introduce now the BGG resolution of $L_{\mu, k}$, which will allow us to compute the affine Macdonald polynomial in terms of our trace functions.  Denote the affine Weyl group of $\asl_2$ by 
\[
W^a = \langle s_0, s_1 \mid s_0^2 = s_1^2 = 1\rangle.
\]
It acts on $\whh$ via 
\[
s_1(\alpha) = - \alpha \qquad s_1(c) = c \qquad s_1(d) = d \qquad s_0(\alpha) = 2c -\alpha \qquad s_0(c) = c \qquad s_0(d) = d + \alpha - c
\]
and on $\whh^*$ via 
\begin{equation} \label{eq:waff-action}
s_1(\alpha) = - \alpha \qquad s_1(\Lambda_0) = \Lambda_0 \qquad s_1(\delta) = \delta \qquad s_0(\alpha) = -\alpha + 2 \delta \qquad s_0(\Lambda_0) = \Lambda_0 + \alpha - \delta \qquad s_0(\delta) = \delta.
\end{equation}
Define the dotted action of $W^a$ on $\whh^*$ by $w \cdot \wmu = w \cdot (\wmu + \wrho) - \wrho$.  For $l > 0$, denote the length $l$ elements of $W^a$ by $w^l_0 := s_0s_1 \cdots$ and $w^l_1 := s_1s_0 \cdots$, where $w^l_0$ and $w^l_1$ contain $l$ reflections.  We compute the actions of $w^l_0$ and $w^l_1$ on $\mu \rho + k \Lambda_0$.

\begin{lemma} \label{lem:w-action}
The dotted action of $w^l_i$ on $\mu\rho + k \Lambda_0$ is given by 
\begin{align*}
w^{2l}_0 \cdot (\mu \rho + k \Lambda_0) &= (\mu + 2l(k + 2))\rho + k \Lambda_0 + (- l(\mu + 1) - l^2(k+2)) \delta \\
w^{2l}_1 \cdot (\mu \rho + k \Lambda_0) &= (\mu - 2l(k + 2))\rho + k \Lambda_0 + (l(\mu + 1) - l^2(k + 2)) \delta \\
w^{2l + 1}_0 \cdot (\mu \rho + k \Lambda_0) &= (-\mu - 2 + 2(l + 1)(k + 2))\rho + k \Lambda_0 + ((l + 1)(\mu + 1) - (l + 1)^2(k + 2)) \delta \\
w^{2l + 1}_1 \cdot (\mu \rho + k \Lambda_0) &= (-\mu - 2 - 2l(k+2))\rho + k \Lambda_0 + (-l(\mu + 1) - l^2(k+2))\delta.
\end{align*}
\end{lemma}
\begin{proof}
This follows from an induction using (\ref{eq:waff-action}).
\end{proof}

For $w \in W^a$ and a reduced decomposition $w = s_{i_1} \cdot s_{i_l}$, define $\alpha^l = \alpha_{i_l}$ and $\alpha^j = (s_{i_l} \cdots s_{i_{j + 1}})(\alpha_{i_j})$.  For a fixed weight $\mu\rho + k\Lambda_0$, define $n_{\mu, k, j} = \frac{2(\mu \rho + k \Lambda_0 + \wrho, \alpha^j)}{(\alpha^j, \alpha^j)}$.  Recalling that $v_{\rho + k \Lambda_0}$ is the highest weight vector in $M_{\mu\rho + k\Lambda_0}$, define the vector $v^{\mu \rho + k \Lambda_0}_{w \cdot (\mu\rho + k \Lambda_0)}$ by 
\[
v^{\mu\rho + k \Lambda_0}_{w \cdot (\mu\rho + k \Lambda_0)} := \frac{f^{n_{\mu, k, 1}}_{i_1}}{[n_{\mu, k, 1}]!} \cdots \frac{f^{n_{\mu, k, l}}_{i_l}}{[n_{\mu, k, l}]!} v_{\mu \rho + k \Lambda_0}.
\]
By \cite[Section 2.7]{EV2}, the vectors $v^{\mu\rho + k \Lambda_0}_{w \cdot (\mu\rho + k \Lambda_0)}$ are independent of the choice of reduced decomposition and span the space of singular vectors in $M_{\mu \rho + k \Lambda_0}$.  They yield a BGG-type resolution for $L_{\mu, k}$ given by Verma modules indexed by $W^a$.

\begin{prop}[{\cite[Theorem 3.2]{HK}}] \label{prop:bgg-res}
For $k \geq \mu \geq 0$, there is an exact complex of $U_q(\asl_2)$-modules
\[
\cdots \to C_2 \to C_1 \to M_{\mu \rho + k \Lambda_0} \to L_{\mu \rho + k \Lambda_0} \to 0
\]
with each term in the resolution given by $C_0 = M_{\mu \rho + k \Lambda_0}$ and
\[
C_l = M_{(w^l_0) \cdot (\mu \rho + k \Lambda_0)} \oplus M_{(w^l_1) \cdot (\mu \rho + k \Lambda_0)} \text{ for $l > 0$}.
\]
\end{prop}

\subsection{Intertwiners for $U_q(\asl_2)$-modules with integral highest weight}

We now give a criterion for the existence of intertwiners involving Verma modules with dominant integral highest weight.

\begin{prop} \label{prop:int-exist}
For integers $k \geq \mu \geq 0$, let $n_i = (\mu \rho + k \Lambda_0, \alpha_i) + 1$.  If $V[n_i \alpha_i] = 0$ for $i = 0, 1$, then for any $v \in V[0]$, there exists a unique intertwiner
\[
\Phi_{\mu, k}^v(z): M_{\mu, k} \to M_{\mu, k} \hotimes V(z)
\]
which satisfies
\[
\Phi_{\mu, k}^v(z)v_{\mu, k} = v_{\mu, k} \otimes v + (\text{l.o.t.}),
\]
where $(\text{l.o.t.})$ denotes terms of lower weight in the first tensor factor.
\end{prop}
\begin{proof}
The space of intertwiners $M_{\mu, k} \to M_{\mu, k} \hotimes V(z)$ is given by
\begin{align*}
\Hom_{U_q(\asl_2)}(M_{\mu, k}, M_{\mu, k} \hotimes V(z)) &\simeq \Hom_{U_q(\asl_2)}(\Ind^{U_q(\asl_2)}_{U_q(\abb_+)} \CC_{\mu, k}, (M_{\mu, k}^\vee)^* \otimes V(z))\\
&\simeq \Hom_{U_q(\abb_+)}(\CC_{\mu, k}, V(z) \otimes (M_{\mu, k}^\vee)^*)\\
&\simeq \Hom_{U_q(\abb_+)}(\CC_{\mu, k} \otimes_k M_{\mu, k}^\vee, V(z))\\
&\simeq \{v \in V[0] \mid I_{\mu, k}v = 0\},
\end{align*}
where $I_{\mu, k} := \{u \in U_q(\asl_2) \mid u \cdot v_{\mu, k}^* = 0\}$ is the annihilator ideal of the lowest weight vector of $M_{\mu, k}^\vee$.  By the identification of singular vectors in Proposition \ref{prop:bgg-res}, the $U_q(\abb_+)$-submodule of $M_{\mu, k}^\vee$ generated by $v_{\mu, k}^*$ has relations $e_i^{n_i} v_{\mu, k}^* = 0$ so that $I_{\mu, k}$ is generated by $e_i^{n_i}$, meaning that $I_{\mu, k}v = 0$ for any $v \in V[0]$ and completing the proof.
\end{proof}

\begin{remark}
By Lemma \ref{lem:rat-fn-me}, for $(\mu, k)$ satisfying the conditions of Proposition \ref{prop:int-exist}, the matrix elements of $\Phi_{\mu, k}^v(z)$ are given by analytic continuation from the generic case.  Therefore, trace functions for such integrable modules are analytic continuations of trace functions for generic highest weight and continue to be given by the expression of Theorem \ref{thm:int-trace}.
\end{remark}

\subsection{The affine dynamical Weyl group action}

Let $(\mu, k)$ be chosen to satisfy the conditions of Proposition \ref{prop:int-exist} with respect to $L_2$.  In \cite{EV2}, for $w$ in the affine Weyl group $W^a$ and $w_0 \in L_2[0]$, it was shown that there exist dynamical Weyl group operators $A_{w, L_2}(\mu \rho + k \Lambda_0)$ so that
\[
\Phi^{w_0}_{\mu, k}(z) v^{\mu \rho + k \Lambda_0}_{w \cdot (\mu \rho + k \Lambda_0)} = v^{\mu \rho + k\Lambda_0}_{w \cdot (\mu\rho + k\Lambda_0)} \otimes A_{w, L_2}(\mu \rho + k\Lambda_0)w_0 + (\text{l.o.t.}).
\]
In particular, $\Phi^{w_0}_{\mu, k}(z)$ restricts to a multiple of an intertwiner
\[
M_{w \cdot (\mu \rho + k \Lambda_0)} \to M_{w \cdot (\mu \rho + k \Lambda_0)} \hotimes L_2(z).
\]
For any reduced decomposition $w = s_{i_1} \cdots s_{i_l}$, we have
\begin{equation} \label{eq:dyn-weyl-recurse}
A_{w, L_2}(\mu \rho + k \Lambda_0) = A_{s_{i_1}}\Big((s_{i_2} \cdots s_{i_l}) \cdot (\mu \rho + k \Lambda_0)\Big) \cdots A_{s_{i_{l-1}}}\Big(s_{i_l} \cdot (\mu \rho + k \Lambda_0)\Big) A_{s_{i_l}}(\mu \rho + k \Lambda_0).
\end{equation}
Because $L_2[0]$ is spanned by $w_0$, the dynamical Weyl group acts by a scalar on it, so we will treat it as a number.  We now compute the action explicitly in Proposition \ref{prop:dyn-weyl-action} by reducing to the trigonometric limit. 

\begin{lemma} \label{lem:trig-sing-mat}
Let $\mu$ be a positive integer.  For $w_0 \in L_2$, the diagonal matrix element of the singular vector $f^{\mu + 1} v_\mu \in M_\mu$ in $\Phi^{w_0, \trig}_\mu$ is $-\frac{[\mu + 2]}{[\mu]}$. 
\end{lemma}
\begin{proof}
By Lemma \ref{lem:trig-inter-val} and the expansion 
\[
\Delta(f^{\mu + 1}) = \sum_{l = 0}^{\mu + 1} \frac{\prod_{i = \mu - l + 2}^{\mu + 1} (1 - q^{-2i})}{\prod_{i = 1}^l (1 - q^{-2i})} f^{\mu + 1 - l} \otimes q^{-(\mu - l + 1)h} f^l,
\]
we find that the $f^{\mu + 1}v_\mu$ coefficient in $\Phi^{w_0, \trig}_\mu(f^{\mu + 1}v_\mu)$ is given by 
\[
1 - \frac{[\mu + 1][2]}{[\mu]} = - \frac{[\mu + 2]}{[\mu]}. \qedhere
\]
\end{proof}

\begin{prop} \label{prop:dyn-weyl-action}
The dynamical Weyl group action on $w_0 \in L_2[0]$ is given by 
\begin{align*}
A_{w^{2l + 1}_1, L_2}(\mu \rho + k \Lambda_0) w_0 &= - q^{-4l - 2} \frac{(q^{2\mu + 4} q^{2l(2k + 4)}; q^{-2k - 4})_{2l + 1}}{(q^{2\mu} q^{2l(2k + 4)}; q^{-2k - 4})_{2l + 1}}\\
A_{w^{2l + 1}_0, L_2}(\mu \rho + k \Lambda_0) w_0 &= - q^{-4l - 2}\frac{(q^{-2\mu} q^{(2l + 1)(2k + 4)}; q^{-2k-4})_{2l + 1}}{(q^{-2\mu - 4} q^{(2l + 1)(2k+4)}; q^{-2k-4})_{2l + 1}}\\
A_{w^{2l}_1, L_2}(\mu \rho + k \Lambda_0) w_0 &= q^{-4l}\frac{(q^{-2\mu} q^{2l(2k + 4)}; q^{-2k-4})_{2l}}{(q^{-2\mu - 4} q^{2l(2k+4)}; q^{-2k-4})_{2l}}\\
A_{w^{2l}_0, L_2}(\mu \rho + k \Lambda_0) w_0 &= q^{-4l} \frac{(q^{2\mu + 4} q^{(2l - 1)(2k + 4)}; q^{-2k - 4})_{2l}}{(q^{2\mu} q^{(2l - 1)(2k + 4)}; q^{-2k - 4})_{2l}}.
\end{align*}
\end{prop}
\begin{proof}
The first value for $l = 0$ follows from Lemma \ref{lem:trig-sing-mat} and the fact that $v^{\mu \rho + k \Lambda_0}_{s_1 \cdot (\mu \rho + k \Lambda_0)}$ is a singular vector for the $U_q(\sl_2)$-submodule $M_\mu \subset M_{\mu \rho + k \Lambda_0}$, so we may compute using the dynamical Weyl group $U_q(\sl_2)$.  The second for $l = 0$ follows by applying the symmetry between $s_0$ and $s_1$ and noting that $\mu \rho + k \Lambda_0 = \mu \Lambda_1 + (k - \mu)\Lambda_0$.  The other values follow by applying Lemma \ref{lem:w-action} and (\ref{eq:dyn-weyl-recurse}).
\end{proof}

\subsection{Proof of Felder-Varchenko's conjecture}

We specialize to $\mk = 2$, where \cite{FV4} conjectured that elliptic and affine Macdonald polynomials are related.  We have then that $\Sym^{2(\mk - 1)}L_1(z) = L_2(z)$ and 
\[
\chi_{\mu, k, 2}(q, \lambda, \omega) = \Tr|_{L_{\mu + 1, k + 2}}\Big(\Phi^{w_0}_{\mu + 1, k + 2}(z) q^{2\lambda} q^{2\omega d}\Big).
\]
In this case, we relate the trace function to the affine Macdonald polynomial via the BGG resolution and apply Theorem \ref{thm:int-trace} to express it as a hypergeometric theta function in the sense of \cite{FV4}.  The proof of Felder-Varchenko's conjecture in Theorem \ref{thm:fv-conj} then follows.  Define the quantity
\begin{multline*}
\chi^0_{\mu, k}(q, \lambda, \omega) := \frac{q^{-\mu + 2} (q^{-4}; q^{-2\omega})}{\theta_0(q^{2\lambda}; q^{-2\omega})(q^{2\lambda - 2} q^{-2\omega}; q^{-2\omega})(q^{-2\lambda - 2}; q^{-2\omega})}\\ \frac{(q^{-2\omega + 2}; q^{-2\omega}, q^{-2\wk})^2}{(q^{-2\omega - 2}; q^{-2\omega}, q^{-2\wk})^2} \frac{(q^{-2\wk}; q^{-2\wk})(q^4 q^{-2\wk}; q^{-2\wk})}{(q^{2\mu + 6} q^{-2\wk}; q^{-2\wk})(q^{-2\mu - 2}; q^{-2\wk})}.
\end{multline*}

\begin{thm} \label{thm:am-tr-value}
For $q^{-2\mu}$ and then $q^{-2\omega}$ sufficiently close to $0$ in the good region of parameters (\ref{eq:good-region}), the trace function $\chi_{\mu, k, 2}(q, \lambda, \omega)$ is given by 
\begin{multline*}
\chi_{\mu, k, 2}(q, \lambda, \omega) \\= \chi^0_{\mu, k}(q, \lambda, \omega) \sum_{j \in \ZZ} q^{-2j(\omega + 2)(\mu + 2 + j\wk)}\Big(u(q, \lambda, \omega, -\mu - 2 - 2j\wk, \wk) - u(q, \lambda, \omega, \mu + 2 + 2j\wk, \wk)\Big),
\end{multline*}
where $\wk = k + 4$.
\end{thm}
\begin{proof}
Define the function 
\[
\sT(q, \lambda, \omega, \mu \rho +  k \Lambda_0 + \Delta \delta) := \Tr|_{M_{\mu\rho + k \Lambda_0 + \Delta \delta}}\Big(\Phi_{\mu, k}^{w_0}(z) q^{2\lambda} q^{2\omega d}\Big)
\]
so that 
\[
\mathsf{T}(q, \lambda, \omega, \mu \rho +  k \Lambda_0 + \Delta \delta) =  q^{2\omega \Delta}T^{w_0}(q, \lambda, \omega, \mu + 1, k + 2).
\]
Applying the BGG resolution of Proposition \ref{prop:bgg-res}, we find that 
\begin{multline} \label{eq:aff-sum-div}
\chi_{\mu, k, 2}(q, \lambda, \omega) \\
= \sT(q, \lambda, \omega, \mu \rho + k \Lambda_0 + \wrho) + \sum_{a = 1}^\infty \sum_{i = 0}^1 (-1)^a A_{w^a_i, L_2}(\mu \rho + k \Lambda_0 + \wrho) \sT\Big(q, \lambda, \omega, w^a_i \cdot (\mu \rho + k \Lambda_0 + \wrho)\Big).
\end{multline}
We compute the sum in (\ref{eq:aff-sum-div}) by dividing it into four pieces depending on the choice of $i$ and the parity of $a$.  In each case, by Lemma \ref{lem:w-action} and Proposition \ref{prop:dyn-weyl-action}, the summand is given by a multiple of a trace function for integral values of $\mu$ and $k$.  By Lemma \ref{lem:rat-fn-me} and uniqueness of the intertwiner
\[
\Phi^{w_0}_{\mu - 1, k - 2}(z): M_{\mu - 1, k - 2} \to M_{\mu - 1, k - 2} \hotimes L_2(z),
\]
when such an intertwiner exists, its matrix elements are given by continuations of the matrix elements given in Lemma \ref{lem:rat-fn-me}.  Therefore, its trace function is given by continuation of the expression of Proposition \ref{prop:rat-fn-trace} even for non-generic $(\mu, k)$.  In particular, for all intertwiners involving the Verma modules which appear in the affine BGG resolution, the corresponding trace function $T^{w_0}(q, \lambda, \omega, \mu, k)$ is given by the formula of Corollary \ref{corr:trace-fv}.  We now analyze each piece of (\ref{eq:aff-sum-div}) in turn.  Setting $\wk = k + 4$, we have
\begin{align*}
A_{w^{2l}_1, L_2}(\mu \rho + k \Lambda_0 + \wrho) &\sT\Big(q, \lambda, \omega, w^{2l}_1 \cdot (\mu \rho + k \Lambda_0 + \wrho)\Big)\\
&= q^{2l(\mu + 2 - l\wk)\omega - 4l}\frac{(q^{-2\mu - 2} q^{4l\wk}; q^{-2\wk})_{2l}}{(q^{-2\mu - 6} q^{4l\wk}; q^{-2\wk})_{2l}} T(q, \lambda, \omega, \mu + 2 - 2l\wk, \wk)\\
&= \chi^0_{\mu, k}(q, \lambda, \omega) q^{2l(\mu + 2 - l\wk)(\omega + 2)} u(q, \lambda, \omega, -\mu - 2 + 2l\wk, \wk)
\end{align*}
and 
\begin{align*}
A_{w^{2l}_0, L_2}(\mu \rho + k \Lambda_0 + \wrho)& \sT\Big(q, \lambda, \omega, w^{2l}_0 \cdot (\mu \rho + k \Lambda_0 + \wrho)\Big)\\
&= q^{-2l(\mu + 2 + l \wk)\omega - 4l} \frac{(q^{2\mu + 6} q^{(4l - 2)\wk}; q^{-2\wk})_{2l}}{(q^{2\mu + 2} q^{(4l - 2)\wk}; q^{-2\wk})_{2l}} T(q, \lambda, \omega, \mu + 2 + 2l\wk, \wk)\\
&= \chi^0_{\mu, k}(q, \lambda, \omega) q^{-2l(\mu + 2 + l \wk)(\omega + 2)} u(q, \lambda, \omega, -\mu - 2 - 2l\wk, \wk)
\end{align*}
and 
\begin{align*}
- A_{w^{2l + 1}_0, L_2}&(\mu \rho + k \Lambda_0 + \wrho) \sT\Big(q, \lambda, \omega, w^{2l + 1}_0 \cdot (\mu \rho + k \Lambda_0 + \wrho)\Big)\\
&= q^{2(l + 1)(\mu + 2 - (l + 1)\wk)\omega - 4l - 2} \frac{(q^{-2\mu - 2} q^{(4l + 2) \wk}; q^{-2\wk})_{2l + 1}}{(q^{-2\mu - 6} q^{(4l + 2)\wk}; q^{-2\wk})_{2l + 1}} T(q, \lambda, \omega, - \mu - 2 + 2(l+1)\wk, \wk)\\
&= - \chi^0_{\mu, k}(q, \lambda, \omega) q^{2(l + 1)(\mu + 2 - (l + 1)\wk)(\omega + 2)} u(q, \lambda, \omega, \mu + 2 - (2l + 2)\wk, \wk)
\end{align*}
and
\begin{align*}
- A_{w^{2l + 1}_1, L_2}&(\mu \rho + k \Lambda_0 + \wrho) \sT\Big(q, \lambda, \omega, w^{2l + 1}_1 \cdot (\mu \rho + k \Lambda_0 + \wrho)\Big)\\
&= q^{-2l(\mu + 2 + l\wk)\omega - 4l - 2} \frac{(q^{2\mu + 6} q^{4l\wk}; q^{-2\wk})_{2l + 1}}{(q^{2\mu + 2} q^{4l\wk}; q^{-2\wk})_{2l + 1}} T(q, \lambda, \omega, - \mu - 2 - 2l\wk, \wk)\\
&= - \chi^0_{\mu, k}(q, \lambda, \omega) q^{-2l(\mu + 2 + l\wk)(\omega + 2)} u(q, \lambda, \omega, \mu + 2 + 2l\wk, \wk).
\end{align*}
Combining our previous computations, we find that 
\begin{align*}
&\frac{\chi_{\mu, k, 2}(q, \lambda, \omega)}{\chi^0_{\mu, k}(q, \lambda, \omega)} = u(q, \lambda, \omega, -\mu - 2, \wk)\\
&\phantom{=}+ \sum_{l = 1}^\infty \Big(q^{2l(\mu + 2 - l\wk)(\omega + 2)} u(q, \lambda, \omega, -\mu - 2 + 2l\wk, \wk) + q^{-2l(\mu + 2 + l \wk)(\omega + 2)} u(q, \lambda, \omega, -\mu - 2 - 2l\wk, \wk)\Big)\\
&\phantom{=}- \sum_{l = 0}^\infty \Big(q^{2(l + 1)(\mu + 2 - (l + 1)\wk)(\omega + 2)} u(q, \lambda, \omega, \mu + 2 - (2l + 2)\wk, \wk) + q^{-2l(\mu + 2 + l\wk)(\omega + 2)} u(q, \lambda, \omega, \mu + 2 + 2l\wk, \wk)\Big)\\
&=  \sum_{j \in \ZZ} q^{-2j(\omega + 2)(\mu + 2 + j\wk)}\Big(u(q, \lambda, \omega, -\mu - 2 - 2j\wk, \wk) - u(q, \lambda, \omega, \mu + 2 + 2j\wk, \wk)\Big).\qedhere
\end{align*}
\end{proof}

\begin{remark}
For Verma modules with dominant integral highest weight, the analogue of Proposition \ref{prop:waki-verma} may not hold, and the Verma module differs in general from the Wakimoto module.  In the classical limit, this difficulty is addressed in \cite{BF} by realizing the irreducible integrable module via a BRST-type two-sided resolution via a complex of Fock modules.  Our proof of Theorem \ref{thm:am-tr-value} avoids the use of this resolution, which to our knowledge is available in the literature for the quantum affine setting only in the bosonization of \cite{Kon2}.  Instead, we apply the affine BGG resolution to relate trace functions for irreducible integrable modules to trace functions for Verma modules with non-generic highest weight, which we then compute by analytic continuation from the case of generic highest weight.  Our approach should apply equally well in setting of the undeformed classical affine algebra.
\end{remark}

Combining Theorem \ref{thm:am-tr-value} with Proposition \ref{prop:ek-const} on $\chi_{0, 0, 2}(q, \lambda, \omega)$ from \cite{EK3}, we obtain Theorem \ref{thm:fv-conj} relating the elliptic and affine Macdonald polynomials.

\begin{prop}[{\cite[Theorem 11.1]{EK3}}] \label{prop:ek-const}
There is a function $f(q, q^{-2\omega})$ with unit constant term whose formal power series expansion in $q^{-2\omega}$ has rational function coefficients in $q$ such that 
\[
\chi_{0, 0, 2}(q, \lambda, \omega) = f(q, q^{-2\omega}) q^\lambda (q^{-2\lambda + 2}; q^{-2\omega})(q^{2\lambda + 2} q^{-2\omega}; q^{-2\omega}).
\]
\end{prop}

\begin{thm} \label{thm:fv-conj}
Let $\wk = k + 4$.  For $q^{-2\mu}$ and then $q^{-2\omega}$ sufficiently close to $0$ in the good region of parameters (\ref{eq:good-region}), the elliptic and affine Macdonald polynomials are related by
\begin{multline*}
J_{\mu, k, 2}(q, \lambda, \omega)\\
 = \frac{\wJ_{\mu, \wk}(q, \lambda, \omega)}{f(q, q^{-2\omega})} \frac{(q^{-4}; q^{-2\omega})(q^{-2\omega}; q^{-2\omega})^3}{(q^{-4}; q^{-2\wk})(q^{-2\wk}; q^{-2\wk})} \frac{(q^{-2\omega + 2}; q^{-2\omega}, q^{-2\wk})^2}{(q^{-2\omega - 2}; q^{-2\omega}, q^{-2\wk})^2} q^{\mu + 4}(q^{-2\mu - 6}; q^{-2\wk})(q^{2\mu + 2} q^{-2\wk}; q^{-2\wk}),
\end{multline*}
where $f(q, q^{-2\omega})$ is the normalizing function of Proposition \ref{prop:ek-const}.
\end{thm}
\begin{proof}
By definition, we have that
\begin{align*}
\frac{\chi^0_{\mu, k}(q, \lambda, \omega)}{\wJ^0_{\mu, \wk}(q, \lambda, \omega)}
&= - \frac{q^{\mu + \lambda + 4}(q^{-4}; q^{-2\omega})(q^{-2\omega}; q^{-2\omega})^3}{(q^{-2\wk}; q^{-2\wk})(q^{-4}; q^{-2\wk})} \frac{(q^{-2\omega + 2}; q^{-2\omega}, q^{-2\wk})^2}{(q^{-2\omega - 2}; q^{-2\omega}, q^{-2\wk})^2}\\
&\phantom{==========} (q^{-2\lambda + 2}; q^{-2\omega})(q^{2\lambda + 2} q^{-2\omega}; q^{-2\omega})(q^{-2\mu - 6}; q^{-2\wk})(q^{2\mu + 2} q^{-2\wk}; q^{-2\wk}).
\end{align*}
By Lemma \ref{lem:ell-mac-trans} and Theorem \ref{thm:am-tr-value} and the fact that $u(q, \lambda, \omega, \mu, k) = u(q, -\lambda, \omega, -\mu, k)$ from \cite[Lemma 2.2]{FV4}, we see that 
\[
\frac{\wJ_{\mu, \wk}(q, \lambda, \omega)}{J^0_{\mu, \wk}(q, \lambda, \omega)} = - \frac{\chi_{\mu, k, 2}(q, \lambda, \omega)}{\chi^0_{\mu, k}(q, \lambda, \omega)}.
\]
We conclude by Proposition \ref{prop:ek-const} that 
\begin{align*}
&J_{\mu, k, 2}(q, \lambda, \omega) = \frac{\chi_{\mu, k, 2}(q, \lambda, \omega)}{\chi_{0, 0, 2}(q, \lambda, \omega)}\\
&\phantom{=}= -\wJ_{\mu, \wk}(q, \lambda, \omega) \frac{\chi^0_{\mu, k}(q, \lambda, \omega)}{\wJ^0_{\mu, \wk}(q, \lambda, \omega)\chi_{0, 0, 2}(q, \lambda, \omega)}\\
&\phantom{=}= \wJ_{\mu, \wk}(q, \lambda, \omega) \frac{q^{\mu + 4}(q^{-4}; q^{-2\omega})(q^{-2\omega}; q^{-2\omega})^3}{f(q, q^{-2\omega})(q^{-2\wk}; q^{-2\wk})(q^{-4}; q^{-2\wk})} \frac{(q^{-2\omega + 2}; q^{-2\omega}, q^{-2\wk})^2}{(q^{-2\omega - 2}; q^{-2\omega}, q^{-2\wk})^2}(q^{-2\mu - 6}; q^{-2\wk})(q^{2\mu + 2} q^{-2\wk}; q^{-2\wk}). \qedhere
\end{align*}
\end{proof}

\appendix

\section{Elliptic functions} \label{sec:ell}

In this appendix, we give our conventions on theta functions, elliptic gamma functions, and phase functions and provide some estimates for these functions.  For a more detailed discussion of the elliptic gamma function, we refer the reader to \cite{FV3}.  We use multiplicative notation throughout the text.

\subsection{Theta functions} \label{sec:theta-def}

We often use the single and double $q$-Pochhammer symbols
\[
(u; q) = \prod_{n \geq 0} (1 - uq^n) \qquad \text{ and } \qquad (u; q, r) = \prod_{n, m \geq 0} (1 - u q^n r^m),
\]
which are convergent for $|q|, |r| < 1$.  We also use the terminating $q$-Pochhammer symbol
\[
(u; q)_m = \frac{(u; q)}{(u q^m; q)}.
\]
Define now the theta function
\[
\theta_0(u; q) = (u; q)(qu^{-1}; q).
\] 
The theta function satisfies the transformation properties
\[
\theta_0(qu; q) = - u^{-1} \theta_0(u; q) \qquad \text{ and } \qquad \theta_0(q^{-1}u; q) = - uq^{-1} \theta_0(u; q) \qquad \text{ and } \qquad \theta_0(u^{-1}; q) = - u^{-1} \theta_0(u; q).
\]
We also use Jacobi's first theta function $\theta(u; q)$, given by
\begin{equation} \label{eq:theta-def}
\theta(u; q) = i e^{\pi i (\tau/4 - z)} (u; q)(qu^{-1}; q) (q; q) = i e^{\pi i (\tau/4 - z)} (q; q) \theta_0(u; q)
\end{equation}
for $q = e^{2\pi i \tau}$ and $u = e^{2\pi i z}$.  We have the following asymptotic estimates for $|\theta_0(u; q)|$ which parallel those in \cite[Appendix C]{FV}.

\begin{lemma} \label{lem:theta-asymp}
If $|q| < 1$, we have the following estimates on $|\theta_0(u; q)|$.

\begin{itemize}
\item[(a)] There is a constant $C_1(q) > 0$ so that for any $u \neq 0$ we have
\[
|\theta_0(u; q)| \leq C_1(q) |u|^{1/2} \exp\left(-\frac{(\log|u|)^2}{2 \log|q|}\right).
\]

\item[(b)] For any $\eps > 0$, there is a constant $C_2(q, \eps) > 0$ so that for any $u \neq 0$ with $\min_n \Big|\log|uq^n|\Big| > \eps$, we have 
\[
|\theta_0(u; q)| \geq C_2(q, \eps)|u|^{1/2} \eps\left(- \frac{(\log|u|)^2}{2\log|q|}\right). 
\]
\end{itemize}
\end{lemma}
\begin{proof}
For (a), define the constant $C_1'(q)$ by 
\[
C_1'(q) := \max_{|q| \leq |u| \leq 1} |\theta_0(u; q)|.
\]
For any $u \neq 0 \in \CC$, choose 
\[
n = -\lfloor \log_{|q|}|u|\rfloor
\]
so that $|q| \leq |u q^n| < 1$. We then have
\[
\theta_0(uq^n; q) = (-1)^n u^{-n} q^{-\frac{n(n-1)}{2}} \theta_0(u; q).
\]
We then see that 
\begin{align*}
|\theta_0(u; q)| = |u|^n |q|^{\frac{n(n-1)}{2}} |\theta_0(uq^n; q)|\leq C_1'(q)\exp\left( n \log |u| + \frac{n(n - 1)}{2} \log|q|\right).
\end{align*}
Notice that 
\[
n \log|u| \leq \begin{cases} - \frac{(\log|u|)^2}{\log |q|} + \log|u| & |u| \geq 1 \\ - \frac{(\log|u|)^2}{\log|q|} & |u| < 1 \end{cases}
\]
and that 
\[
\frac{n(n - 1)}{2} \log|q| \leq \frac{(\log|u|)^2}{2 \log|q|} + \begin{cases} -\frac{1}{2} \log|u| + 1 & |u| \geq 1 \\ 
\frac{1}{2} \log |u| + 1 & |u| < 1 \end{cases}.
\]
Combining these, for $C_1(q) := e C_1'(q)$ we find that 
\[
|\theta_0(u; q)| \leq C_1(q) |u|^{1/2} \exp\left(- \frac{(\log|u|)^2}{2 \log|q|}\right).
\]

For (b), define the constant $C_2'(q, \eps)$ by 
\[
C_2'(q, \eps) := \min_{\substack{|q| \leq |u| \leq 1 \\ |\log|u|| \geq \eps, |\log|uq^{-1}|| \geq \eps}} |\theta_0(u; q)|.
\]
If $\min_n \Big|\log|uq^n|\Big| > \eps$, then we have that 
\[
|\theta_0(u; q)| = |u|^n |q|^{\frac{n(n-1)}{2}} |\theta_0(uq^n; q)|\geq C_2'(q, \eps)\exp\left( n \log |u| + \frac{n(n - 1)}{2} \log|q|\right).
\]
Notice now that 
\[
n \log|u| \geq \begin{cases} - \frac{(\log|u|)^2}{\log |q|} & |u| \geq 1 \\ - \frac{(\log|u|)^2}{\log|q|} + \log|u| & |u| < 1 \end{cases}
\]
and that 
\[
\frac{n(n - 1)}{2} \log|q| \geq \frac{(\log|u|)^2}{2 \log|q|} + \begin{cases} \frac{1}{2} \log|u| - 1 & |u| \geq 1 \\ 
-\frac{1}{2} \log |u| - 1 & |u| < 1 \end{cases}.
\]
This implies that for $C_2(q, \eps) := e^{-1} C_2'(q, \eps)$, we have the desired bound
\[
|\theta_0(u; q)| \geq C_2(q, \eps) |u|^{1/2} \eps\left(- \frac{(\log|u|)^2}{2\log|q|}\right). \qedhere
\]
\end{proof}

\begin{corr} \label{corr:theta-ratio}
If $|q| < 1$, for any $\eps > 0$ there are constants $D_1(q, \eps), D_2(q, \eps) > 0$ so that for $a, b$ and $z \neq 0$ satisfying
\[
\min_n \left|\log |zq^bq^n|\right| > \eps \text{ and } \min_n \left|\log|zq^aq^n|\right| > \eps,
\]
we have
\[
D_2(q, \eps) q^{\frac{a - b}{2}} |z^2 q^{a + b}|^{- \frac{a - b}{2}} \leq \left|\frac{\theta_0(zq^a; q)}{\theta_0(zq^b;q)}\right| \leq D_1(q, \eps) q^{\frac{a - b}{2}} |z^2q^{a + b}|^{- \frac{a - b}{2}}.
\]
\end{corr}
\begin{proof}
Apply Lemma \ref{lem:theta-asymp} and the factorization $(\log|zq^a|)^2 - (\log|zq^b|)^2 = \log|z^2q^{a + b}| \log|q^{a - b}|$.
\end{proof}

\subsection{Elliptic gamma function and phase function}

For $|r|, |p| < 1$, define the elliptic gamma function $\Gamma(z; r, p)$ by 
\[
\Gamma(z; r, p) = \frac{(z^{-1} r p; r, p)}{(z; r, p)}.
\]
Define the phase function $\Omega_{a}(t/z; r, p)$ by 
\begin{equation} \label{eq:phase-func-def}
\Omega_{a}(z; r, p) = \frac{(z a^{-1}; r, p)(z^{-1} a^{-1} rp; r, p)}{(z a; r, p)(z^{-1} a rp; r, p)} = \frac{\Gamma(za;r, p)}{\Gamma(za^{-1}; r, p)}.
\end{equation}
It has the following transformation properties.

\begin{lemma} \label{lem:phase-trans}
The phase function satisfies 
\begin{align*}
\Omega_{a}(z; r, p) &= \Omega_{a}(z^{-1}; r, p) \frac{\theta_0(z^{-1}a; p)\theta_0(za^{-1};r)}{\theta_0(z^{-1}a^{-1};p) \theta_0(za; r)}\\
\Omega_a(pz; r, p) &= \frac{\theta_0(za; r)}{\theta_0(za^{-1};r)}\Omega_a(z; r, p) \\
\Omega_a(p^{-1}z; r, p) &= \frac{\theta_0(za^{-1}p^{-1}; r)}{\theta_0(za p^{-1};r)} \Omega_a(z; r, p).
\end{align*}
\end{lemma}
\begin{proof}
Observe that 
\begin{align*} 
\Omega_{a}(z; r, p) &= \Omega_{a}(z^{-1}; r, p) \frac{(1 - z a^{-1})(za^{-1} r; r)(za^{-1}p, p)}{(1 - za) (zar;r)(zap;p)} \frac{(1-z^{-1}a)(z^{-1}ap;p)(z^{-1}ar;r)}{(1 - z^{-1}a^{-1})(z^{-1}a^{-1}p;p)(z^{-1}a^{-1}r;r)}\\
&= \Omega_{a}(z^{-1}; r, p) \frac{\theta_0(z^{-1}a; p)\theta_0(za^{-1};r)}{\theta_0(z^{-1}a^{-1};p) \theta_0(za; r)}.
\end{align*}
In addition, we have that 
\[
\Omega_a(pz; r, p) = \frac{(za; r)(z^{-1}a^{-1}r;r)}{(za^{-1}; r)(z^{-1}ar;r)} \Omega_a(z; r, p)= \frac{\theta_0(za; r)}{\theta_0(za^{-1};r)} \Omega_a(z; r, p) 
\]
and that 
\[
\Omega_a(p^{-1}z; r, p) = \frac{(za^{-1}p^{-1}; r) (z^{-1}ar p;r)}{(za p^{-1};r)(z^{-1} a^{-1} p r; r)} \Omega_a(z; r, p) = \frac{\theta_0(za^{-1}p^{-1}; r)}{\theta_0(za p^{-1};r)} \Omega_a(z; r, p). \qedhere
\]
\end{proof}

\section{Computations of vertex operators, OPE's, and one loop correlation functions} \label{sec:comps}

We give some explicit computations of OPE's and one loop correlation functions in the method of coherent states. For each OPE, we write $=$ to denote equality of analytic continuations outside of the specified domain.

\subsection{Table of vertex operators} \label{sec:vert-op}

We give the vertex operators used in the free field construction of \cite{Mat}. Define first $Y^{\pm}(z)$, $Z_{\pm}(z)$, $W_{\pm}(z)$, and $U(z)$ by
\begin{align*}
Y^{\pm}(z) &= \exp\left(\pm \sum_{m > 0} q^{\mp \frac{km}{2}} \frac{z^m}{[km]} (\alpha_{-m} + \bar{\alpha}_{-m})\right) e^{\pm 2 (\alpha + \bar{\alpha})} z^{\pm \frac{1}{k}(\alpha_0 + \bar{\alpha}_0)} \exp\left(\mp \sum_{m > 0} q^{\mp \frac{km}{2}} \frac{z^m}{[km]} (\alpha_m + \bar{\alpha}_m)\right)\\
Z_{\pm}(z) &= \exp\left(\mp(q - q^{-1}) \sum_{m > 0} z^{\mp m} \frac{[m]}{[2m]} \bar{\alpha}_{\pm m}\right) q^{\mp \bar{\alpha}_0}\\
W_{\pm}(z) &= \exp\left(\mp (q - q^{-1}) \sum_{m > 0} z^{\mp m} \frac{[m]}{[2m]} \beta_{\pm m}\right) q^{\mp \frac{1}{2}\beta_0}\\
U(z) &= \exp\left(-\sum_{m > 0} z^m \frac{q^{- \frac{k+2}{2}m}}{[(k+2)m]} \beta_{-m}\right) e^{-2 \beta} z^{-\frac{1}{k+2}\beta_0}\exp\left(\sum_{m > 0}z^{-m} \frac{q^{-\frac{k+2}{2}m}}{[(k+2)m]}\beta_m\right).
\end{align*}
In terms of these operators, we define
\[
X^+(z) := \frac{1}{(q - q^{-1})z} (X^+_+(z) - X^+_-(z))\qquad \text{ and } \qquad X^-(z) : = -\frac{1}{(q - q^{-1})z} (X^-_+(z) - X^-_-(z))
\]
for 
\begin{align*}
X^+_+(z) &:=\, : Y^+(z) Z_+(q^{-\frac{k +2}{2}} z) W_+(q^{-\frac{k}{2}}z): \qquad \text{ and } \qquad X^+_-(z) :=\, : Y^+(z) W_-(q^{\frac{k}{2}}z) Z_-(q^{\frac{k+2}{2}}z):\\
X^-_+(z) &:=\, : Y^-(z) Z_+(q^{\frac{k +2}{2}} z) W_+(q^{\frac{k}{2}}z)^{-1}: \qquad \text{ and } \qquad X^-_-(z) :=\, : Y^-(z) W_-(q^{-\frac{k}{2}}z)^{-1} Z_-(q^{-\frac{k+2}{2}}z):.
\end{align*}
Define also the screening operator 
\[
S(z) := - \frac{1}{(q - q^{-1})z} (S_+(z) - S_-(z))
\]
for 
\[
S_+(z) :=\, : U(z) Z_+(q^{-\frac{k+2}{2}}z)^{-1} W_+(q^{-\frac{k}{2}}z)^{-1}: \qquad \text{ and } \qquad S_-(z) :=\, : U(z) W_-(q^{\frac{k}{2}}z)^{-1}Z_-(q^{\frac{k+2}{2}}z)^{-1}:.
\]
Finally, define the operators
\begin{align*}
\eta(w_0) &= \exp\Big(\sum_{m > 0} \frac{w_0^m}{[2m]}(q^{\frac{k}{2}m} \beta_{-m} + q^{\frac{k+2}{2}m}\bar{\alpha}_{-m})\Big) e^{(k+2)\beta + k \bar{\alpha}} w_0^{\frac{1}{2}(\beta_0 + \bar{\alpha}_0)} \exp\Big(\!\!-\!\!\sum_{m > 0} \frac{w_0^{-m}}{[2m]}(q^{\frac{k}{2}m}\beta_m + q^{\frac{k+2}{2}m}\bar{\alpha}_m)\Big)\\
\xi(z_0) &=  \exp\Big(\!\!-\!\!\sum_{m > 0} \frac{z_0^m}{[2m]}(q^{\frac{k}{2}m} \beta_{-m} + q^{\frac{k+2}{2}m}\bar{\alpha}_{-m})\Big) e^{-(k+2)\beta - k \bar{\alpha}} z_0^{-\frac{1}{2}(\beta_0 + \bar{\alpha}_0)} \exp\Big(\sum_{m > 0} \frac{z_0^{-m}}{[2m]}(q^{\frac{k}{2}m}\beta_m + q^{\frac{k+2}{2}m}\bar{\alpha}_m)\Big).
\end{align*}
and the intertwining vertex operator
\begin{align} \nonumber
\phi_j(z) = \exp\Big(\sum_{m > 0}& \frac{(q^{k + 2} z)^m q^{\frac{km}{2}} [2jm]}{[km][2m]} (\alpha_{-m} + \ba_{-m})\Big) e^{2j(\alpha + \ba)}z^{\frac{j}{k+2}(\alpha_0 + \ba_0)}\exp\Big(\!\!-\!\!\sum_{m > 0} \frac{(q^{k+2}z)^{-m} q^{\frac{km}{2}}[2jm]}{[km][2m]} (\alpha_m + \ba_m)\Big) \\ 
&\phantom{=} \exp\Big(\sum_{m > 0} \frac{(q^{k+2} z)^m q^{\frac{k+2}{2} m } [2jm]}{[(k+2)m][2m]} \beta_{-m}\Big) e^{2j\beta} z^{\frac{j}{k+2} \beta_0} \exp\Big(-\sum_{m > 0} \frac{(q^{k+2}z)^{-m} q^{\frac{k+2}{2}m} [2jm]}{[2m][(k+2)m]}\beta_m\Big). \label{eq:phi-def}
\end{align}
We define the modes $\{a_n\}$ of each vertex operator $A(z)$ listed by 
\[
A(z) = \sum_{n \in \ZZ} a_n z^{-n - 1}
\]
with the exception of $\xi(z)$, whose modes are defined by $\xi(z) = \sum_{n \in \ZZ}\xi_n z^{-n}$.

\begin{remark}
Our definition of $\phi_j(z)$ corrects a typographical error in \cite[Equation 5.1]{Mat}.
\end{remark}

\subsection{Operator product expansions} \label{sec:ope-comp}

We give some OPE's which will be used in our computations.  In each case, we mention the domain on which the relevant OPE converges.  If no domain is specified, then the OPE converges for all values of the variables.

\subsubsection{OPE's between $S(t)$ and $X^-(w)$}

We record the OPE's between $S(t)$ and $X^-(w)$.  We have that 
\begin{align*}
S_+(t) X^-_+(w) &= \frac{qt - wq^{k+1}}{t - wq^{k+2}}:S_+(t) X^-_+(w): \,\, |wq^{k+2}| < |t|\\
X^-_+(w) S_+(t) &= \frac{qt - w q^{k+1}}{t - wq^{k+2}} :S_+(t) X^-_+(w):\,\, |wq^{k+2}| > |t|\\
S_+(t) X^-_-(w) &= q :S_+(t) X^-_-(w):\\
X^-_-(w) S_+(t) &= q :S_+(t) X^-_-(w):\\
S_-(t) X^-_+(w) &= q^{-1} :S_-(t) X^-_+(w):\\
X^-_+(w) S_-(t) &= q^{-1} :S_-(t) X^-_+(w):\\
S_-(t) X^-_-(w) &= \frac{tq^{k+1} - wq}{tq^{k+2} - w}:S_-(t) X^-_-(w):  \,\, |w| < |tq^{k+2}|\\
X^-_-(w) S_-(t) &= \frac{tq^{k+1} - wq}{tq^{k+2} - w}:S_-(t) X^-_-(w): \,\, |w| > |tq^{k+2}|.
\end{align*}
We will require the following computation for analysis of convergence of the Jackson integral.
\begin{lemma} \label{lem:x0comm}
We have that 
\begin{multline*}
[x_0^-, S(t)] = - \frac{1}{(q - q^{-1})t} \Big(:U(t) Y^-(t q^{-k-2}) W_+(t q^{-\frac{k}{2} - 2})^{-1} W_+(tq^{-\frac{k}{2}})^{-1}:\\ - :U(t) Y^-(t q^{k + 2}) W_-(t q^{\frac{k}{2} + 2})^{-1} W_-(tq^{\frac{k}{2}})^{-1}:\Big).
\end{multline*}
\end{lemma}
\begin{proof}
This follows from the OPE's between $X^-_\pm(w)$ and $S_\pm(t)$ computed above.
\end{proof}

\subsubsection{OPE's between $S(t)$ and $X^+(w)$}

We record the OPE's between $S(t)$ and $X^+(w)$.  We have that 
\begin{align*}
S_+(t) X^+_+(w) &= \frac{t - wq^2}{q(t - w)} :S_+(t) X^+_+(w): \,\, |w| < |t|\\
X^+_+(w) S_+(t) &= \frac{t - wq^2}{q(t - w)}:S_+(t) X^+_+(w): \,\, |w| > |t|\\
S_+(t) X^+_-(w) &= q^{-1} :S_+(t) X^+_-(w):\\
X^+_-(w) S_+(t) &= q^{-1} :S_+(t) X^+_-(w):\\
S_-(t) X^+_+(w) &= q :S_-(t) X^+_+(w):\\
X^+_+(w) S_-(t) &= q :S_-(t) X^+_+(w):\\
S_-(t) X^+_-(w) &= \frac{q^2t - w}{q(t - w)}:S_-(t) X^+_-(w): \,\, |w| < |t| \\
X^+_-(w) S_-(t) &= \frac{q^2t - w}{q(t - w)}:S_-(t) X^+_-(w): \,\, |w| > |t|.
\end{align*}

\subsubsection{OPE's between $S(t)$ and $\phi_1(z)$}

We compute OPE's between $S(t)$ and $\phi_1(z)$.  We obtain
\begin{align*}
S_+(t) \phi_1(z) &= t^{-\frac{2}{k + 2}} \frac{(z t^{-1} q^{-2}; q^{-2k-4})}{(zt^{-1} q^2; q^{-2k-4})} :S_+(t) \phi_1(z):\,\, |t| > |zq^2|, |zq^{-2}|\\
\phi_1(z) S_+(t) &= z^{-\frac{2}{k + 2}} \frac{(t z^{-1} q^{-2k-6}; q^{-2k-4})}{(tz^{-1} q^{-2k-2}; q^{-2k-4})} :S_+(t) \phi_1(z):\,\, |t| < |zq^{2k + 2}|, |zq^{2k+6}|\\
S_-(t) \phi_1(z) &= t^{-\frac{2}{k + 2}} \frac{(z t^{-1} q^{-2}; q^{-2k-4})}{(zt^{-1} q^2; q^{-2k-4})} :S_-(t) \phi_1(z):\,\, |t| > |zq^{-2}|, |zq^2|\\
\phi_1(z) S_-(t) &=  z^{-\frac{2}{k + 2}} \frac{(t z^{-1} q^{-2k-6}; q^{-2k-4})}{(tz^{-1} q^{-2k-2}; q^{-2k-4})} :S_-(t) \phi_1(z):\,\, |t| < |zq^{2k+6}|, |zq^{2k+2}|,
\end{align*}
which we summarize as
\begin{align*}
S_a(t) \phi_1(z) &= t^{-\frac{2}{k+2}} \frac{(zt^{-1} q^{-2}; q^{-2k-4})}{(zt^{-1}q^2; q^{-2k-4})} : S_a(t) \phi_1(z): \qquad |t| > |zq^2|, |zq^{-2}| \\
\phi_1(z) S_a(t) &= z^{-\frac{2}{k+2}} \frac{(tz^{-1} q^{-2k-6}; q^{-2k-4})}{(tz^{-1} q^{-2k-2}; q^{-2k-4})} : S_a(t) \phi_1(z): \qquad |t| < |zq^{2k+6}|, |zq^{2k+2}|.
\end{align*}

\subsubsection{OPE's between $\phi_j(z)$ and $X^-_b(w)$}

Computing OPE's, we obtain
\begin{align*}
\phi_1(z) X^-_+(w) &= :\phi_1(z) X^-_+(w):\\
X^-_+(w) \phi_1(z) &= \frac{q^2(w - zq^k)}{w - z q^{k+4}}:\phi_1(z) X^-_+(w): \,\,\,\,\, |w| > |z q^{k+4}|, |zq^k| \\
\phi_1(z) X^-_-(w) &= \frac{z q^{k+4} - w}{z q^{k+4} - q^4 w}:\phi_1(z) X^-_-(w):  \,\,\,\,\, |w| < |z q^{k+4}|, |zq^k| \\
X^-_-(w) \phi_1(z) &= q^{-2} :X^-_-(w) \phi_1(z):.
\end{align*}
We conclude that 
\begin{align*}
\phi_1(z) X^-_b(w) &= q^{2b - 2} \frac{zq^{k+4} - w}{z q^{k+ 2b + 2} - w} : \phi_1(z) X^-_b(w): \,\, |w| < |z q^{k+ 4}|, |zq^k|\\
X^-_b(w) \phi_1(z) &= q^{2b} \frac{w - z q^k}{w - zq^{k + 2 + 2b}}  : \phi_1(z) X^-_b(w): \,\, |w| > |z q^{k+ 4}|, |zq^k|
\end{align*}
and therefore that 
\begin{equation} \label{eq:phi-x-comm}
[\phi_1(z), X^-_b(w)]_{q^2} = \sum_{c \in \pm \{1\}} (-1)^{\frac{c - 1}{2}} q^{2b} \frac{w q^{-2c} - z q^{k + 2}}{w - zq^{k + 2 + 2b}} : \phi_1(z) X^-_b(w):,
\end{equation}
where in (\ref{eq:phi-x-comm}) the analytic continuation holds in the region $|w| < |zq^k|, |z q^{k+4}|$ for $c = 1$ and $|w| > |zq^k|, |zq^{k+4}|$ for $c = -1$. 

\subsubsection{OPE's between $X_a^\pm(z)$ and $X_b^\pm(w)$}

Computing OPE's, we obtain
\begin{align*}
X^+_+(z) X^+_+(w) &= q \frac{1 - \frac{w}{z}}{1 - q^2 \frac{w}{z}} :X^+_+(z) X^+_+(w): \qquad |z| > |w|, |wq^2|\\
X^+_+(z) X^+_-(w) &= q \frac{1 - q^{-2} \frac{w}{z}}{1 - q^2 \frac{w}{z}} :X^+_+(z) X^+_-(w): \qquad |z| > |wq^2|, |wq^{-2}|\\
X^+_+(z) X^-_+(w) &= q^{-1} \frac{1 - q^{k + 2} \frac{w}{z}}{1 - q^k\frac{w}{z}} :X^+_+(z)X^-_+(w): \qquad |z| > |wq^k|, |wq^{k + 2}|\\
X^+_+(z) X^-_-(w) &= q^{-1} :X^+_+(z) X^-_-(w):
\end{align*}
and
\begin{align*}
X^+_-(z) X^+_+(w) &= q^{-1} : X^+_-(z) X^+_+(w): \\
X^+_-(z) X^+_-(w) &= q^{-1} \frac{1 - \frac{w}{z}}{1 - q^2 \frac{w}{z}} : X^+_-(z) X^+_-(w): \qquad |z| > |w|, |wq^2| \\
X^+_-(z) X^-_+(w) &= q :X^+_-(z) X^-_+(w): \\
X^+_-(z) X^-_-(w) &= q \frac{1 - q^{k + 2} \frac{w}{z}}{1 - q^k \frac{w}{z}} : X^+_-(z) X^-_-(w): \qquad |z| > |wq^k|, |wq^{k+2}|
\end{align*}
and
\begin{align*}
X^-_+(z) X^+_+(w) &= q \frac{1 - q^{-k-2} \frac{w}{z}}{1 - q^{-k} \frac{w}{z}} : X^-_+(z)X^+_+(w): \qquad |z| > |wq^{-k}|, |wq^{-k-2}| \\
X^-_+(z) X^+_-(w) &= q :X^-_+(z) X^+_-(w):\\
X^-_+(z) X^-_+(w) &= q^{-1} \frac{1 - \frac{w}{z}}{1 - q^{-2} \frac{w}{z}} :X^-_+(z) X^-_+(w): \qquad |z| > |w|, |wq^{-2}|\\
X^-_+(z) X^-_-(w) &= q^{-1} \frac{1 - q^2 \frac{w}{z}}{1 - q^{-2} \frac{w}{z}} :X^-_+(z) X^-_-(w): \qquad |z| > |wq^2|, |wq^{-2}|
\end{align*}
and
\begin{align*}
X^-_-(z) X^+_+(w) &= q^{-1} :X^-_-(z) X^+_+(w):\\
X^-_-(z) X^+_-(w) &= q^{-1} \frac{1 - q^{k+2} \frac{w}{z}}{1 - q^k \frac{w}{z}}:X^-_-(z) X^+_-(w): \qquad |z| > |wq^{k+2}|, |wq^k|\\
X^-_-(z) X^-_+(w) &= q :X^-_-(z) X^-_+(w):\\
X^-_-(z) X^-_-(w) &= q \frac{1 - \frac{w}{z}}{1 - q^{-2} \frac{w}{z}} :X^-_-(z) X^-_-(w): \qquad |z| > |w|, |wq^{-2}|.
\end{align*}

\subsubsection{OPE's between $U(t)$, $W_\pm(t)$, and $\phi_1(z)$}

We record OPE's between the constituents $U(t)$ and $W_{\pm}(t)$ of $S(t)$ and $\phi_1(z)$.  We have that 
\begin{align*}
U(t) \phi_1(z) &= t^{-\frac{2}{k + 2}}\frac{(\frac{z}{t}q^{-2}; q^{-2k - 4})}{(\frac{z}{t} q^2; q^{-2k - 4})} :U(t) \phi_1(z):\qquad |t| > |z q^2|, |zq^{-2}|\\
W_+(t)\phi_1(z) &= q^{-1} \frac{1 - \frac{z}{t}q^{\frac{3}{2}k + 4}}{1 - \frac{z}{t} q^{\frac{3}{2}k + 2}} :W_+(t)\phi_1(z): \qquad |t| > |zq^{\frac{3}{2}k + 2}|, |zq^{\frac{3}{2}k + 4}|\\
W_-(t)\phi_1(z) &= q:W_-(t)\phi_1(z):.
\end{align*}

\subsection{Trace of the degree $0$ part} \label{sec:zero-mode-trace}

In this section we compute the action of the intertwiner in the degree $0$ part of Fock space.

\begin{prop} \label{prop:zero-mode-trace}
The trace in the degree $0$ part of Fock space is 
\[
\Tr|_{\FF_{\mu, s}^0}\Big(\eta(w_0)\xi(z_0) S_a(t) :\phi_1(z) X^-_b(w):q^{2\lambda\rho +2 \omega d}\Big) = q^{(2\lambda + b - a)s + (a + b)\mu + a} z^{\frac{2\mu}{\kappa}} t^{-\frac{2(\mu + 1)}{\kappa}} z_0^{-\mu + s} w_0^{\mu - s - 1}.
\]
\end{prop}
\begin{proof}
Because the degree $0$ part of Fock space has dimension $1$, the trace is given by the action of the degree zero part of the intertwiner on the highest weight vector, which is given by
\begin{multline*}
v_{\mu, s} \mapsto q^{2\lambda s} v_{\mu, s} \mapsto q^{2\lambda s} q^{bs} q^{b \mu} z^{\frac{2\mu}{\kappa}} v_{\mu + 1, s} \mapsto q^{2\lambda s} q^{bs} q^{b \mu} z^{\frac{2\mu}{\kappa}} t^{-\frac{2(\mu + 1)}{\kappa}} q^{a(\mu + 1)} q^{-a} v_{\mu, s} \\
\mapsto q^{2\lambda s} q^{bs} q^{b \mu} z^{\frac{2\mu}{\kappa}} t^{-\frac{2(\mu + 1)}{\kappa}} q^{a(\mu + 1)} q^{-as} z_0^{-\mu + s} v_{\mu - \frac{\kappa}{2}, s, s - \frac{k}{2}} \\
\mapsto q^{2\lambda s} q^{bs} q^{b \mu} z^{\frac{2\mu}{\kappa}} t^{-\frac{2(\mu + 1)}{\kappa}} q^{a(\mu + 1)} q^{-as} z_0^{-\mu + s} w_0^{\mu - s - 1} v_{\mu, s},
\end{multline*}
yielding the result after simplification.
\end{proof}

\subsection{One loop correlation functions} \label{sec:one-loop-comp}

In this section we compute the one loop correlation functions
\[
T_{PQ} := (q^{-2\omega}; q^{-2\omega})^3 \prod_{m \geq 1} \Tr|_{\FF_{\beta, \mu, m} \otimes \FF_{\alpha, s, m} \otimes \FF_{\bar{\alpha}_s, m}} \Big(P(z) Q(w) q^{2\omega d}\Big)
\]
for each pair of the vertex operators which appear in 
\[
\Tr|_{\FF_{\mu, s}^{>0}}\Big(\eta(w_0)\xi(z_0) S_a(t) :\phi_1(z) X^-_b(w):q^{2\lambda\rho +2 \omega d}\Big).
\]
We use a general computation of traces in Heisenberg algebras. Consider the algebra
\[
\AA = \langle \alpha_-, \alpha_+ \mid [\alpha_+, \alpha_-] = 1\rangle.
\]
Let $\FF$ be the Fock space for $\AA$ with highest weight vector $v_\FF$. 

\begin{lemma} \label{lem:h-trace}
If $e^z < 1$, the operator $\psi: \FF \to \FF$ defined by 
\[
\psi = \prod_{i = 1}^N \exp(x_i \alpha_-) \exp(y_i \alpha_+) \cdot \exp(z \alpha_- \alpha_+)
\]
has trace
\[
\Tr|_\FF(\psi) = \frac{1}{1 - e^z} \exp\Big(\sum_{i > j} \frac{x_i y_j}{1 - e^z} + \sum_{i \leq j} \frac{e^z x_i y_j}{1 - e^z}\Big).
\]
\end{lemma}
\begin{proof}
Let $X = \sum_i x_i$ and $Y = \sum_i y_i$.  Recalling that for operators $A$ and $B$ with $[A, B]$ central we have $e^A e^B = e^B e^A e^{[A, B]}$, we obtain
\[
\psi = \exp(Y \alpha_+) \exp(X \alpha_-)\exp(z \alpha_- \alpha_+) \exp\Big(-\sum_{i \leq j} x_i y_j\Big).
\]
Applying \cite[Equation B.9]{KT} to compute $\Tr|_\FF(\psi)$, we obtain
\begin{align*}
\Tr|_\FF(\psi) &= \frac{1}{\pi}\int_{\CC} d^2\lambda e^{-|\lambda|^2} \sum_{n, m \geq 0} \left\langle \frac{(\lambda + \bar{X})^n}{n!} \alpha_+^n v_\FF^*,\exp\Big(-\sum_{i\leq j} x_i y_j\Big) \frac{(e^z \lambda + Y)^m}{m!} \alpha_-^m v_\FF\right\rangle\\
&= \frac{1}{\pi}e^{-\sum_{i\leq j} x_i y_j} e^{XY} \int_\RR \exp\Big((e^z - 1) a^2 + (Y + X e^z) a\Big) da \cdot \int_\RR \exp\Big((e^z - 1)b^2 + i(X e^z - Y)b\Big) db \\
&= \frac{1}{1 - e^z} \exp\Big(\frac{XY}{1 - e^z} - \sum_{i \leq j} x_i y_j\Big)\\
&= \frac{1}{1 - e^z} \exp\Big(\sum_{i > j} \frac{x_i y_j}{1 - e^z} + \sum_{i \leq j} \frac{ e^z x_i y_j}{1 - e^z}\Big)
\end{align*}
using the fact that $\int e^{-ax^2 + bx} dx = \frac{\sqrt{\pi}}{\sqrt{a}} e^{\frac{b^2}{4a}}$ for $a > 0$.
\end{proof}

\begin{corr} \label{corr:scale-trace}
If $e^{zc} < 1$, if $[\beta_+, \beta_-] = c$, then 
\[
\Tr|_\FF\Big(\prod_i \exp(x_i \beta_-) \exp(y_i\beta_+) \cdot \exp(z \beta_- \beta_+)\Big) = \frac{1}{1 - e^{zc}} \exp\Big(\sum_{i > j} \frac{cx_i y_j}{1 - e^{zc}} + \sum_{i \leq j}\frac{c x_i y_j e^{zc}}{1 - e^{zc}}\Big).
\]
\end{corr}
\begin{proof}
The conclusion follows by taking $\alpha_\pm = \frac{\beta_\pm}{\sqrt{c}}$ in Lemma \ref{lem:h-trace}.
\end{proof}

We may now compute each $T_{PQ}$ by applying Corollary \ref{corr:scale-trace} on the Fock spaces for the Heisenberg algebra generated by $\star_{-m}$ and $\star_m$ for $m > 0$ and $\star \in \{\alpha, \ba, \beta\}$ and multiplying the results.  We state the results and the domains on which the relevant one loop correlation functions converge.  These are also recorded in Table \ref{tab:two-point}.  We first record $T_{PP}$ for each choice of $P$.  We have
\begin{align*}
T_{\eta \eta} &= T_{\xi\xi} = (q^{-2\omega}; q^{-2\omega})\\
T_{S_aS_a} &= \frac{(q^{-2\omega + 2} q^{-2\kappa}; q^{-2\omega}, q^{-2\kappa}) (q^{-2\omega}; q^{-2\omega})}{(q^{-2\omega - 2} q^{-2\kappa}; q^{-2\omega}, q^{-2\kappa})(q^{-2\omega-2}; q^{-2\omega})}\\
T_{\phi \phi} & = \frac{(q^{-2 \omega + 2}; q^{-2\kappa}, q^{-2 \omega})}{(q^{-2 \omega - 2}; q^{-2\kappa}, q^{-2 \omega})}\\
T_{X^-_b X^-_b} &= \frac{(q^{-2\omega}; q^{-2\omega})}{(q^{-2\omega - 2}; q^{-2\omega})}.
\end{align*}
We now record $T_{\eta Q}$ for $Q$ different from $\eta$.  We have
\begin{align*}
T_{\eta \xi} &= \theta_0\Big(\frac{z_0}{w_0}; q^{-2\omega}\Big)^{-1} \qquad |w_0 q^{-2\omega}| < |z_0| < |w_0|\\
T_{\eta S_a} &= \theta_0\Big(\frac{t}{w_0} q^{-a}; q^{-2 \omega}\Big)^{-1} \qquad |w_0 q^{-2\omega}| < |tq^{-a}| < |w_0|\\
T_{\eta \phi}&= 1\\
T_{\eta X^-_b} &= \theta_0\Big(\frac{w}{w_0} q^{b(k + 1)}; q^{-2 \omega}\Big) \qquad |w_0 q^{-2\omega}| < |w q^{(k + 1)b}| < |w_0|.
\end{align*}
For any $P \neq \eta$, we have that $T_{\xi P}(z_0) = T_{\eta P}^{-1}(z_0)$, where we mean that $w_0$ is replaced by $z_0$ and the region of convergence has the same constraints.  The other one loop correlation functions are given by
\begin{align*}
T_{S_a \phi}&= \Omega_{q^2}(\frac{t}{z}; q^{-2\omega}, q^{-2\kappa}) \frac{\theta_0(\frac{t}{z} q^{2}; q^{-2\kappa})\theta_0(\frac{z}{t}q^{-2};q^{-2\omega})}{\theta_0(\frac{t}{z}q^{-2};q^{-2\kappa}) \theta_0(\frac{z}{t}q^2; q^{-2\omega})} \qquad |tq^{-2\omega}q^{-2\kappa}q^2|, |tq^{-2\omega}q^{-2\kappa}q^{-2}| < |z| < |tq^2|, |tq^{-2}|\\
T_{S_a X^-_a}&= \begin{cases} \frac{\theta_0(\frac{w}{t}q^k; q^{-2\omega})}{\theta_0(\frac{w}{t}q^{k+2}; q^{-2\omega})} \qquad |tq^{-k} q^{-2\omega}|, |tq^{-k-2} q^{-2\omega}| < |w| < |tq^{-k}|, |tq^{-k-2}| & a = 1 \\ \frac{\theta_0(\frac{w}{t}q^{-k}; q^{-2\omega})}{\theta_0(\frac{w}{t} q^{-k-2}; q^{-2\omega})} \qquad |tq^k q^{-2\omega}|, |tq^{k+2}q^{-2\omega}| < |w| < |tq^k|, |tq^{k+2}| & a = -1 \end{cases}\\
T_{S_a X^-_{-a}} &= 1\\
T_{:\phi X^-_b:} &= \begin{cases} \frac{(\frac{z}{w} q^k q^{-2\omega}; q^{-2\omega})}{(\frac{z}{w} q^{k+4} q^{-2\omega}; q^{-2\omega})}\qquad |zq^k q^{-2\omega}|, |zq^{k+4}q^{-2\omega}| < |w| & b = 1 \\ \frac{(\frac{w}{z}q^{-k-4}q^{-2\omega}; q^{-2\omega})}{(\frac{w}{z} q^{-k}q^{-2\omega}; q^{-2\omega})}\qquad |w| < |zq^k q^{2\omega}|, |zq^{k+4} q^{2\omega}| & b = -1 \end{cases}.
\end{align*}

\bibliographystyle{alpha}
\bibliography{qafftr-bib}

\def\cprime{$'$}
\begin{thebibliography}{Kon94b}

\bibitem[AAR99]{AAR}
G.~Andrews, R.~Askey, and R.~Roy.
\newblock {\em Special functions}, volume~71 of {\em Encyclopedia of
  Mathematics and its Applications}.
\newblock Cambridge University Press, Cambridge, 1999.

\bibitem[BF90]{BF}
D.~Bernard and G.~Felder.
\newblock Fock representations and {BRST} cohomology in {${\rm SL}(2)$} current
  algebra.
\newblock {\em Comm. Math. Phys.}, 127(1):145--168, 1990.

\bibitem[EFK98]{EFK}
P.~Etingof, I.~Frenkel, and A.~Kirillov, Jr.
\newblock {\em Lectures on representation theory and {K}nizhnik-{Z}amolodchikov
  equations}, volume~58 of {\em Mathematical Surveys and Monographs}.
\newblock American Mathematical Society, Providence, RI, 1998.

\bibitem[EK94]{EK2}
P.~Etingof and A.~Kirillov, Jr.
\newblock Representations of affine {L}ie algebras, parabolic differential
  equations, and {L}am\'e functions.
\newblock {\em Duke Math. J.}, 74(3):585--614, 1994.

\bibitem[EK95]{EK3}
P.~Etingof and A.~Kirillov, Jr.
\newblock On the affine analogue of {J}ack and {M}acdonald polynomials.
\newblock {\em Duke Math. J.}, 78(2):229--256, 1995.

\bibitem[ES95]{ESt2}
Pavel Etingof and Konstantin Styrkas.
\newblock Algebraic integrability of {S}chr\"odinger operators and
  representations of {L}ie algebras.
\newblock {\em Compositio Math.}, 98(1):91--112, 1995.

\bibitem[ES98]{ESt}
P.~Etingof and K.~Styrkas.
\newblock Algebraic integrability of {M}acdonald operators and representations
  of quantum groups.
\newblock {\em Compositio Math.}, 114(2):125--152, 1998.

\bibitem[ES01]{ES2}
P.~Etingof and O.~Schiffmann.
\newblock Lectures on the dynamical {Y}ang-{B}axter equations.
\newblock In {\em Quantum groups and {L}ie theory ({D}urham, 1999)}, volume 290
  of {\em London Math. Soc. Lecture Note Ser.}, pages 89--129. Cambridge Univ.
  Press, Cambridge, 2001.

\bibitem[ESV02]{ESV}
P.~Etingof, O.~Schiffmann, and A.~Varchenko.
\newblock Traces of intertwiners for quantum groups and difference equations,
  {II}.
\newblock {\em Lett. Math. Phys.}, 62(2):143--158, 2002.

\bibitem[Eti95]{Eti3}
P.~Etingof.
\newblock Quantum integrable systems and representations of {L}ie algebras.
\newblock {\em J. Math. Phys.}, 36(6):2636--2651, 1995.

\bibitem[EV00]{EV}
P.~Etingof and A.~Varchenko.
\newblock Traces of intertwiners for quantum groups and difference equations,
  {I}.
\newblock {\em Duke Math. J.}, 104(3):391--432, 2000.

\bibitem[EV02]{EV2}
P.~Etingof and A.~Varchenko.
\newblock Dynamical {W}eyl groups and applications.
\newblock {\em Adv. Math.}, 167(1):74--127, 2002.

\bibitem[EV05]{EV3}
P.~Etingof and A.~Varchenko.
\newblock Orthogonality and the $q$-{KZB}-heat equation for traces of
  {$U_q(\mathfrak{g})$}-intertwiners.
\newblock {\em Duke Math. J.}, 128(1):83--117, 2005.

\bibitem[FSV03]{FSV1}
G.~Felder, L.~Stevens, and A.~Varchenko.
\newblock Elliptic {S}elberg integrals and conformal blocks.
\newblock {\em Math. Res. Lett.}, 10(5-6):671--684, 2003.

\bibitem[FTV97]{FTV}
G.~Felder, V.~Tarasov, and A.~Varchenko.
\newblock Solutions of the elliptic $q$-{KZB} equations and {B}ethe ansatz.
  {I}.
\newblock In {\em Topics in singularity theory}, volume 180 of {\em Amer. Math.
  Soc. Transl. Ser. 2}, pages 45--75. Amer. Math. Soc., Providence, RI, 1997.

\bibitem[FTV99]{FTV2}
G.~Felder, V.~Tarasov, and A.~Varchenko.
\newblock Monodromy of solutions of the elliptic quantum
  {K}nizhnik-{Z}amolodchikov-{B}ernard difference equations.
\newblock {\em Internat. J. Math.}, 10(8):943--975, 1999.

\bibitem[FV00]{FV3}
G.~Felder and A.~Varchenko.
\newblock The elliptic gamma function and {${\rm
  SL}(3,\mathbb{Z})\ltimes{\mathbb{Z}^3}$}.
\newblock {\em Adv. Math.}, 156(1):44--76, 2000.

\bibitem[FV01]{FV2}
G.~Felder and A.~Varchenko.
\newblock The {$q$}-deformed {K}nizhnik-{Z}amolodchikov-{B}ernard heat
  equation.
\newblock {\em Comm. Math. Phys.}, 221(3):549--571, 2001.

\bibitem[FV02]{FV}
G.~Felder and A.~Varchenko.
\newblock {$q$}-deformed {KZB} heat equation: completeness, modular properties
  and {${\rm SL}(3,\Bbb Z)$}.
\newblock {\em Adv. Math.}, 171(2):228--275, 2002.

\bibitem[FV04]{FV4}
G.~Felder and A.~Varchenko.
\newblock Hypergeometric theta functions and elliptic {M}acdonald polynomials.
\newblock {\em Int. Math. Res. Not.}, (21):1037--1055, 2004.

\bibitem[HK07]{HK}
I.~Heckenberger and S.~Kolb.
\newblock On the {B}ernstein-{G}elfand-{G}elfand resolution for {K}ac-{M}oody
  algebras and quantized enveloping algebras.
\newblock {\em Transform. Groups}, 12(4):647--655, 2007.

\bibitem[Jos95]{Jos1995}
A.~Joseph.
\newblock {\em Quantum groups and their primitive ideals}, volume~29 of {\em
  Ergebnisse der Mathematik und ihrer Grenzgebiete (3) [Results in Mathematics
  and Related Areas (3)]}.
\newblock Springer-Verlag, Berlin, 1995.

\bibitem[KK79]{KK}
V.~Kac and D.~Kazhdan.
\newblock Structure of representations with highest weight of
  infinite-dimensional {L}ie algebras.
\newblock {\em Adv. in Math.}, 34(1):97--108, 1979.

\bibitem[Kon94a]{Kon2}
H.~Konno.
\newblock B{RST} cohomology in quantum affine algebra
  {$U_q(\widehat{\mathfrak{sl}}_2)$}.
\newblock {\em Modern Phys. Lett. A}, 9(14):1253--1265, 1994.

\bibitem[Kon94b]{Kon}
H.~Konno.
\newblock Free-field representation of the quantum affine algebra
  {$U_q(\widehat{\mathfrak{sl}}_2)$} and form factors in the higher-spin
  {$XXZ$} model.
\newblock {\em Nuclear Phys. B}, 432(3):457--486, 1994.

\bibitem[KT99]{KT}
G.~Kuroki and T.~Takebe.
\newblock Bosonization and integral representation of solutions of the
  {K}nizhnik-{Z}amolodchikov-{B}ernard equations.
\newblock {\em Comm. Math. Phys.}, 204(3):587--618, 1999.

\bibitem[Mat94]{Mat}
A.~Matsuo.
\newblock A {$q$}-deformation of {W}akimoto modules, primary fields and
  screening operators.
\newblock {\em Comm. Math. Phys.}, 160(1):33--48, 1994.

\bibitem[Neg09]{Neg1}
A.~Negut.
\newblock Laumon spaces and the {C}alogero-{S}utherland integrable system.
\newblock {\em Invent. Math.}, 178(2):299--331, 2009.

\bibitem[Neg11]{Neg2}
A.~Negut.
\newblock Affine {L}aumon spaces and the {C}alogero-{M}oser integrable system.
\newblock Preprint, 2011.
\newblock \url{http://arxiv.org/abs/1112.1756}.

\bibitem[Rai10]{Rai}
E.~Rains.
\newblock Transformations of elliptic hypergeometric integrals.
\newblock {\em Ann. of Math. (2)}, 171(1):169--243, 2010.

\bibitem[SV07]{SV}
K.~Styrkas and A.~Varchenko.
\newblock Resonance relations, holomorphic trace functions and hypergeometric
  solutions to q{KZB} and {M}acdonald-{R}uijsenaars equations.
\newblock {\em Int. Math. Res. Pap. IMRP}, (4):Art. ID rpm008, 78 pp. (2008),
  2007.

\bibitem[Tsy10]{Tsy}
A.~Tsymbaliuk.
\newblock Quantum affine {G}elfand-{T}setlin bases and quantum toroidal algebra
  via {$K$}-theory of affine {L}aumon spaces.
\newblock {\em Selecta Math. (N.S.)}, 16(2):173--200, 2010.

\bibitem[Wak86]{Wak}
M.~Wakimoto.
\newblock Fock representations of the affine {L}ie algebra {$A^{(1)}_1$}.
\newblock {\em Comm. Math. Phys.}, 104(4):605--609, 1986.

\end{thebibliography}
\end{document}